\documentclass[12pt]{article}

\usepackage[margin=0.6in,footskip=0.25in]{geometry}
\usepackage{graphicx}
\usepackage{amsthm, amsmath, amssymb, enumerate, stmaryrd, yhmath}
\usepackage{hyperref}
\usepackage{tikz, ctable}
\usepackage{algorithm}
\usepackage[noend]{algpseudocode}
\usepackage[all,cmtip]{xy}
\usepackage{authblk}

\newtheorem{theorem}{Theorem}[section]

\newtheorem{proposition}[theorem]{Proposition}
\newtheorem{lemma}[theorem]{Lemma}
\newtheorem{corollary}[theorem]{Corollary}
\newtheorem{definition}{Definition}
\newtheorem{example}{Example}[section]
\newtheorem{remark}{Remark}[section]

\algblock{Input}{EndInput}
\algnotext{EndInput}
\algblock{Output}{EndOutput}
\algnotext{EndOutput}
\newcommand{\Desc}[2]{\State \makebox[2em][l]{#1}#2}

\title{Computing singular simplicial homologies of digraphs and quivers}
\author[1]{Matthew Burfitt}
\author[1]{Jie Wu}
\author[1,2]{Stephen S.-T. Yau}
\author[3]{Shing-Tung Yau}
\affil[1]{Beijing Key Laboratory of Topological Statistics and Applications for Complex Systems, Beijing Institute of Mathematical Sciences and Applications (BIMSA)}
\affil[2]{Department of Mathematical Sciences, Tsinghua University}
\affil[3]{Yau Mathematical Sciences Center, Tsinghua University}
\date{\vspace{-5ex}}

\begin{document}

\maketitle

\begin{abstract}
    The dynamics of large complex systems are predominately modeled through pairwise interactions, the principle underlying structure being a network of the form of a digraph or quiver. Significant success has been obtained in applying the homology of the directed flag complex to study digraphs arising as networks within numerous scientific disciplines. This homology of directed cliques enjoys relative ease of computation when compared to other digraph homologies, making it preferable for use in applications concerning large networks. By extending the ideas of singular simplicial homology to quivers in categories of different morphism types, several new singular simplicial homology theories have recently been constructed.

    Computationally efficient homologies for quivers have in general not previously been seriously considered. In this paper we develop further the homotopy theory of quivers necessary to derive functors that realise isomorphisms between the singular simplicial quiver homologies and the homologies of certain spaces. The simplicial chains of these spaces arise in a conveniently compact form that is at least as convenient as the directed flag complex for computations.
    
    Moreover, our constructions are natural with respect to the isomorphisms on homology making them suitable for applications in conjunction with persistent homology for practical use. In particular, for each of the singular simplicial homologies considered, we provide efficient algorithms for the computation of their persistent homology.
\end{abstract}

\tableofcontents

\section{Introduction}

The investigation of complex systems is the study of a large number of distinct entries influencing each over though pairwise interactions from which emergent behavior results in global dynamics. Complex systems are ubiquitous, describing a vast verity of natural and manmade phenomenon, with greater prevalence resulting from the growing availability of large networked data sets.  As a consequence of their importance, mathematical modeling of complex networks is extremely diverse \cite{Boccaletti2006}. However, network models are based universally on a collection of nodes and pairwise relations between them. Formally, a network has the mathematical structure of a graph, with greater flexibility provided by digraphs where pairwise relations are oriented between nodes. Notably, directed networks demonstrate phenomenon that cannot be captured by symmetric relations alone \cite{Asllani2014}. Significant examples of directed complex networks include food webs \cite{Brose2019}, spreading of disease \cite{Keeling2005}, or neural connectivity both biological \cite{Goodfellow2022} and artificial \cite{Seungwoong2021}.

Yet more structure can be encodes using quivers, allowing for multiple interactions of the same direction in the form of multiple edges and self interactions of a single node in the form of loops. In particular, both of these interaction types occur in circumstances such as gene regularly networks \cite{Macneil2011}, or representations of gauge theories \cite{Douglas1996}.

The development of methods for the analysis of digraphs and quivers therefore forms a foundation for more sophisticated models including higher order interactions \cite{Feng2024}, and networked dynamical systems \cite{Vespignani2012}. Nevertheless, the graphical structure itself is in may circumstances sufficient \cite{Ganmor2011, Merchan2016}. Important large scale network properties often captured though local sub-networks combining to provide global structure \cite{Ugander2013, Karuza2016, Li2018, Hu2024}.

Classical algebraic topology provides invariants of topological spaces, including singular homology, that extract global structure up to continuous deformations know as homotopies \cite{Munkers1984, Hatcher2002}. For computational purposes it is convenient to describe spaces as simplicial complexes, or more generally $\Delta$-sets. Given a weighted undirected network, a filtered sequences of Vietoris-Rips complexes can be obtained and persistent homology applied to capture multiscale network features. Using the graph metric induced by the weightings, persistent homology in demission zero is equivalent to single linkage hierarchical clustering \cite{Carlsson2010}, describing network communities. In higher dimensions persistent homology features can be interpreted as measuring the scale of cycles and voids contained within the network, naturally formalising a description of higher-order structures. The fact that classical community structure alone is in general insufficient to capture all system properties \cite{KartunGiles2019, Ritchie2017, Iacopini2019} makes the detection of higher order structures particularly important. More generally, persistent homology is a widely applied method in the field of topological data analysis \cite{Pun2022}, can be effectively computed \cite{Bauer2021}, and satisfies desirable stability proprieties \cite{Chazal2008}. 

When considering a directed graph, one obtains a space known as the directed flag complex by associating to each directed clique of vertices a simplex, as first described by Masulli and Villa \cite{Masulli2017}. Persistents homology can then be applied to extract global geometric structure from a weighted directed network though the topology of its associated filtered directed flag complex. More generally, once a filtered simplicial complex is obtained from a weighted directed network, other specialized methods can be developed for its analysis without the use of persistent homology \cite{Riihimaki2023, Jones2025}. Applications of the directed flag complex include the interpretation of brain activity \cite{Reimann2017}, dynamics of artificial neural networks \cite{Masulli2016}, protein-ligand binding affinity  \cite{Zia2025}, and the study of various interaction networks \cite{Kannan2019}. To enable the incorporation of persistent homology, in this work we primally consider weighted directed networks, allowing the further capacity to study networks with a real value associated to each interaction.

While the directed flag complex is valued for the simplicity of its construction, in the case of directed networks it is not always immediately obvious to decide which directed sub-networks should be considered topologically trivial within the context of a given application. Therefore, it is harder to determine an appropriate filtered space to associate to the networks underlying digraph or quiver. Furthermore, major drawbacks of the the directed flag complex lie in the fact that it does not satisfy good homotopy \cite[\S 6]{Chaplin2024} or stability properties \cite[\S 5]{Chaplin2024}\cite[\S 5.2]{Turner2019}.

An alternative well developed approach to the homology of directed graphs is provided by the path homology of Grigor’yan, Lin, Muranov and Yau \cite{Grigoryan2013}, formulated around the idea of detecting cycles in equal length paths within a digraph. In contrast to the directed flag complex, path homology has been shown to satisfy variations of the most important properties of classical homology of spaces \cite{Grigoryan2014, Grigoryan2017, Grigoryan2018}, and has stable persistent homology when applied to weighted networks \cite[\S 5]{Chowdhury2018}. However, there is no agreed extensions of path homology to quivers \cite{Grigoryan2018b}\cite[\S 12]{Ivanov2024}, with any extension satisfying the desired homotopy invariance being trivial (see Remark~\ref{rmk:NonStrongRedcuedquiverCase}). Moreover, effective computation of path homology faces major challenges \cite[\S 1.7 and \S 2.3]{Grigoryan2022}. The computational problems originate primarily from difficulties in explicitly describing a basis of the path homology chain complex. Consequently, efficient computation of path homology is presently limited to dimensions zero and one \cite{Dey2022}.

However, computation does not pose as significant a problem for singular homologies such as the directed flag complex, with developed highly parallelizable algorithms publicly available \cite{Lutgehetmann2020} making it possible to analyse large scale complex systems. For the purposes of applications, it would therefore be highly desirable to develop algorithms for singular homology theories with good theoretical proprieties that are also applicable to quivers in general, and that retain the computational advantages of the directed flag complex. 

The digraph consisting of a single clique on $n+1$ vertices $0,\dots,n$ with directed edges from any vertex to one with a larger value, may be considered a discrete directed analog of the usual geometric $n$-simplex. Given a category whose objects are quivers, a singular simplex is a morphism from the directed $n$-simplex to a given quiver. The singular homology of the quiver is then the homology of the chain complex freely generated by singular mappings in each dimension. In a recent paper Li, Muranov, Wu and Yau \cite{Li2024} gave three natural morphisms types: quiver mappings for which edges can be sent to edges or vertices, quiver homomorphisms which only send edges to edges, and the usual quiver inclusions. Using these three categories, three associated singular simplicial homologies $H_*^{\Delta,m}$, $H_*^{\Delta,h}$, and $H_*^{\Delta,i}$ of quiver were introduced with associated functorial and homotopy proprieties provided. Notably, persistent $H_*^{\Delta,m}$ coincides with the homology of the ordered tuple complex of a digraph introduced by Turner, who proved a stability theorem \cite[Theorem 21]{Turner2019} in this case. In particular, strong homotopy invariance \cite[Theorem 5.8]{Li2024} and stability provide $H_*^{\Delta,m}$ with a significant advantage over the homology of the directed flag complex when applied to complex networks.

In this work we further develop the homotopy theoretic properties of the singular simplicial homologies $H_*^{\Delta,m}$, $H_*^{\Delta,h}$, and $H_*^{\Delta,i}$, providing efficient algorithms necessary to apply them to complex systems in practice. This is achieved by associating to each quiver $G$ and singular homology a spaces such that its homology coincides with $H_*^{\Delta,m}(G)$, $H_*^{\Delta,h}(G)$, or $H_*^{\Delta,i}(G)$. Significantly, for a finite quiver the chain complexes of these spaces typically consist of a considerably smaller number of generators. Other approaches to singular homologies of digraphs considered in the context of applied topology are Dowker complexes \cite{Chowdhury2018b}, the preorder homology of strongly connected components \cite[\S 6]{Turner2019}, complexes of tournaments \cite{Govc2021} and the homology of closure spaces \cite{Bubenik2024}.

We begin in Section~\ref{sec:Background} by introducing the necessary algebraic topology background and results from \cite{Li2024}. Furthermore, we provide in Section~\ref{sec:DirectedFlagComplex} a generalised directed flag functor to the category of $\Delta$-sets whose homology naturally coincides with $H_*^{\Delta,i}(G)$ for an arbitrary quiver $G$. Under certain restrictions on $G$, the homologies $H_*^{\Delta,m}$, $H_*^{\Delta,h}$, and $H_*^{\Delta,i}$ are the same, and these relationships are fully described in Section~\ref{sec:HomologyRelations}. Explicitly, $H_*^{\Delta,h}$ and $H_*^{\Delta,i}$ agree when $G$ has no loops, and all three homologies coincide when considering digraphs without double edges. Moreover, these results are seen to be strict through the construction of counter examples in the absence of the necessary conditions.

The main constructions presented in this work begin in Section~\ref{sec:ComputingMappingHomology}, where we detail the necessary theoretical machinery to achieve efficient computation of $H_*^{\Delta,m}$. We initially describe in Section~\ref{sec:ReducedDigraph} the reduced digraph $\bar{\mathcal{R}}(G)$ of a quiver $G$, which we show is a strong homotopy deformation retraction of $G$. As a consequence the strong homotopy invariance of $H_*^{\Delta,m}$ now implies that $H_*^{\Delta,m}(G)$ is not dependent on the quiver structure of $G$, being invariant under operations collapsing loops and multiple edges. This simplification leads in Section~\ref{sec:RducedDirectedFlag} to a proof that $H_*^{\Delta,m}(G)$ is isomorphic to the homology of a construction we call the reduced directed flag complex $\bar{\mathcal{F}}(G)$. Moreover, $\bar{\mathcal{F}}$ is a functor acting naturally with respect to this isomorphism, providing a suitable alternative framework for computation of persistent $H_*^{\Delta,m}$. In fact computing the homology of $\bar{\mathcal{\mathcal{F}}}(G)$ is always simpler than computing $H_*^{\Delta,m}(G)$ on its own, with $H_*^{\Delta,m}$ chains in general being prohibitively large to compute the homology directly. This advantage is made explicit when we consider the computational complexity of the corresponding algorithms in Appendix~\ref{sec:ComplexityDirectedFlag}. 

Building on the ideas of Section~\ref{sec:ComputingMappingHomology}, we develop analogous computational results for $H_*^{\Delta,h}$ in Section~\ref{sec:ComputingHomomorphismHomology}. However, in this case the situation is somewhat more complex. We make use of a new notion of local strong $h$-homotopy initially set-out in Section~\ref{sec:HomotopyTheory}, which $H_*^{\Delta,h}$ is shown to be invariant under. The local strong $h$-homotopies are more flexible than the strong $h$-homotopies provided as an invariant of $H_*^{\Delta,h}$ in \cite{Li2024}. Analogously to the reduced digraph $\bar{\mathcal{R}}$ from Section~\ref{sec:ComputingMappingHomology}, local strong $h$-homotopies yield in \ref{sec:PartiallyReducedQuiver} a deformation retraction of a quiver $G$ to the partially reduced quiver $\tilde{\mathcal{R}}(G)$, which collapses multiple edges incident to vertices with a loop. This construction leads in Section~\ref{sec:PartialFlagComplex} to a proof that $H_*^{\Delta,m}(G)$ is isomorphic to the homology of a construction we call the partially reduced direct flag complex $\tilde{\mathcal{F}}_<(G)$, which in turn is applied to show $H^{\Delta,h}$ is invariant under a weaker notion of weak local strong homotopy of quivers. Crucially, the $\Delta$-set $\tilde{\mathcal{F}}_<(G)$ like $\bar{\mathcal{F}}$ for $H_*^{\Delta,m}$ provides a better intuitive description of the information captured by $H_*^{\Delta,h}$ and is a basis for more efficient computation of persistent $H_*^{\Delta,h}$. In addition, the structure of $\tilde{\mathcal{F}}_<(G)$ demonstrates the dependence of $H_*^{\Delta,h}$ on both loops and multiple edges implying that unlike $H_*^{\Delta,m}$, $H_*^{\Delta,h}$ detects the additional structure of the quiver not captured in its reduced digraph $\bar{\mathcal{F}}(G)$, while simultaneously satisfying a more flexible homotopy invariance proprieties than $H_*^{\Delta,i}$.

Finally, based on the existing memory efficient parallelizable algorithms for computation of the directed flag complex of a digraph \cite{Lutgehetmann2020}, in Appendix~\ref{sec:Algorithms} we detail the precise steps for the computation of the directed flag complex $\mathcal{F}(G)$, reduced directed flag complex $\bar{\mathcal{F}}(G)$, and partially reduced directed flag complex $\tilde{\mathcal{F}}(G)$ of a quiver $G$. All algorithms are considered over filtered quivers so that they may be applied together with existing packages for persistent homology and demonstration code is made available at \cite{Burfitt2025}. We additionally analyse the computational complexity of each algorithm, detailing the time savings when computing the persistent homology of $\bar{\mathcal{F}}(G)$ and $\tilde{\mathcal{F}}(G)$ over direct computation of persistent $H_*^{\Delta,m}(G)$ and $H_*^{\Delta,h}(G)$, respectively.

\section{Background}\label{sec:Background}

We now provide the necessary background on simplicial homology and singular simplicial homologies of digraphs and quivers used during the remainder of this work.
Throughout the section assume that $n$ is a non-negative integer and $R$ a commutative ring, unless stated otherwise.

\subsection{Abstract simplicial complexes and $\Delta$-sets}\label{sec:Spaces}

In this subsection we detail the properties of the models of space used in this work. We make use of abstract simplicial complexes and $\Delta$-sets as they provide flexible combinatorial structures that work well for computer computations.

An \emph{abstract simplicial complex} $(S,V)$ consist of a \emph{vertex set} $V$ and a non-empty set of finite subsets $S$ of $V$ such that
\[
    s\in S \; \text{and} \; s'\subseteq s \implies s'\in S.
\]
An $s \in S$ with size $n+1$ is called an \emph{$n$-simplex} of $S$.

Let $(S_1,V_1)$ and $(S_2,V_2)$ be abstract simplicial complexes. A \emph{simplicial map} of abstract simplicial complexes $f\colon (S_1,V_1) \to (S_2,V_2)$ is a map of sets $f\colon V_1\to V_2$ such that
\[
    \{v_0,\dots,v_n\}\in S_1 \implies \{ f(v_0),\dots,f(v_n)\} \in S_2.
\]
Denote the category of abstract simplicial complexes and simplicial maps by $\textbf{ASim}$.

A \emph{$\Delta$-set} or \emph{semi-simplicial set} $X$ consists of a sequence of sets $\{X_n\}_{n=0}^\infty$ and \emph{face maps}
\[
    d_i^n \colon X_{n+1} \to X_{n}
\]
for each integer $n\geq 0$ and $i=0,\dots,n+1$, such that
\begin{equation}\label{eq:FaceMapConditions}
    d_i^{n} \circ d_j^{n+1} = d_{j-1}^{n} \circ d_i^{n+1}
\end{equation}
where $j=1,\dots,n+1$ and $i<j$. In this case, $X_0$ is the set of \emph{vertices} of the $\Delta$-set $X$.
More generally, $X_n$ is the set of \emph{$n$-simplices} of the $\Delta$-set $X$.

Let $(S,V)$ be an abstract simplicial complex and $<$ a total order on the vertex set $V$.
We obtain a $\Delta$-set $X$ from $(S,V)$ and $<$ by setting each $X_n$ to be the set of $n$-simplices of $(S,V)$ and
\[
    d_i^n(\{v_0,\dots,v_{n+1}\}) = \{v_0,\dots,v_{n+1}\} \setminus \{v_i\}
\]
when $v_0,\dots,v_{n+1}$ are ordered by $<$.
Conversely, not every delta set can be obtained from an abstract simplicial complex.
For example, a $\Delta$-set can poses elements $x_1,x_2\in X_n$ whose set of images in $X_0$ under compositions of face maps
\[
    d^0_{i_0}\circ \cdots \circ d^{n-1}_{i_{n-1}}
\]
for some $i_j=0,\dots,j+1$ and $j=0,\dots,n-1$, are identical.
For a $\Delta$-set to be realised as an abstract simplicial complex using the construction above, the set of vertices of each simplex $x\in X_n$ described by the images of the previous equation must be unique and have size $n+1$ for each $n\geq 0$.

Let $X$ be $\Delta$-set and $R$ a commutative ring.
Then we form a chain complex on the free graded $R$-module $C_n(X)=R[X_n]$, with differential
\begin{equation}\label{eq:DeltaSetDiff}
    \partial_n = \sum_{i=0}^n (-1)^id^{n-1}_i.
\end{equation}
The homology $H_*(X)$ of a $\Delta$-set $X$ is the homology of the chain complex described above.
In the case of a $\Delta$-set obtained from an abstract simplicial complex, the homology is independent of the total order on the vertex set. Therefore, the homology of an abstract simplicial complex may be defined by the same construction.

A morphism $m\colon X \to Y$ between $\Delta$-sets $X$ and $Y$ is a sequence of functions $m_n\colon X_n \to Y_n$ such that
\[
    d_i^n \circ m_{n+1} = m_n \circ d_i^n
\]
for each integer $n\geq 0$ and $i=0,\dots,n+1$.
We denote the category of $\Delta$-sets and morphisms of $\Delta$-sets by $\textbf{DSets}$. 

Given a morphism of $\Delta$-sets $m \colon X\to Y$, the object wise construction of chain complexes described by equation~\eqref{eq:DeltaSetDiff} can be extended to a functor $C_*$ from $\textbf{DSets}$ to the category of chain complex by linearly extending each $m_n$ to a chain map ${m_n}_\#\colon R(X_n) \to R(Y_n)$. Passing to homology, this makes the homology of $\Delta$-sets a functor.

\subsection{Digraphs and quivers}\label{sec:Quivers}

The central objects of study in this work are digraph and quivers, the former being a special case of the latter. The primary distinction between digraphs and quivers lies in the types of edges allowed in each instance. Digraph and quivers appear frequently in applications as the fundamental objects describing directed pairwise relations in complex networks. While the development of homology theories of digraphs has received significant attention, quivers have received far less.

A \emph{digraph} $G=(V_G, E_G)$, consists of a non-empty set of \emph{vertices} $V_G$ and a set of \emph{edges}
\[
    E_G \subseteq \{ (u,v) \in V_G \times V_G \: | \: u \neq v \}.
\]
We also denote an edge $(u,v) \in E_G$ by $u \to v$.

A \emph{quiver} $G=(V_G, E_G, s_G, t_G)$, consists of a non-empty set of \emph{vertices} $V_G$ and a set of \emph{edges} $E_G$, as well as maps
\[
    s_G\colon E_G \to V_G
    \;\;\; \text{and} \;\;\;
    t_G\colon E_G \to V_G
\]
called the \emph{source} and \emph{target} functions, respectively. 
We often drop the subscript $G$ from $s_G$ and $t_G$, as it is usually clear from the context which quiver they are associated to.
A \emph{subquiver} of $G$ is a quiver $H$ such that $V_H \subseteq V_G$, $E_H\subseteq E_G$ and $s_H$, $t_H$ are the restrictions of $s_G$, $t_G$ to $E_H$.
A subquiver $H$ of $G$ is \emph{full} if $E_H=\{e \in E_G \mid s_G(e) \in V_H \: \text{and} \: t_G(e) \in V_H \}$.
In addition, given a quiver $G$ and vertices $u,v\in V_G$, denote by $G_{u,v}$ the number of edges $e\in E_G$ such that $s_G(e)=u$ and $t_G(e)=v$.

A digraph $G$ naturally carries the structure of a quiver by setting
\[
    s_G((u,v)) = u
    \;\;\; \text{and} \;\;\;
    t_G((u,v)) = v
\]
for each $(u,v)\in E_G$.
Therefore, from now on we usually treat a digraph as a special case of a quiver without explicitly stating the application of the above identification.

We also make use of the following additional terminology for quivers.
A quiver $G$ is said to have a \emph{double edge} between $u,v \in E_G$ if there are $d_1,d_2\in E_G$ such that
\[
    s(d_1)=t(d_2)=u\neq v=t(d_1)=s(d_2).
\]
A quiver $G$ \emph{has double edges} if it has at least one double edge between some pair of vertices.
An edge $d \in E_G$ is called a \emph{double edge} if there is a double edge between $s(d)$ and $t(d)$.

A quiver $G$ has a \emph{multiple edge} between $u \in V_G$ and $v\in V_G$
if there are $m_1,m_2\in E_G$ such that
\[
    m_1 \neq m_2, \;\;\;
    s(m_1)=s(m_2) = u,
    \;\;\; \text{and} \;\;\;
    t(m_1)=t(m_2) = v.
\]
A quiver $G$ \emph{has multiple edges} if it has at least one multiple edge between some pair of vertices. An edge $m \in E_G$ is called a \emph{multiple edge} if there is are multiple edges between $s(m)$ and $t(m)$.

Finally, a quiver $G$ is said to have a loop at vertex $v \in V_G$, if there is an $l\in E_G$ such that
\[
    s(l)=t(l).
\]
A quiver $G$ \emph{has loops} if it has at least one loop at a vertex.
An edge $e\in E_G$ is said to \emph{have a loop}, when either of the vertices $s(e)$ or $t(e)$ have a loop. As a quiver, a digraph can have double edges but not multiple edges or loops 

Throughout this work three different class of morphisms between quivers are considered. These morphisms are; quiver maps, quiver homomorphisms, and quiver inclusions, with each a specialisation of the preceding morphism type.

More precisely, let $G$ and $G'$ be quivers. 
A \emph{map} of quivers $m\colon G \to G'$ is a pair of functions
\[
m_V \colon V_G \to V_{G'}
\;\;\; \text{and} \;\;\;
m_E \colon E_{G} \to E_{G'} \coprod V_{G'}
\]
such that for every $e\in E_G$, either
\begin{align*}
    &
    m_E(e) \in E_{G'}
    \: \text{and} \:
    m_V(s(e)) = s(m_E(e)),\: m_V(t(e)) = t(m_E(e)) \: \\
    \text{or} \; &
    m_E(e) = m_V(s(e)) = m_V(t(e)) \in V_{G'}.
\end{align*}
When the image of $m_E$ lies in $E_{G'}$, $m$ is called a \emph{homomorphism}.
If in addition $m_V$ and $m_E$ are also injective, then $m$ is an \emph{inclusion}.
When it is clear from the context what is meant, from now on we write $m$ instead of $m_V$ or $m_E$.

Each choice of morphism type provides a different category of quivers, and each such category will be a subcategory of a category formed using a less strict morphism type. We denote the categories of quivers formed using inclusions, homomorphisms, and maps by
\[
    \textbf{Quiv}_i \subset \textbf{Quiv}_h \subset \textbf{Quiv}_m,
\]
respectively.

\subsection{Singular simplicial homologies of digraphs and quivers}

The homology theories presented in this subsection were defined in \cite{Li2024} alongside cubical homology theories on two types of singular $n$-cubes with respect to the three morphism types set out in the previous subsection. Throughout this section assume that $G$ is a quiver and $n\geq 0$ an integer, unless otherwise stated.

The \emph{$n$-dimensional directed simplex} $\Delta^n$ is a digraph on vertex set
\[
    V_{\Delta^n} = \{0,\dots,n\}
\]
and edge set
\[
    E_{\Delta^n} = \{(a,b) \in V_{\Delta^n} \times V_{\Delta^n} \: | \: a < b  \}.
\]
Face maps $\delta_i\colon \Delta^{n} \to \Delta^{n+1}$ for each $i=0,\dots,n+1$ are given by the unique digraph maps satisfying
\begin{equation}\label{eq:SingularSimplicialFaceMap}
    \delta_i(j) = 
    \begin{cases}
        j \ &\text{if} \: j<i,
        \\
        j+1  &\text{otherwise.}
    \end{cases}
\end{equation}

A \emph{singular $n$-simplex} in $G$ is a quiver map $f\colon\Delta^n \to G$, a \emph{singular $n$-simplex homomorphism} is a homomorphism $f\colon\Delta^n \to G$, and a \emph{singular $n$-simplex inclusion} is an inclusion $f\colon\Delta^n \to G$. 
Collectively each of the types of singular morphism $f\colon\Delta^n \to G$ are referred to as \emph{singular simplices}.
We also usually drop the $n$ from each of the notions of singular simplices when the dimension is clear from the context, writing singular simplex instead of singular $n$-simplex.

\begin{definition}
    Let $G$ be a quiver, $n\geq 0$ an integer and $R$ a commutative ring.
    \begin{itemize}
        \item
        Denote by $C^{\Delta,m}_n(G;R)$ the free $R$-module generated by singular $n$-simplices in $G$.
        \item
        Denote by $C^{\Delta,h}_n(G;R)$ the free $R$-module generated by singular $n$-simplex homomorphisms in $G$.
        \item
        Denote by $C^{\Delta,i}_n(G;R)$ the free $R$-module generated by singular $n$-simplex inclusions in $G$.
    \end{itemize}
    In each case, define $\partial_n(f)\colon \Delta^{n-1} \to G$ by
    \[
        \partial_n(f) = \sum^n_{i=0} (-1)^i(f \circ \delta_i)
    \]
    for $f\colon \Delta^n \to G$ and $n \geq1$, which linearly extends to a differential on each graded $R$-modules, respectively. 
    Collectively the resulting chain complexes are referred to as \emph{singular simplicial chains}.
    The homology of these chain complexes provides \emph{singular simplicial homologies}
    \[
        H^{\Delta,m}_*(G;R),\;\;\; H^{\Delta,h}_*(G;R),\;\;\; \text{and} \;\;\; H^{\Delta,i}_*(G;R),
    \]
    called \emph{singular simplicial homology}, \emph{singular simplicial homomorphism homology}, and \emph{singular simplicial inclusion homology}, respectively.
\end{definition}

The results presented in this work are not affected by the choice of coefficients $R$. Therefore, following standard convention we now suppress the $R$ dependence from each of the chain complexes and homologies defined above.

Given a quiver map, homomorphism, or inclusion $\phi \colon G \to G'$, there are induced chain maps
\begin{equation}\label{eq:InducedMapHomology}
    \phi_{\#} \colon C^{\Delta,m}_*(G) \to C^{\Delta,m}_*(G'),\;\;\;
    \phi_{\#} \colon C^{\Delta,h}_*(G) \to C^{\Delta,h}_*(G'),
    \;\;\; \text{and} \;\;\;
    \phi_{\#} \colon C^{\Delta,i}_*(G) \to C^{\Delta,i}_*(G') 
\end{equation}
given by linearly extending $\phi_\#(f)=\phi \circ f$ for each type of singular simplex $f \colon \Delta^n \to G$, respectively. These chain maps induce graded module homomorphisms
\[
    \phi_* \colon H^{\Delta,m}_*(G) \to H^{\Delta,m}_*(G'),\;\;\;
    \phi_* \colon H^{\Delta,h}_*(G) \to H^{\Delta,h}_*(G'),
    \;\;\; \text{and} \;\;\;
    \phi_* \colon H^{\Delta,i}_*(G) \to H^{\Delta,i}_*(G') 
\]
in each case, respectively.

The following proposition resulting from \cite[Proposition 5.2, 5.13 and \S 6]{Li2024} is obtained from the construction of the induced homomorphisms above.

\begin{proposition}[\cite{Li2024}]\label{prop:InitialFuctoriality}
    Using the induced homomorphisms provided above, the following statements hold.
    \begin{enumerate}[(1)]
        \item 
        The homology $H^{\Delta,m}_*$ is functorial with respect to quiver maps.
        \item 
        The homology $H^{\Delta,h}_*$ is functorial with respect to of quiver homomorphisms.
        \item
        The homology $H^{\Delta,i}_*$ is functorial with respect to quiver inclusions.
    \end{enumerate}
\end{proposition}

\subsection{The acyclic carrier theorem}

Motivated by the structure of simplicial chains, the acyclic carrier theorem is a classical result that provides a general algebraic framework for the construction of chain homotopies.
The material covered this subsection can be found in \cite[\S 1.13]{Munkers1984} and provides a key tool for the proofs of some of our main results. In this subsection, all chain complexes are considered to be over a commutative ring $R$, unless stated otherwise.

Throughout this work, we denote by $\textbf{Chain}_R$ or $\textbf{Chain}$ the category of chain complexes and chain maps with coefficient in $R$. Given a chain map $g\colon C_* \to D_*$, we denote by $g_* \colon H_*(C_*)\to H_*(D_*)$ the induced map on the homology of the chain complexes $C_*$ and $D_*$.

An \emph{augmentation} of the chain complex $(C_*,\partial)$ is a ring homomorphism $\varepsilon \colon C_0 \to R$ such that
\[
    \varepsilon \circ \partial = 0.
\]
In this case $(C_*,\partial,\varepsilon)$ is called an \emph{augmented chain complex}. The homology of the chain complex obtained from an augmented chain complex $(C_*,\partial,\varepsilon)$ by setting
\[
    C_{-1}=R 
    \;\;\; \text{and} \;\;\;
    \partial_{-1} = \varepsilon
\]
is called the \emph{reduced homology} of $(C_*,\partial,\varepsilon)$.
If all the reduced homology groups of an augmented chain complex are trivial then it is called \emph{acyclic}.

All chain complexes $C_*$ considered in this work are freely generated in degree $0$. In this case, we may chose an augmentation by linearly extending $\varepsilon(v)=1_R$ for each element $v$ in a given basis of $C_0$.
A quiver $G$ is called \emph{acyclic} with respect to a particular singular simplicial homology, if its chain complex is acyclic with respect to the augmentation above provided by the basis of singular $0$-simplices.

A chain map $f_*\colon C_* \to C'_*$ between augmented chain complexes $(C_*,\partial,\varepsilon)$ and $(C'_*,\partial',\varepsilon')$ is called \emph{augmentation preserving} if
\[
    \varepsilon' \circ f_0 = \varepsilon.
\]

\begin{definition}
    Let $(C_*,\partial,\varepsilon)$ and $(C'_*,\partial',\varepsilon')$ be augmented chain complexes, such that $C_*$ is a free module in each degree.
    Suppose that $\{ c_n^i \}_{i\in I_n}$ is an $R$-basis of $C_n$ indexed over the set $I_n$ for each $n\geq 0$.
    Then an \emph{acyclic carrier}, $\varphi$ from $C_*$ to $C'_*$ on basis $\{c_*^i\}_{I_*}$ is a sequence of functions $\varphi_n$ assigning $c_n^i$ to a chain subcomplex $\varphi_n(c_n^i)$ of $C'_*$ for each integer $n\geq 0$ and $i\in I_n$, such that:
    \begin{enumerate}[(1)]
        \item 
        Each $\varphi_n(c_n^i)$ is augmented by $\varepsilon'$ and is acyclic.
        \item
        For $n\geq 1$ and each $j\in I_{n-1}$ such that element $c_{n-1}^{j}$ appears in a non-zero term of some $\partial c_n^i$, we have that $\varphi_n(c_{n-1}^j)$ is a subcomplex of $\varphi_n(c_n^i)$.
    \end{enumerate}
    An augmentation preserving chain map $f_*\colon C_* \to C'_*$ is said to be carried by $\varphi$ if
    \[
        f_n(c_n^i) \in \varphi_n(c_n^i)
    \]
    for any $n\geq 0$ and $i \in I_n$.
\end{definition}

The following theorem is know as the acyclic carrier theorem.

\begin{theorem}\label{thm:AcyclicCarriers}
    Let $(C_*,\partial,\varepsilon)$ and $(C'_*,\partial',\varepsilon')$ be augmented chain complexes, such that $C_*$ is a free module in each degree.
    Suppose that $\varphi$ is an acyclic carrier from $C_*$ to $C'_*$ with respect to some chosen basis on $C_*$.
    Then any two augmentation preserving chain maps carried by $\varphi$ are chain homotopy equivalent.
\end{theorem}

A proof of the acyclic carrier theorem can be found in \cite[proof of Theorem 13.4]{Munkers1984}.

\subsection{The directed flag complex}\label{sec:DirectedFlagComplex}

The directed flag or clique complex was first defined in \cite{Masulli2017} as a simplicial complex constructed from a digraph without double edges. Following its construction, the directed flag complex has been widely used in applications due to its intuitive simplicity and relative ease of computation \cite{Lutgehetmann2020}.

For the purposes of comparison to the other spaces constructed later in this work, we provide here a generalised $\Delta$-set construction of the directed flag complex defined over any quiver. This new functor generalise the usual simplicial complex associated to a digraph, see \cite[\S 4.1.3]{Reimann2017}. In addition, we provide an efficient algorithm for the computation of the directed flag complex in Appendix~\ref{sec:AlgDirectedFlagComplex}. However, we note that the presence of loops in a quiver does not effect the construction of the directed flag complex (see Prostitution~\ref{prop:SimplicalNoLoop}).

\begin{definition}\label{def:DirectedFlag}
    Let $G$ be a quiver. Then the \emph{directed flag complex} $\mathcal{F}(G)$ is a $\Delta$-set defined as follows. The $n$-simplices of $\mathcal{F}(G)$ are given by
    \begin{align*}
        \mathcal{F}(G)_n = \{ & \{ v_i \in V_G , e_{j,k} \in E_G \} \: | \: i=0,\dots,n, \: \text{the} \: v_i \: \text{are distinct,}
        \\ &
        0 \leq j < k \leq n \: \text{are integers, and} \: s(e_{j,k}) = v_j, \:t(e_{j,k}) = v_k \}.
    \end{align*}
    The face map $d^{n-1}_t(x)$ for $t= 0,\dots,n$ and $x\in \mathcal{F}_n(G)$ is given by the subset of $x$ obtained by removing $v_t$, all edges $e_{a,t} \in x$ for $a=0,\dots,t-1$, and all edges $e_{t,b} \in x$ for $b=t+1,\dots,n$.
\end{definition}

As the elements of $\mathcal{F}(G)_n$ coincidence precisely with the images of singular simplex inclusions, we immediately obtain that
\begin{equation}\label{eq:FlagComplexInclusionHomology}
    H_*^{\Delta,i}(G) = H_*(\mathcal{F}(G)).
\end{equation}
Moreover, an inclusion of quivers $\phi \colon G \to G'$ induces an inclusion of $\Delta$-sets 
\begin{equation}\label{eq:FlagFunctor}
    \mathcal{F}(\phi)\colon \mathcal{F}(G) \to F(G')
    \;\;\; \text{given by} \;\;\;
    \mathcal{F}(\{ v_i , e_{l,k} \}_{\substack{0\leq i \leq n \\ 0 \leq j < k \leq n}}) = \{ \phi(v_i) , \phi(e_{j,k}) \}_{\substack{0\leq i \leq n \\ 0 \leq j < k \leq n}},
\end{equation}
making $\mathcal{F}$ a functor that acts naturally with respect to the isomorphism in equation~\eqref{eq:FlagComplexInclusionHomology}. In particular, the naturality of the functor $\mathcal{F}$ provides the structure necessary for persistent $H_*^{\Delta,i}$ to be computed as the persistent homology of the corresponding filtered directed flag complex.

\subsection{Quiver homotopies}

The first construction of digraph homotopies of which we are aware appeared in \cite{Grigoryan2014} as an invariant of the path homology of digraphs containing an earlier constructions of homotopies for undirected graphs as a special cases.
A more general framework of homotopies for path complexes (originating in \cite[\S 3]{Grigoryan2013}) was recently provided by \cite{Chaplin2024}, within the context of an abstract categorical formulation.
In addition, homotopies in the alterative setting of closure spaces that restrict to several notions of homotopy within the category of digraphs where detailed in \cite[\S 4]{Bubenik2024}.
For the purposes of this work we focus on the generalisations of digraph homotopy to quivers first appearing in \cite{Li2024}, which we further develop in Section~\ref{sec:HomotopyTheory}. Throughout this section assume that $G$ is a quiver, unless otherwise stated.

The \emph{line digraph}, $I$ is the digraph with
\begin{equation}\label{eq:Interval}
    V_I = \{ 0,1 \}
    \;\;\; \text{and} \;\;\;
    E_I = \{ (0,1) \}
\end{equation}
which coincides with the digraph $1$-simplex $\Delta^1$.

The original notion of digraph homotopy was based on the box or cartesian product $G \square I$ of a digraph $G$ with $I$. However, it is the strong box product that is the categorical product in $\textbf{Quiv}_m$ and for the singular simplicial homology theories considered in this work we require only the strong box product of quivers. Therefore, for the purpose of avoiding confusion we retain the word strong throughout the reminder of the paper. 

\begin{definition}\label{def:StrongBoxProduct}
    Let $G=(V_G,E_G,s_G,t_G)$ and $G=(V_{G'},E_{G'},s_{G'},t_{G'})$ be quivers.
    For simplicity, assume also that $V_G,\:E_G,\:V_{G'}$ and $E_{G'}$ are disjoint sets.
    The \emph{strong box product} of $G$ and $G'$ denoted $G \boxtimes G'$, is the quiver with
    \begin{align*}
        V_{G \boxtimes G'}
        = V_G \times V_{G'}, \;\;\;
        E_{G \boxtimes G'}
        = (E_G \times V_{G'}) \cup (V_G \times E_{G'}) \cup (E_G \times E_{G'})
    \end{align*}
    and
    \begin{align*}
        s_{G\boxtimes G'}((e,v')) &= (s_G(e),v'), \;\;\;
        s_{G\boxtimes G'}((v,e')) = (v,s_{G'}(e')), \;\;\;
        s_{G\boxtimes G'}((e,e')) = (s_G(e),s_{G'}(e')), \\
        t_{G\boxtimes G'}((e,v')) &= (t_G(e),v'), \;\;\;\:
        t_{G\boxtimes G'}((v,e')) = (v,t_{G'}(e')), \;\;\;\:
        t_{G\boxtimes G'}((e,e')) = (t_G(e),t_{G'}(e')) \;
    \end{align*}
    where $v\in V_G$, $v'\in V_{G'}$, $e\in E_G$, and $e' \in E_{G'}$.
\end{definition}

In the case of $G \boxtimes I$, there are two natural inclusions
\begin{align}
     &
     i_0 \colon G \to G \boxtimes I
     \;\;\; \text{given by} \;\;\;
     v \mapsto (v,0),
     \;
     e \mapsto (e,0)
     \nonumber
     \\
     \text{and} \;\;\;
     &
     i_1 \colon G \to G \boxtimes I
     \;\;\; \text{given by} \;\;\;
     v \mapsto (v,1),
     \;
     e \mapsto (e,1)
     \label{eq:BoxInclusion}
\end{align}
where $v \in V_G$ and $e\in E_G$.

A choice of quiver product and type of morphisms leads to the construction of digraph homotopies in the following manner.

\begin{definition}\label{def:StrongHomtopoic}
    Quiver maps $f_0,f_1\colon G \to H$ are \emph{$1$-step strong homotopic} if there is a quiver map
    \begin{align*}
        F\colon G\boxtimes I \to H
        \;\;\; \text{such that} \;\;\;
        F \circ i_0 = f_0,
        \;
        F \circ i_1 = f_1.
    \end{align*}
    In this case we write $f_0 \simeq_1^S f_1$.
    More generally, quiver maps $f,g \colon G \to H$ are \emph{strong homotopic} if they are related by the equivalence relation generated by $1$-step strong homotopies.
    In which case we write $f \simeq^S g$.
    Quivers $G$, $G'$ are \emph{strong homotopy equivalent} if there are quiver maps
    \begin{align*}
        f\colon G \to G',
        \;
        h\colon G' \to G
        \;\;\; \text{such that} \;\;\;
        h \circ f \simeq^S \text{id}_G,
        \;
        f \circ h \simeq^S \text{id}_{G'}.
    \end{align*}
    A quiver is called \emph{strongly contractable} if it is strong homotopy equivalent to a quiver consisting of a single vertex and no edges.
    Similarly, we also obtain the construction of \emph{$1$-step strong $h$-homotopic}, \emph{strong $h$-homotopic}, $\simeq^{Sh}$, \emph{strong $h$-homotopy equivalent}, and \emph{strongly $h$-contactable}, by making the above definitions with respect to quiver homomorphisms rather than quiver maps.
\end{definition}

In the reminder of this work, we often develop the strong homotopy and strong $h$-homotopy cases in parallel indicating inside brackets the differences in terminology.

\begin{remark}\label{rmk:OneStepHomotopyTransitivity}
    To be sure $\simeq_1^S$ generates an equivalence relation with a well defined composition of homotopy classes, we need to check that the \emph{transitivity} property. More precisely, given quiver maps $f \colon G_1 \to G_2$, $g,g' \colon G_2 \to G_3$, and $h \colon G_3 \to G_4$ we require that
    \[
        g \simeq_1^S g' \implies 
        g\circ f \simeq_1^S g' \circ f
        \;\;\; \text{and} \;\;\;
        h \circ g \simeq_1^S h \circ g'.
    \]
    If $F \colon G_2 \boxtimes I \to G_3$ is a one step strong homotopy from $g$ to $g'$, then $h \circ g \simeq_1^S h \circ g'$ is realised by $h \circ F$. Furthermore, let $F'\colon G_1 \boxtimes I \to G_2 \boxtimes I$ be given by 
    \begin{align*}
        F'((v,0)) &= (f(v),0),\; 
        F'((v,1)) = (f(v),1),\;
        F'((e,0)) = (f(e),0), \\
        F'((e,1)) &= (f(e),1),
        F'((v,0\to 1)) = (f(v), 0\to 1), 
        \; \text{and} \;
        F'((e,0\to 1)) = (f(e), 0\to 1)
    \end{align*}
    for each $v \in V_{G_1}$ and $e \in E_{G_1}$. Then we verify that $g\circ f \simeq_1^S g' \circ f$ using the map $F \circ F'$. Moreover, when $f$ is a homomorphism $F'$ is also a homomorphism. Therefore, transitivity of one step strong $h$-homotopies is obtained by precisely the same argument as above.
\end{remark}
 
A map $r \colon G \to G$ is called a \emph{retraction} onto subquiver $H$ of $G$ if $r(V_G) \subseteq V_H$, $r(E_G) \subseteq V_H \coprod E_H$ and the restrictions of $r$ to the vertices and edges of $H$ are the identities.
Retraction $r \colon G \to G$ is called a \emph{strong ($h$-)deformation retraction} if there is a strong ($h$-)homotopy from $r$ to the identity on $G$. As in classical homotopy theory, the existence of a strong ($h$-)deformation retraction of a quiver $G$ onto a subquiver $H$ implies that $G$ is strong ($h$-)homotopy equivalent to $H$.

The next theorem is obtained from \cite[Theorem 5.8 and Theorem 5.17]{Li2024}.

\begin{theorem}[\cite{Li2024}]
\label{thm:PreviousHomotopyInv}
    The following homotopy invariance properties hold.
    \begin{enumerate}[(1)]
        \item 
        If quiver maps $f,g \colon G \to G'$ satisfy $f \simeq^S g$, then 
        \[
            f_*=g_* \colon H_*^{\Delta,m}(G) \to H_*^{\Delta,m}(G).
        \]
        \item 
        If quiver homomorphisms $f,g \colon G \to G'$ satisfy $f \simeq^{Sh} g$, then 
        \[
            f_*=g_* \colon H_*^{\Delta,h}(G) \to H_*^{\Delta,h}(G).
        \]
    \end{enumerate}
\end{theorem}

As stated in \cite[\S 6]{Li2024}, it is straightforward to check that $H_*^{\Delta,i}$ is invariant under isometries.
In the case of digraphs, a more general construction of homotopies under which $H_*^{\Delta,i}$ is invariant is provided for a certain class of triangle collapsing maps on directed flag complexes in \cite[Definition 5.13]{Chaplin2024}.
In particular, these homotopies are derived from homotopies of path complexes, which generalise the original notion of digraph homotopy.

\section{Conditions on equivalence between singular simplicial morphisms and homologies}\label{sec:HomologyRelations}

In this section we examine the relationships between the singular simplicial homologies within the categories $\textbf{Quiv}_m$, $\textbf{Quiv}_h$, and $\textbf{Quiv}_i$. More precisely, we investigate under what conditions homologies $H^{\Delta,m}_*$, $H^{\Delta,h}_*$, $H^{\Delta,i}_*$ coincide and provide examples to demonstrate when these singular simplicial homologies are in general distinct for quivers with and without double edges, multiple edges, or loops. The main results presented in the section largely follow directly from the realisation that under certain conditions the different types of singular simplices are in fact identical, and we do not require the construction of quasi-isomorphisms between the chain complexes.

Throughout the section assume that $G$ is a quiver and $n\geq 0$ is an integer, unless otherwise stated.
Note that some results in this section require the axiom of choice if the quiver $G$ contains an infinite number of edges.

We first consider the case when the quiver has no loops.

\begin{proposition}\label{prop:SimplicalNoLoop}
    Let $G$ be a quiver without loops, then a singular simplicial map $f\colon \Delta^n \to G$ is an inclusion if and only if it is a homomorphism.
    In particular,
    \[
        H^{\Delta,i}_*(G) = H^{\Delta,h}_*(G)
        .
    \]
\end{proposition}

\begin{proof}
    When $n=0$ homomorphisms $f \colon \Delta^0 \to G$ are the same as inclusions of a vertex.
    For $n\geq 1$, suppose $f\colon \Delta^n \to G$ is a homomorphism which is not an inclusion, then there must be a pair of distinct vertices sent to the same vertex under $f$.
    Any pair of distinct vertices in the $n$-simplex $\Delta^n$ are connected by a directed edge in some direction.
    Therefore, as $f$ is a homomorphism there must exist an edge of $\Delta^n$ sent to a loop under $f$, which contradicts the assumption that $G$ is a quiver without loops. 
\end{proof}

Given a quiver with at least one loop, the statement of the proposition above immediately fails to hold,
as in this case there is a homomorphism of $\Delta^1$ onto any loop and this is not an inclusion. In particular, $H^{\Delta,i}_*$ is unaffected by the removal or addition of loops in the quiver, while $H^{\Delta,h}_*$ might vary. However, it is straightforward to see that all homomorphisms from $\Delta^n$ to a loop in dimensions $n\geq 1$ generate an acyclic sub-chain complex of $C^{\Delta,h}_*(G)$. A demonstration that $H^{\Delta,i}_*$ and $H^{\Delta,h}_*$ differ in general is provided in Example~\ref{ex:DoubleCone}.

\begin{definition}
    A singular simplex $f \colon \Delta^n \to G$ is called \emph{degenerate} if there is an $i=1,\dots,n$ such that
    \[
        f(i-1) = f(i) = f(i-1 \to i).
    \]
    Otherwise, $f$ is called \emph{non-degenerate}.
\end{definition}

Ignoring the additional condition on the image of the edge,  the definition of degenerate singular simplices provided above coincides with the property satisfied by the elements of the conical basis corresponding to degenerated simplices in the usual chain complex associated to a $\Delta$-sets. Consequently, the $C^{\Delta,m}_*(G)$ sub-chain complex $C^{\Delta,dm}_*(G)$ generated by degenerate simplices is an acyclic direct summand of $C^{\Delta,m}_*(G)$. Therefore, we obtain that
\begin{equation}\label{eq:DegenerateQuotientHomology}
    H^{\Delta,m}_*(G) = H_*\left(\frac{C^{\Delta,m}_*(G)}{C^{\Delta,dm}_*(G)}\right).
\end{equation}

We now consider when $H^{\Delta,m}_*$ coincides with $H^{\Delta,i}_*(G)$ and $ H^{\Delta,h}_*(G)$ in general.

\begin{proposition}\label{prop:SimplicalNoDoubleMultiLoop}
    Let $G$ be a digraph without double edges,
    then non-degenerate singular simplicial maps coincide with singular simplicial homomorphisms and singular simplicial inclusions.
    In particular,
    \begin{align*}
        &
        H^{\Delta,i}_n(G) = H^{\Delta,h}_n(G)= H^{\Delta,m}_n(G)
        .
    \end{align*}
\end{proposition}

\begin{proof}
    The equivalence between singular simplicial inclusions and homomorphisms follow immediately from Proposition~\ref{prop:SimplicalNoLoop}.
    Therefore, we need only prove that singular simplicial inclusions and non-degenerate singular simplicial maps coincide.
    
    When $n=0$ or $n=1$ all non-degenerate maps $f \colon \Delta^n \to G$ are inclusions.
    For $n\geq 2$, suppose that $f\colon \Delta^n \to G$ is a non-degenerate map which is not an inclusion. In this case, there must be a pair of vertices $a$ and $c$ of $\Delta^n$ such that $a<c$ and $f(a)=f(c)$. Without loss of generality, we may also assume that $c-a$ is minimal among all pair of vertices satisfying the same property.
    If $c=a+1$, as $f$ is non-degenerate the image of $f$ must contain a loop contradicting the fact that $G$ is a digraph.
    This means that, $c > a+1$ and there is a vertex $b$ of $\Delta^n$ such that $a<b<c$. In particular, as $\Delta^n$ is a simplex $a\to b$, $a\to c$, and $b\to c$ are edges of $\Delta^n$. 
    If $f(c)=f(a)=f(b)$ then $c-a$ would not be minimal among vertices such that $f(a) = f(c)$, contradicting our assumption.
    Hence, $f(c) = f(a) \neq f(b)$ and the image of $f$ contains a double edge as the image of edges $a\to b$ and $b\to c$, which contradicts the assumption in the statement of the proposition.
    Therefore, $f$ must be an inclusion as required.
    
    Finally, the isomorphism between modules $H^{\Delta,i}_n(G)$ and $H^{\Delta,m}_n(G)$ now follows from equation~\eqref{eq:DegenerateQuotientHomology}.
\end{proof}

Given a quiver $G$ with a double edge between vertices $u,v\in G$, a non-degenerate digraph map $f\colon\Delta^2 \to V_G$ that is not an inclusion is constructed by setting $f(0)= f(2) = u $ and $f(1) = v$.
Furthermore, when $u$ has a loop the map $f$ described above can be made into a homomorphism by sending $0\to 2$ to the loop at $u$.
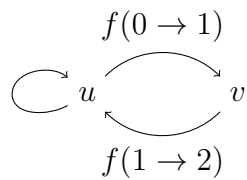
\begin{center}
    \tikz {
        \node (a) at (0,0) {$u$};
        \node (b) at (2,0) {$v$};
        \node (d) at (-1,-0.145) {};
        \node (c) at (-1,0.145) {};
        \draw[->] (a) to [out=45,in=135] node[pos=0.5,above] {$f(0\to 1)$} (b);
        \draw[->] (b) to [out=225,in=315] node[pos=0.5,below] {$f(1 \to 2)$} (a);
        \draw[-] (a) to [out=210,in=270] (c);
        \draw[->] (d) to [out=90,in=140] (a);
        }
\end{center}

Similarly, given a quiver $G$ with a multiple edge between vertices $u$ and $v$, there is a non-degrade digraph map $f'\colon\Delta^2 \to G$ that is not an inclusion constructed by setting by setting $f'(0) = u$, $f'(1) = f'(2) = v$, and requiring that the edges from $0\to 1$ and $0 \to 2$ have different images under $f'$.
Furthermore, when $v$ has a loop the map $f'$ described above can be made into a homomorphism by sending $1\to 2$ to the loop at $v$.
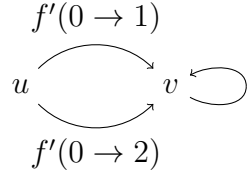
\begin{center}
    \tikz {
        \node (a) at (0,0) {$u$};
        \node (b) at (2,0) {$v$};
        \node (d) at (3,-0.145) {};
        \node (c) at (3,0.145) {};
        \draw[->] (a) to [out=45,in=135] node[pos=0.5,above] {$f'(0\to 1)$} (b);
        \draw[->] (a) to [out=315,in=225] node[pos=0.5,below] {$f'(0 \to 2)$} (b);
        \draw[-] (b) to [out=330,in=270] (c);
        \draw[->] (d) to [out=90,in=30] (b);
    }
\end{center}

While the previous constructions established that singular non-degenerate maps, singular homomorphisms, and singular inclusions do not in general coincide, the following example demonstrates that the presence of double edges, multiple edge, and loops may result in distinct singular simplicial homology in each case.

\begin{example}\label{ex:DoubleCone}
   Let $G$ be the following digraph.
   
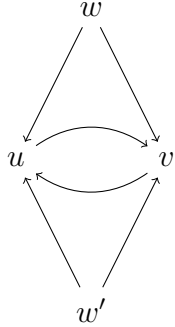
\begin{center}
        \tikz {
            \node (u) at (0,0) {$u$};
            \node (v) at (2,0) {$v$};
            \node (w) at (1,2) {$w$};
            \node (w') at (1,-2) {$w'$};
            \draw[->] (u) to [out=35,in=145] (v);
            \draw[->] (v) to [out=215,in=325] (u);
            \draw[->] (w) -- (u);
            \draw[->] (w) -- (v);
            \draw[->] (w') -- (u);
            \draw[->] (w') -- (v);
        }
    \end{center}
    The inclusion and homomorphism singular simplicial homologies of $G$ are given by
    \begin{equation*}
        H^{\Delta,i}_n(G;\mathbb{Z})=H^{\Delta,h}_n(G;\mathbb{Z})
        =
        \begin{cases}
            \mathbb{Z} &
            \text{if} \:
            n=0 \: \text{or} \: n=2
            \\
            0 & \text{otherwise}.
        \end{cases}
    \end{equation*}
    Where the dimension $2$ homology is generated by the sum of the image of the two singular $2$-simplices with $0$ sent to $w$ subtracted from the sum of the two singular $2$-simplices with $0$ sent to $w'$.
    
    When a loop is added at vertex $u$ to obtain a quiver $G'$ from $G$, the chain complex $C^{\Delta,i}_*(G')$ remains unchanged and $H^{\Delta,i}_n(G';\mathbb{Z}) = H^{\Delta,i}_n(G;\mathbb{Z})$. However, there is a single additional singular generator of $C^{\Delta,h}_{1}(G)$
    lying on the loop, one additional non-degenerate singular $2$-simplex $f$ and two additional non-degenerate singular $3$-simplices $g$ and $h$, uniquely determined by their vertex images 
    \begin{align*}
        f(0)&=u,\:f(1)=v,\:f(2)=u,
        \\
        g(0)&=w,\:g(1)=u,\:g(2)=v,\:g(3)=u
        \\
        \text{and} \;
        h(0)&=w',\:h(1)=u,\:h(2)=v,\:h(3)=u.
    \end{align*}
    As a result, we have that $H^{\Delta,h}_n(G';\mathbb{Z})=0$ for each $n\geq 1$.
    In either case, the homology of the quivers $G$ and $G'$ is acyclic under $H^{\Delta,m}_*$ by Theorem~\ref{thm:PreviousHomotopyInv}, as both $G$ and $G'$ are strongly contractible.
    
    In addition, when the double edge between $u$ and $v$ in the original digraph $G$ is replaced with a single edge, any homology in dimension $2$ vanishes in all cases. Furthermore, the above homologies of $G$ and $G'$ remain unchanged if the double edge $v \to u$ is replaced with an additional edge $u \to v$ forming a multiple edge.
\end{example}

\section{Local strong $h$-homotopy}\label{sec:HomotopyTheory}

We now develop new quiver homotopy theory that will be applied in the subsequent sections of this work.
Motivated by the constrained conditions under which strong $h$-homotopies can be obtained \cite[Lemma 3.7]{Li2024}, the central construction of this section is a weaker notion called local strong $h$-homotopy under which $H_*^{\Delta,h}$ remains invariant. This result greatly improves upon the $H_*^{\Delta,h}$ strong $h$-homotopy invariance provided by \cite[Theorem 5.8]{Li2024} and serves as a key tool for revealing further properties of $H_*^{\Delta,h}$. Throughout the section assume that $G$ is a quiver and $n\geq 0$ an integer, unless otherwise stated.

We are not aware of the next definition having been made previously.
However, the definition would be a special case of the homotopy mapping cylinder construction \cite{Grigoryan2018} if we were making use of the box products of digraphs rather the strong product of quivers.

\begin{definition}\label{def:MappingCylinder}
    Let $G$ be a quiver and $G'$ a full subquiver.
    Then define the \emph{strong mapping cylinder} $M_{G' \hookrightarrow G}^S$ of the inclusion $G' \hookrightarrow G$ to be the quiver with vertices
    \[
        V_{M_{G' \hookrightarrow G}^S}
        = V_G \times \{ 0 \} \cup V_{G'} \times \{ 1 \}
    \]
    and edges
    \begin{align*}
        E_{M_{G' \hookrightarrow G}^S} =
        (E_G \times \{ 0 \}) \cup E_{G' \boxtimes I} \cup E(G,G')
    \end{align*}
    where
    \begin{align*}
        E(G,G') =
        \{ e_{e_G} \: | \:
        e_G \in E_G \setminus E_{G'}, 
        \; s(e_G) \in V_{G'} \; \text{and} \; t(e_G) \notin V_{G'},\; 
        \\
        \text{or}
        \; s(e_G) \notin V_{G'} \; \text{and} \; t(e_G) \in V_{G'} \}
    \end{align*}
    such that
    \[
        s(e_{e_G}) = 
        \begin{cases}
            (s(e_G),1) \ &\text{if} \; s(e_G) \in V_{G'}
            \\
            (s(e_G),0)  &\text{otherwise,}
        \end{cases}
        \;\;\;
        t(e_{e_G}) =
        \begin{cases}
            (t(e_G),1) \ &\text{if} \; t(e_G) \in V_{G'}
            \\
            (t(e_G),0)  &\text{otherwise.}
        \end{cases}
    \]
\end{definition}

Similarly to equation~\eqref{eq:BoxInclusion}, we can define two natural inclusions into the strong mapping cylinder.
The first being
\begin{align*}
     &
     i_0 \colon G \to V_{M_{G' \hookrightarrow G}^S}
     \;\;\; \text{given by} \;\;\;
     v \mapsto (0,v),
     \;
     e_G \mapsto (0,e_G)
     \nonumber
\end{align*}
where $v \in V_G$ and $e_G\in E_G$ and
the second $i_1 \colon G \to M_{G' \hookrightarrow G}^S$ given by
\begin{align*}
     i_1(v)=&
     \begin{cases}
        (v,1) \ &\text{if} \; v \in V_{G'}
        \\
        (v,0)  &\text{otherwise,}
    \end{cases}
     \\
     i_1(e_G)=&
     \begin{cases}
        (e_G,1) \ &\text{if} \; e_G \in E_{G'}
        \\
        (e_G,0) \ &\text{if} \; s(e_G) \notin V_{G'} \: \text{and} \: t(e_G) \notin V_{G'}
        \\
        e_{e_G}  &\text{otherwise.}
    \end{cases}
\end{align*}

The construction of $M_{G' \hookrightarrow G}^S$ is demonstrated in following example.

\begin{example}
    Let $G$ be the quiver
    
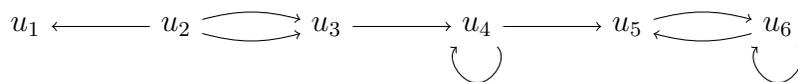
\begin{center}
        \tikz {
            \node (u1) at (0,0) {$u_1$};
            \node (u2) at (2,0) {$u_2$};
            \node (u3) at (4,0) {$u_3$};
            \node (u4) at (6,0) {$u_4$};
            \node (u5) at (8,0) {$u_5$};
            \node (u6) at (10,0) {$u_6$};
            \node (u4b1) at (5.855,-0.75) {};
            \node (u4b2) at (6.145,-0.75) {};
            \node (u6b1) at (9.855,-0.75) {};
            \node (u6b2) at (10.145,-0.75) {};
            \draw[->] (u2) -- (u1);
            \draw[->] (u2) to [out=15,in=165] (u3);
            \draw[->] (u2) to [out=345,in=195] (u3);
            \draw[->] (u3) -- (u4);
            \draw[->] (u4) -- (u5);
            \draw[->] (u5) to [out=15,in=165] (u6);
            \draw[->] (u6) to [out=195,in=345] (u5);
            \draw[-] (u4) to [out=315,in=0] (u4b1);
            \draw[->] (u4b2) to [out=180,in=225] (u4);
            \draw[-] (u6) to [out=315,in=0] (u6b1);
            \draw[->] (u6b2) to [out=180,in=225] (u6);
        }
    \end{center}
    and $G'$ the full subquiver with $v_{G'} = \{ u_3, u_4 \}$.
    Then $M^S_{G'\hookrightarrow G}$ is the following quiver.
    
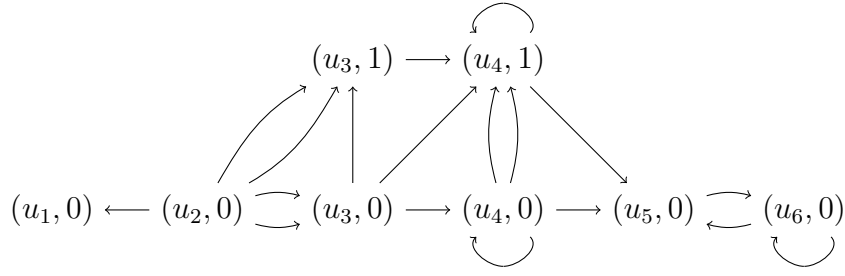
\begin{center}
        \tikz {
            \node (u1) at (0,0) {$(u_1,0)$};
            \node (u2) at (2,0) {$(u_2,0)$};
            \node (u3) at (4,0) {$(u_3,0)$};
            \node (u4) at (6,0) {$(u_4,0)$};
            \node (u5) at (8,0) {$(u_5,0)$};
            \node (u6) at (10,0) {$(u_6,0)$};
            \node (u4b1) at (5.855,-0.75) {};
            \node (u4b2) at (6.145,-0.75) {};
            \node (u6b1) at (9.855,-0.75) {};
            \node (u6b2) at (10.145,-0.75) {};
            \node (u3a) at (4,2) {$(u_3,1)$};
            \node (u4a) at (6,2) {$(u_4,1)$};
            \node (u4ab2) at (6.145,2.75) {};
            \node (u4ab1) at (5.855,2.75) {};
            \draw[->] (u2) -- (u1);
            \draw[->] (u2) to [out=15,in=165] (u3);
            \draw[->] (u2) to [out=345,in=195] (u3);
            \draw[->] (u3) -- (u4);
            \draw[->] (u4) -- (u5);
            \draw[->] (u5) to [out=15,in=165] (u6);
            \draw[->] (u6) to [out=195,in=345] (u5);
            \draw[-] (u4) to [out=315,in=0] (u4b1);
            \draw[->] (u4b2) to [out=180,in=225] (u4);
            \draw[-] (u6) to [out=315,in=0] (u6b1);
            \draw[->] (u6b2) to [out=180,in=225] (u6);
            \draw[->] (u3) -- (u3a);
            \draw[->] (u3) -- (u4a);
            \draw[->] (u3a) -- (u4a);
            \draw[->] (u4) to [out=105,in=255] (u4a);
            \draw[->] (u4) to [out=75,in=285] (u4a);
            \draw[-] (u4a) to [out=45,in=0] (u4ab1);
            \draw[->] (u4ab2) to [out=180,in=135] (u4a);
            \draw[->] (u4a) -- (u5);
            \draw[->] (u2) to [out=30,in=240] (u3a);
            \draw[->] (u2) to [out=60,in=210] (u3a);
        }
    \end{center}
\end{example}
There is a retraction
\[
    r \colon M_{G' \hookrightarrow G}^S \to M_{G' \hookrightarrow G}^S
\]
onto $G$, given by
\begin{align*}
    r((v,0)) &= r((v,1)) = (v,0), \\
    \text{and} \;
    r((e_G,0)) &= r((e_G,1)) = r((e_G, 0\to 1))= r(e_{e_G}) = (e_G,0)
\end{align*}
where $v \in v_G$ and $e_G \in E_G$.
However, as demonstrated by the next example, $r \colon M_{G' \hookrightarrow G}^S \to M_{G' \hookrightarrow G}^S$ is not a strong $h$-retraction, which differs from the situation for mapping cylinders of topological spaces.

\begin{example}
    Consider the line digraph $I$ given in equation~\eqref{eq:Interval} with itself as a full subquiver, then
    \[
        M_{I\hookrightarrow I}^S \boxtimes I = (I \boxtimes I) \boxtimes I.
    \]
    However, there are no digraph homomorphisms from $M_{I\hookrightarrow I}^S \boxtimes I$ to $I$ as any digraph map is required to send at least one edge to a vertex.
    Therefore, $r\colon M_{I\hookrightarrow I}^S \to M_{I\hookrightarrow I}^S$ is not a strong $h$-deformation retraction.

    More generally, it is shown \cite[Theorem 3.8]{Li2024} that a $1$-step strong $h$-homotopy can only exist when there is a directed closed walks contained in the quiver.
\end{example}

We now provide a weaker notion of homotopy containing strong $h$-homotopy as a special case, under which we will show that $H^{h,\Delta}_*$ remains invariant.

\begin{definition}\label{def:LocalStrongHHomotopy}
    Two quiver maps $f_0,f_1\colon G_1 \to G_2$ are \emph{$1$-step local strong $h$-homotopic} if there is a full subquiver $G'_1$ of $G_1$ and quiver homomorphism
    \begin{align*}
        F\colon M^S_{G'_1 \hookrightarrow G_1} \to G_2
        \;\;\; \text{such that} \;\;\;
        F \circ i_0 = f_0,
        \;
        F \circ i_1 = f_1.
    \end{align*}
    In which case we write $f_0 \simeq_1^{lSh} f_1$.
    Quiver maps $f$ and $g$ are \emph{local strong $h$-homotopic} if they are related by the equivalence relation generated by $1$-step local strong $h$-homotopyies.
    In which case we write $f \simeq^{lSh} g$.
    Quivers $G_1$, $G_2$ are \emph{locally strong $h$-homotopy equivalent} if there are quiver maps
    \begin{align*}
        f\colon G_1 \to G_2,
        \;
        h\colon G_2 \to G_1
        \;\;\; \text{such that} \;\;\;
        h \circ f \simeq^{lSh} \text{id}_{G_1},
        \;
        f \circ h \simeq^{lSh} \text{id}_{G_2}.
    \end{align*}
    A quiver is called \emph{locally strongly $h$-contactable} if it is strong homotopy equivalent to a quiver consisting of a single vertex and no edges.
\end{definition}

We note that to fully justify the definition of local strong homotopies we must check that it is compatible with composition of homotopy classes. This can be achieved similarly to the construction in Remark~\ref{rmk:OneStepHomotopyTransitivity}.

Since strong $h$-homotopy is a special case of local strong $h$-homotopy, the next theorem is a stronger version of the homotopy invariance of $H^{h,\Delta}_*$ with respect to strong $h$-homotopy from \cite[Theorem 5.17]{Li2024}.

\begin{theorem}\label{thm:LocalStrongHInvIntial}
    Let $\phi,\varphi \colon G_1 \to G_2$ be quiver homomorphisms such that $\phi \simeq^{lSh} \varphi$, then
    \[
        \phi_* = \varphi_* \colon H^{h,\Delta}_*(G_1) \to H^{h,\Delta}_*(G_2).
    \]
\end{theorem}

\begin{proof}
    By the construction of local strong homotopies, it is sufficient to consider only $1$-step local strong $h$-homotopies. Therefore, assume $\phi \simeq_1^{lSh} \varphi$. This implies that there exists an full subquiver $G'_1$ of $G_1$ and a quiver homomorphism $F\colon M^S_{G_1' \hookrightarrow G_1} \to G_2$ such that $F \circ i_0 = \phi$, $ F \circ i_1 = \varphi$.
    
    Let $f \colon \Delta^n \to G_1$ be a singular $n$-simplex homomorphism.
    When $f(i) \in V_{G_1} \setminus V_{G'_1}$ for each $i=0,\dots,n$, then by construction of $M^S_{G_1' \hookrightarrow G_1}$ and $F$ we have $\phi_{\#} \circ f = \varphi_{\#} \circ f$.
    Alternatively, when $f(i) \in V_{G'_1}$ for each $i=0,\dots,n$, then $\phi_{\#} \circ f$ and $\varphi_{\#} \circ f$ are related by the boundary of a quiver construction analogous to the classical prism operator \cite[proof of Corollary 2.11]{Hatcher2002}.
    More generally, we require the construction of a hybrid prism operator transitioning between the two cases above, which we obtain as follows. 

    Let $f_{M^S_{G'_1 \hookrightarrow G_1}}$ be the subquiver of $M^S_{G_1' \hookrightarrow G_1}$ on vertex set
    \[
        \{
        (f(i),0)
        \: | \:
        i=0,\dots,n
        \}
        \cup
        \{
        (f(i),1)
        \: | \:
        i = 0,\dots,n
        \: \text{with} \:
        f(i) \in G'_1
        \}
    \]
    and with edges
    \begin{align*}
        &(f(j\to k),0);
        \\
        & (f(j\to k),1), \;
        (f(j\to k),0\to 1), \;
        \text{or} \; (f(i),0\to 1) \;
        \text{when} \; f(i),f(j),f(k) \in G'_1;
        \\
        & e_{f(j\to k)} \; \text{when} \; 
        f(j) \in G'_1 \; \text{and} \;  f(k) \notin G'_1, \; \text{or} \;
        f(j) \notin G'_1 \; \text{and} \;  f(k) \in G'_1
    \end{align*}
    for $i=0,\dots,n$ and integers $0\leq j < k \leq n$.
    Define the indicator function $\mathcal{I}_f\colon \{0,\dots,n \}\to \{0,1\}$ by
    \[
        \mathcal{I}_f(i) =
        \begin{cases}
            1 & \text{if} \; f(i) \in G'_1
            \\
            0 & \text{otherwise}.
        \end{cases}
    \]
    Then the modified prism operator is given by
    \[
        P(f) = \sum_{i=0}^n
        \mathcal{I}_f(i)
        (-1)^i
        F\left(
        f_{M^S_{G'_1 \hookrightarrow G'_1}}|_{(f(0),0),\dots,(f(i),0),(f(i),\mathcal{I}_f(i)),\dots,(f(n),\mathcal{I}_f(n))}\right)
    \]
    where
    $f_{M^S_{G'_1 \hookrightarrow G'_1}}|_{(f(0),0),\dots,(f(i),0),(f(i),\mathcal{I}_f(i)),\dots,(f(n),\mathcal{I}_f(n))}$ indicates the singular simplex
    $g_i\colon \Delta^{n+1} \to M^S_{G'_1 \hookrightarrow G'}$
    given on vertices by
    \[
        g_i(t) = 
        \begin{cases}
            (f(t),0) & \text{if} \: t \leq i \\
            (f(t),\mathcal{I}_f(t)) & \text{otherwise}
        \end{cases}
    \]
    for $t=0,\dots,n+1$ and on edges by
    \begin{equation*}
        g_i(j\to k) =
        \begin{cases}
            (f(j\to k),0) & \text{if} \:  i \geq k \: \text{or} \: \mathcal{I}_f(j) = \mathcal{I}_f(k) = 0 \\
            (f(j\to k),1) & \text{if} \: i < j \: \text{and} \: \mathcal{I}_f(j) = \mathcal{I}_f(k) = 1 \\
            (f(j\to k), 0\to 1) & \text{if} \: j \leq i < k \: \text{and} \: \mathcal{I}_f(j) = \mathcal{I}_f(k) = 1  \\
            e_{f(j\to k)} & \text{otherwise}
        \end{cases}
    \end{equation*}
    for integers $0\leq j < k \leq n+1$.

    When $\mathcal{I}_{f}(i) = 0$ for some singular $n$-simplex $f\colon \Delta^n \to G$ and each $i=0,\dots,n$, then $P(f) = 0$. In this case we have $\partial_{n+1}P(f) = P\partial_{n}(f) = 0$ and $\phi_{\#}(f) = \varphi_{\#}(f)$.
    Otherwise, define $i_{\text{min}}$ and $i_{\text{max}}$ to be the minimum and maximum value among $0,\dots,n$ such that $\mathcal{I}_{f}(i_{\text{min}}) \neq 0$ and $\mathcal{I}_{f}(i_{\text{max}}) \neq 0$, respectively.
    Making use of this notation, we obtain that
    \begin{align*}
        \partial_{n+1} P(f) = &
        \sum_{0\leq j \leq i \leq n}
        \mathcal{I}_f(i)
        (-1)^i (-1)^j
        F\left(f_{M^S_{G'_1 \hookrightarrow G'_1}}|_{(f(0),0),\dots,\widehat{(f(j),0)},\dots,(f(i),0),(f(i),\mathcal{I}_f(i)),\dots,(f(n),\mathcal{I}_f(n))}\right)
        \\ + &
        \sum_{0\leq i \leq j \leq n}
        \mathcal{I}_f(i)
        (-1)^i (-1)^{j+1}
        F\left(f_{M^S_{G'_1 \hookrightarrow G'_1}}|_{(f(0),0),\dots,(f(i),0),(f(i),\mathcal{I}_f(i)),\dots,,\widehat{(f(j),\mathcal{I}_f(j))},\dots,(f(n),\mathcal{I}_f(n))}\right)
        \\ = &
        \mathcal{I}_f(i_{\text{min}})
        F\left(f_{M^S_{G'_1 \hookrightarrow G'_1}}|_{(f(0),0),\dots,\widehat{(f(i_{\text{min}}),0)},(f(i_{\text{min}}),\mathcal{I}_f(i_{\text{min}})),\dots,(f(n),\mathcal{I}_f(n))}\right)
        \\ + &
        \sum_{0\leq j < i \leq n}
        \mathcal{I}_f(i)
        (-1)^i (-1)^j
        F\left(f_{M^S_{G'_1 \hookrightarrow G'_1}}|_{(f(0),0),\dots,\widehat{(f(j),0)},\dots,(f(i),0),(f(i),\mathcal{I}_f(i)),\dots,(f(n),\mathcal{I}_f(n))}\right)
        \\ + &
        \sum_{0\leq i < j \leq n}
        \mathcal{I}_f(i)
        (-1)^{i-1} (-1)^{j}
        F\left(f_{M^S_{G'_1 \hookrightarrow G'_1}}|_{(f(0),0),\dots,(f(i),0),(f(i),\mathcal{I}_f(i)),\dots,\widehat{(f(j),\mathcal{I}_f(j))},\dots,(f(n),\mathcal{I}_f(n))}\right)
        \\ - &
        \mathcal{I}_f(i_{\text{max}})
        F\left(f_{M^S_{G'_1 \hookrightarrow G'_1}}|_{(f(0),0),\dots,(f(i_{\text{max}}),0),\widehat{(f(i_{\text{max}}),\mathcal{I}_f(i_{\text{max}}))},\dots,(f(n),\mathcal{I}_f(n))}\right)
        \\ = &
        \varphi_{\#}(f) - P \partial_{n} - \phi_{\#}(f)
        .
    \end{align*}
    Hence,
    \[
        \partial_{n+1} P + P \partial_{n} = \varphi_{\#} - \phi_{\#}
    \]
    making $P$ a chain homotopy between $\phi_\#$ and $\varphi_\#$.
    Therefore, $\phi_{\#} $ and $\varphi_{\#}$
    induce the same graded module homomorphism on homology as required.
\end{proof}

The next example demonstrates a further useful context in which $H^{h,\Delta}_*$ is invariant that is not captured by considering homomorphisms from a strong mapping cylinders alone.  

\begin{example}
    Consider the line digraph $I$ given in equation~\eqref{eq:Interval}. In particular, $H^{\Delta,h}_*(I)$ is acyclic.
    There are three possible full subquivers of $I$; $I$ itself, $I_0$ consisting of only vertex $0$, and $I_1$ consisting of only vertex $1$. In all three cases, there does not exist a digraph homomorphism
    \[
        h\colon M_{I\hookrightarrow I}^S \to I, \;\;\; h\colon M_{I_0\hookrightarrow I}^S \to I,\;\;\;\; \text{or} \;\;\; h\colon M_{I_1\hookrightarrow I}^S \to I.
    \]
    However, if we form a quiver $I'$ by placing an additional loop at vertex $1$ of $I$ and consider the full subquiver $I'_1$ consisting of vertex $1$ along with its loop, there are digraph homomorphisms
    \[
        h\colon M_{I'_1\hookrightarrow I'}^S \to I'.
    \]
    Moreover, $H^{\Delta,h}_*(I')$ remains acyclic.
    The homomorphisms $h$ can be chosen such that $i_0 \circ h = \text{id}_{I'}$, and when factored though $i_1$ the image consists precisely of $I'_1$. Therefore, $I'_1$ is a strong $h$-deformation retraction of $I'$.
\end{example}

The example above motivates the following definitions.

\begin{definition}\label{def:WeakLocalStrongHHomotopy}
    A loop $l$ at $v$ is called \emph{degenerate} if it is the unique loop at $v$ and for any $u\in V_G$ such that $u\neq v$ with an $e \in E_G$ such that
        \[
            s(e) = u 
            \;\: \text{and} \;\:
            t(e) = v
            \;\;\; \text{or} \;\;\;
            s(e) = v
            \;\: \text{and} \;\:
            t(e) = u
        \]
        the edge $e$ is not a multiple edge, and either
        \begin{enumerate}[(i)]
            \item 
            $u$ has a loop or
            \item
            $e$ is not a double edge.
        \end{enumerate}
        A subquiver $G'$ obtained from $G$ by removing one or more degenerate loops from $E_G$ is called a \emph{loop contraction} of $G$.
\end{definition}

\begin{definition}\label{def:QuiverLocalStrongHomotopy}
    A pair of quivers $G_1$, $G_2$ are \emph{$1$-step  weak local strong $h$-homotopy equivalent} if 
    \begin{enumerate}[(1)]
        \item 
        there are quiver homomorphisms $f\colon G_1 \to G_2$, $g\colon G_2 \to G_1$ such that
        $g \circ f \simeq_1^{lSh} \text{id}_{G_1}$, $f \circ g \simeq_1^{lSh} \text{id}_{G_2}$;
        \item 
        $G_1$ is a loop contraction of $G_2$;
        \item 
        or $G_2$ is a loop contraction of $G_1$.
    \end{enumerate}
    In which case we write $G_1 \simeq^{wlSh}_1 G_2$.
    Quivers $G_1$ and $G_2$ are \emph{weak local strong $h$-homotopic} if they are related by the equivalence relation generated by (1), (2) and (3) above, in which case we write $G_1 \simeq^{wlSh} G_2$.
    A quiver is called \emph{weak local strongly $h$-contactable} if it is weak local strong $h$-homotopy equivalent to a quiver on one vertex with no edges.
\end{definition}

We will prove at the end of Section~\ref{sec:PartialFlagComplex} that $H_*^{\Delta}$ is invariant under weak local strong $h$-homotopy of quivers. To complete the section we consider some partial results in this direction. In order make these statements, we introduce the following terminology.

Let $G$ be a quiver and $l$ a loop in $G$ and recall that we assume all coefficients lie in a commutative ring $R$. Then consider the $C_*^{\Delta,h}(G)$ submodules,
\begin{align*}
    C_*^{\Delta,h,l}(G) =
    R[\{ 
    & f \colon \Delta^n \to G 
    \: | \:
    f \: \text{is a homomorphism}, \:
    f(a\to c) = l \:
    \text{for some} \: 0 \leq a < c \leq n,
    \\ &
    \text{and} \:
    \forall 
    \: 0 \leq a < b < c \leq n 
    \: \text{with} \:
    f(a\to c) = l \:
    \text{we have} \:
    f(b) = f(a)
    \}]
\end{align*}
and 
\begin{align*}
    C_*^{\Delta,h,ld}(G) =
    R[\{ &
    f \colon \Delta^n \to G 
    \: | \:
    f \: \text{is a homomorphism}, \: \text{and}
    \\ & \exists \:
    0 \leq a < b < c \leq n \:
    \text{with} \:
    f(a\to c) = l
    \: \text{and} \:
    f(b) \neq f(a)
    \}].
\end{align*}
Any singular simplex containing the loop $l$ in its image lies in $C_*^{\Delta,h,l}(G)$ or $C_*^{\Delta,h,ld}(G)$ and $C_*^{\Delta,h,l}(G) \cap C_*^{\Delta,h,ld}(G) = \emptyset$. Crucially, when $l$ is degenerate singular simplices containing the loop $l$ in their image are precisely the generators of $C_*^{\Delta,h}(G)$ that no longer exist after a loop contraction removing $l$.
The key distinction between singular simplicies $f\colon \Delta^n \to G$ generating $C_*^{\Delta,h,l}(G)$ and $C_*^{\Delta,h,ld}(G)$ is provided by the presence of a $G$ subquiver of the form
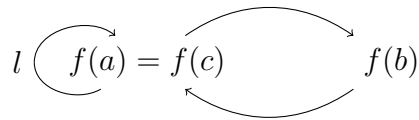
\begin{center}
    \tikz {
        \node (a) at (0,0) {$f(a)=f(c)$};
        \node (b) at (3.25,0) {$f(b)$};
        \node (d) at (-1.5,-0.145) {};
        \node (c) at (-1.5,0.145) {};
        \node (l) at (-1.725,0) {$l$};
        \draw[->] (a) to [out=35,in=145] (b);
        \draw[->] (b) to [out=215,in=325] (a);
        \draw[-] (a) to [out=210,in=270] (c);
        \draw[->] (d) to [out=90,in=140] (a);
        }
\end{center}
in the image of $f$ when $f \in C_*^{\Delta,h,ld}(G)$. Such an image of a singular simplex homomorphism can only occur when $l$ is incident to a double edge and in this case $H_*^{\Delta,h}(G)$ can be altered after the removal of $l$ if there is no loop at $f(b)$, as made explicit in Example~\ref{ex:DoubleCone}. The case of $C_*^{\Delta,h,l}(G)$ alone is easier handle. To demonstrate this we make use of the following operations.

Let $f \colon \Delta^n \to G$ be a singular $n$-simplex homomorphism in $C_*^{\Delta,h}(G)$ such that $f(i)$ has a loop $l$ for some $i=0,\dots,n$. Then define singular $(n+1)$-simplex $s_{i,l}(f)$ in $C_{n+1}^{\Delta,h}(G)$, given on vertices by
\[
    s_{i,l}(f)(t) = 
    \begin{cases}
        f(t) &
        \text{if} \: t \leq i \\
        f(i) &
        \text{if} \: t = i + 1 \\
        f(t-1) &
        \text{otherwise}
    \end{cases}
\]
for $t=0,\dots,n+1$ and on edges by
\[
    s_{i,l}(f)(j \to k) = 
    \begin{cases}
        f(j \to k) &
        \text{if} \: k \leq i \\
        f(j \to i) &
        \text{if} \: j< i, \: k = i + 1 \\
        l &
        \text{if} \: j = i, \: k = i + 1 \\
        f(i \to k-1) &
        \text{if} \: j = i+1 \\
        f(j-1 \to k - 1) &
        \text{otherwise}
    \end{cases}
\]
for integers $0 \leq j < k \leq n+1$.

Analogously to the construction of simplicial sets, we can think of $s_{i,l}(f)$ as a degeneracy operator on a vertex with a loop. However, in the present case we require that a singular simplex homomorphism contains a loop in its image, since a homomorphism cannot map an edge to a vertex. Nevertheless, given a singular simplex homomorphism $\Delta^n \colon G$ such that $f(i)$ has a loop $l$ for some $i = 0,\dots,n$, $s_{i,l}(f)$ and the face operators $\delta_j$ satisfy the usual relations
\begin{equation}\label{eq:LoopExtensionRelations}
    s_{i,l}(f) \circ \delta_j = 
    \begin{cases}
        s_{i-1,l}(f\circ \delta_j)
        & \text{if} \: j < i \\
        f
        & \text{if} \: i = j \: \text{or} \: i + 1 = j \\
        s_{i,l}(f\circ \delta_{j-1})
        & \text{if} \: j > i + 1
    \end{cases}
\end{equation}
for each $j = 0,\dots,n+1$.

\begin{lemma}\label{lem:FirstLoopCase}
    Let $G$ be a quiver and $l$ the unique loop in $G$ at vertex $v_l\in V_G$ with no multiple edges indecent to $v_l$. Then $C_*^{\Delta,h,l}(G)$ is an acyclic sub-chain complex of $C_*^{\Delta,h}(G)$. 
\end{lemma}

\begin{proof}
    Suppose that $f$ is a singular $n$-simplex homomorphism in $C_n^{\Delta,h,l}(G)$. Then there are unique integers $0 \leq a < c \leq n$ maximizing $c-a$ such that $f(a \to c) = l$.
    In this case, for any integer $0 < i< a$ or $c < i\leq n$, $f \circ \delta_{i} \in C_{n-1}^{\Delta,h,l}(G)$.
    
    As $l$ is unique and $f \in C_n^{\Delta,h,l}(G)$, for any $0\leq a < b < c \leq n$ we have $f(b)=f(a)=f(c)$ and $f(a\to b)=f(b \to c) = l$. 
    Hence, when $c-a > 1$ we also have that $f \circ \delta_{i} \in C_{n-1}^{\Delta,h,l}(G)$ for $i=a,\dots,c$.
    If $c-a = 1$, then as there are no multiple edges incident to $v_l$ $f \circ \delta_{a}-f \circ \delta_{c}= 0 \in C_{n-1}^{\Delta,h,l}(G)$.
    Therefore, $C_*^{\Delta,h,l}(G)$ is a sub-chain subcomplex of $C_*^{\Delta,h}(G)$ and it remains to show that $C_*^{\Delta,h,l}(G)$ is acyclic.

    To this end, define $h_n \colon C_n^{\Delta,h,l}(G) \to C_{n+1}^{\Delta,h,l}(G)$ by linearly extending $h_n(f) = (-1)^i s_{i,l}(f)$ where $f \colon \Delta^n \to G$ is a singular $n$-simplex homomorphism in $C_n^{\Delta,h,l}(G)$ and $i$ is the least integer from $0,\dots,n$ such that $f(i) = v_l$. 
    Applying equation~\eqref{eq:LoopExtensionRelations}, for each singular $n$-simplex homomorphism $f \in C_n^{\Delta,h,l}(G)$ we have
    \[
        (\partial_{n+1} \circ h_n)(f)
        = 
        (-1)^i\sum_{j=0}^{n+1}(-1)^j s_{i,l}(f) \circ \delta_j
        =
        (-1)^i \sum_{0\leq j < i} (-1)^j s_{i-1,l}(f \circ \delta_j)
        +
        (-1)^i \sum_{i+1 < j \leq n+1} (-1)^j s_{i,l}(f \circ \delta_{j-1})
    \]
    and
    \begin{align*}
        (h_{n-1} \circ \partial_{n})(f)
        = &
        \sum_{j=0}^n(-1)^j h_{n-1}(f \circ \delta_j)
        \\ = &
        (-1)^i h_{n-1}(f \circ \delta_i)
        +
        (-1)^{i-1} \sum_{0\leq j <i} (-1)^js_{i-1,l}(f\circ \delta_j)
        +
        (-1)^i \sum_{i < j \leq n} (-1)^j s_{i,l}(f\circ \delta_j).
    \end{align*}
    Since $f \in C_*^{\Delta,h,l}(G)$ it is necessarily the case that $f(i \to i+1) = l$ and $f(i+1) = f(i) = v_l$. This implies that $h_{n-1}(f \circ \delta_i) = (-1)^i s_{i,l}(f \circ \delta_i) = (-1)^i f$.
    Therefore, $\partial_{n+1} \circ h_n + h_{n} \circ \partial_{n} = \text{id}_{C_*^{\Delta,h,l}(G)}$.
    Which implies that $C_*^{\Delta,h,l}(G)$ is chain homotopic to the zero complex.
    Since $C_0^{\Delta,h,l}(G) = 0$, any augmented chain complex is also chain homotopic to the zero complex and $C_*^{\Delta,h,l}(G)$ is acyclic, as required.
\end{proof}

The next lemma is included here due to the fact that the lemma following it makes use of a related deformation retraction.

\begin{lemma}\label{lem:ExtenededSimplexStrongContraction}
    For any digraph $G$, if there exists a digraph map $m\colon \Delta^n \to G$ that is surjective on vertices, then $G$ is strongly contractable.
\end{lemma}

\begin{proof}
    Let $r \colon G \to G$ be the unique map sending all vertices and edges of $G$ to $m(n)$.
    Since $m\colon \Delta^n \to G$ is a simplex,  $(v,m(n))\in E_G$ for all $v \in V_G \setminus \{m(n)\}$.
    A $1$-step strong homotopy between $\text{id}_G$ and $r$ is given by $F\colon G \boxtimes I \to G$, with
    \begin{align*}
        & F(v,0) = v, \\
        & F(e,0) = e,  \\
        & F(v,1) = F(e,1) = m(n),  \\
        & F(v, 0 \to 1) = 
        \begin{cases}
            m(n) & \text{if} \: v = m(n) \\
            (v, m(n)) & \text{otherwise}
        \end{cases}
        \\
        \text{and} \;
        & F(e, 0\to 1) = 
        \begin{cases}
            m(n) & \text{if} \: s_G(e) = m(n) \\
            (s(e), m(n)) & \text{otherwise}
        \end{cases}
    \end{align*}
    where $v\in V_G$ and $e \in E_G$.
    Therefore, $r$ is a strong deformation retract of $G$ onto vertex $m(n)$ and $G$ is strongly contactable as required.
\end{proof} 

The next lemma determines an initial special case when $C_*^{\Delta,h}(G)$ acyclic under certain assumptions.

\begin{lemma}\label{lem:MaxLoopSimplexAcyclic}
    Let $G$ be a quiver and $h\colon \Delta^n \to G$ a homomorphism surjective on vertices such that $h(n)$ has a loop.
    Then $C_*^{\Delta,h}(G)$ is acyclic.
\end{lemma}

\begin{proof}
    With the modification that every edge that would have been sent to the vertex $m(n)$ being instead sent to the loop at $h(n)$, the retraction $r$ and strong homotopy given in the proof of Lemma~\ref{lem:ExtenededSimplexStrongContraction} provide a strong $h$-deformation retraction to $G$ subquiver $G'$ consisting of vertex $h(n)$ and a loop at $h(n)$.
    Since $C_*^{\Delta,h}(G')$ is acyclic, the strong homotopy invariance of $H_*^{\Delta,h}$ from Theorem~\ref{thm:PreviousHomotopyInv} now implies that $C_*^{\Delta,h}(G)$ acyclic.
\end{proof}

We can now give a variation of the previous lemma under different assumptions.

\begin{lemma}\label{lem:ExtenededSimplexStrongLocalHContraction}
    Let $G$ be a quiver with no multiple edges, such that any double edge has a loop, there exists a digraph homomorphism $h \colon \Delta^n \to G$ surjective on vertices, and for every edge of $e \in E_G$ there exists some digraph homomorphism $h_e\colon \Delta^n \to G$ surjective on vertices of $G$ with $e$ lying in its image.
    Then $C_*^{\Delta,h}(G)$ is acyclic.
\end{lemma}

\begin{proof}
    When $h(n)$ has a loop, $G$ is acyclic by Lemma~\ref{lem:MaxLoopSimplexAcyclic}.
    If $h(n)$ has no loop and is not incident to a double edge in $G$, then let $G'$ be the quiver obtained from $G$ by adding an additional loop $l$ at $h(n)$.
    In this case, as $h$ is surjective on vertices and there are no double edges indecent to $h(n)$, all edges incident to $h(n)$ in $G'$ have target vertex $h(n)$.
    This implies that for any integer $m \geq 0$ and singular $m$-simplex homomorphism $f \colon \Delta^m \to G'$, $f(i) \neq h(n)$ for each $i=0,\dots,m-1$. Hence, $C_*^{\Delta,h,ld}(G') = 0$ which implies $C_*^{\Delta,h}(G')= C_*^{\Delta,h}(G) + C_*^{\Delta,h,l}(G')$ as a chain complexes. Since $G$, hence, $G'$ contains no multiple edges, Lemma~\ref{lem:FirstLoopCase} implies that $C_*^{\Delta,h,l}(G')$ is an acyclic summand of $C_*^{\Delta,h}(G')$. Therefore, $H_*^{\Delta,h}(G) = H_*^{\Delta,h}(G')$ and $C_*^{\Delta,h}(G)$ is acyclic as $C_*^{\Delta,h}(G')$ is acyclic by Lemma~\ref{lem:MaxLoopSimplexAcyclic}.

    The only remaining possibility is that $h(n)$ has no loop and is incident to a double edge $d$ such that $s(d) = h(n)$ and $t(d) \neq h(n)$. Applying the conditions in the statement of the lemma, there exists a $h_d \colon \Delta^n \to G$ surjective on vertices of $G$ with $d$ in its image. 
    If $h_d(n) = h(n)$ then as $d$ lies in the image of $h_d$, there is an $i=0,\dots,n-1$ such that $h_d(i) = s(d) = h(n)$. This implies that $h_d(i \to n)$ is a loop at $h(n)$, which contradicts the present assumption that $h(n)$ has no loop. 
    Therefore, $h_d(n) \neq h(n)$. By construction of singular simplex homomorphisms $h$ and $h_d$ there must be an edge from every element of $V_G \setminus \{h(n)\}$ to $h(n)$ and from every from every element of $V_G \setminus \{h_d(n)\}$ to $h_d(n)$.
    This implies there is a double edge in $G$ between $h_d(n)$ and $h(n)$. Using the conditions of the lemma, since $h(n)$ does not have loop $h_d(n)$ must have a loop. Hence, $C_*^{\Delta,h}(G)$ is acyclic by Lemma~\ref{lem:MaxLoopSimplexAcyclic} with the singular simplex homomorphism $h_d$.
\end{proof}

\section{Computation of \texorpdfstring{$H^{\Delta,m}_*$}{singular quiver homology}}\label{sec:ComputingMappingHomology}

In this section, we provide an object-wise map from a quiver to a suitably simpler simplicial complex called the reduced directed flag complex, whose simplicial homology coincides with $H^{\Delta,m}_*$. In particular, this construction extends to a functor natural with recept to the isomorphism between homologies. The naturality of our result implies that persistent $H^{\Delta,m}_*$ can be computed more efficiently using the persistent simplicial homology of the reduced directed flag complex.

Throughout the section assume that $G$ is a quiver and $n\geq 0$ is an integer, unless otherwise stated. Note also that some results in the section require the axiom of choice if the quiver $G$ contains an infinite number of edges.

\subsection{The reduced digraph of a quiver}\label{sec:ReducedDigraph}

In this subsection, we show that the presence of loops and multiple edges does not effect the structure of $H^{\Delta,m}_*(G)$. More precisely, we show that the removal of loops and multiple edges realises a strong homotopy equivalence between any quiver and a digraph. The strong homotopy invariance of $H_*^{\Delta,m}$ shown in \cite{Li2024}, then provides us with our desired result.

\begin{definition}
    Let $G$ be a quiver. Define the \emph{reduced digraph} $\bar{\mathcal{R}}(G)$ of the quiver $G$ to have vertices $V_{\bar{\mathcal{R}}(G)}=V_G$ and edges
    \[
        E_{\bar{\mathcal{R}}(G)} =
        \{ \bar{\mathcal{R}}(G,u,v)
        \: | \:
        \exists \:
        e \in E_G \:
        \text{with} \:
        s(e) = u, \:
        t(e) =v \:
        \text{and} \:
        u \neq v
        \}
    \]
    where
    \[
        s(\bar{\mathcal{R}}(G,u,v)) = u
        \;\;\; \text{and} \;\;\;
        t(\bar{\mathcal{R}}(G,u,v)) = v.
    \]
\end{definition}
The reduced digraph construction extends to a functor $\bar{\mathcal{R}}\colon \textbf{Quiv}_m \to \textbf{Quiv}_m$ by setting
\begin{equation}\label{eq:ReducedDigraphFunctor}
        \bar{\mathcal{R}}(m)(v) = m(v)
        \;\;\; \text{and} \;\;\;
        \bar{\mathcal{R}}(m)(\bar{\mathcal{R}}(G,s(e),t(e))) = (\bar{\mathcal{R}}(G',s(m(e)),t(m(e))))
\end{equation}
for quiver map $m \colon G \to G'$, $v\in V_G$ and $e\in E_G$.

A weaker version of the following proposition for quiver homotopies with respect to the box product rather than strong homotopies can be found in course notes of Muranov \cite[Lecture 7]{Muranov2024}, and is expected to appear in a textbook on the subject of digraph homotopy and homology.  

\begin{proposition}\label{prop:StrongMultiEdgeContraction}
    Quivers $G$ and $\bar{\mathcal{R}}(G)$ are strong homotopy equivalent.
\end{proposition}

\begin{proof}
    Define $f\colon G \to \bar{\mathcal{R}}(G)$ by
    \[
        f(v) = v
        \;\;\; \text{and} \;\;\;
        f(e) = 
        \begin{cases}
            s(e) & \text{if} \; s(e) = t(e)
            \\
            \bar{\mathcal{R}}(G,s(e),t(e)) & \text{otherwise} 
        \end{cases}
    \]
    for any $v \in V_G$ and $e \in E_G$.
    Given distinct $u,v\in E_G$ such that there is an $e\in E_G$ with $s(e)=u$ and $t(e)=v$, choose some edge
    \[
        e_{u,v}\in \{ e \in E_G \: | \: s(e) = u \: \text{and} \: t(e) = v \}.
    \]
    Define $g\colon \bar{\mathcal{R}}(G) \to G$ by
    \[
        g(v) = v
        \;\;\; \text{and} \;\;\;
        g(\bar{\mathcal{R}}(G,u,v)) = e_{u,v}.
    \]
    By construction of $\bar{\mathcal{R}}(G)$, both $f$ and $g$ are well defined and $f \circ g = \text{id}_{\bar{\mathcal{R}}(G)}$.
    Moreover, $g\circ f \simeq_1^S \text{id}_G$ as we have well defined $F\colon G \boxtimes I \to G$ given on edges by
    \begin{align*}
        &F((v,0)) = F((v,1)) = F((v, 0\to 1)) = v, \\
        &F((e,0)) = 
        \begin{cases}
            s(e) & \text{if} \; s(e) = t(e)
            \\
             e_{u,v} & \text{otherwise,} 
        \end{cases}
        \\
        \text{and} \;
        &F((e, 1)) = F((e, 0\to 1)) = e
    \end{align*}
    for all $v\in V_G$ and $e\in E_G$.
\end{proof}

In fact we have proved that $\bar{\mathcal{R}}(G)$ is a strong deformation retraction of $G$.
Note however, that the proof of the lemma above does not hold in the category of quivers and quiver homomorphism $\textbf{Quiv}_h$, as we are always required to send at least one edge of the form $(v, 0\to 1)\in E_{G \boxtimes I}$ for $v\in V_G$ to a vertex.

\begin{remark}\label{rmk:NonStrongRedcuedquiverCase}
    As stated prior to Proposition~\ref{prop:StrongMultiEdgeContraction}, the cartesian or box product $G \square G'$ of quivers $G$ and $G'$ can be obtained similarly to the strong box product in Definition~\ref{def:StrongBoxProduct} by restricting only to edges $E_{G \square G'} = E_G \times V_{G'} \cup V_G \times E_{G'}$. In precisely the same way as for strong homotopies in Definition~\ref{def:StrongHomtopoic}, we obtain a notion of homotopy of digraphs and quivers using $\square$ in place of $\boxtimes$. In particular, Proposition~\ref{prop:StrongMultiEdgeContraction} is equally valid for quivers homotopies, as less conditions are required to be checked in the poof. The path homology of a digraph is known to be a homotopy invariant of digraphs \cite{Grigoryan2014}. Therefore, any extension of path homology that is a homotopy invariant of quivers when applied to a quiver $G$ would coincide with the path homology of the reduced digraph $\bar{\mathcal{R}}(G)$, which is a digraph by construction. 
\end{remark}

Proposition~\ref{prop:StrongMultiEdgeContraction}, has the following immediate consequence. 

\begin{corollary}\label{cor:SimplexMapReducedOnly}
    Let $G$ be a quiver, then
    \[
        H_*^{\Delta,m}(G) = H_*^{\Delta,m}(\bar{\mathcal{R}}(G))
    \]
    naturally with respect to induced maps.
\end{corollary}

\begin{proof}
    It is shown in \cite[Theorem 5.6]{Li2024} that $H_*^{\Delta,m}$ is a strong homotopy invariant. Hence the isomorphism of homology follows from Proposition~\ref{prop:StrongMultiEdgeContraction}. The naturality of this isomorphism can be obtained directly from the functorial construction of $\bar{\mathcal{R}}$ in equation~\eqref{eq:ReducedDigraphFunctor} and induced maps of $H_*^{\Delta,m}$ given in equation~\eqref{eq:InducedMapHomology}.
\end{proof}

\subsection{The reduced directed flag complex}\label{sec:RducedDirectedFlag}

Motivated by computation, we now extend the construction of the reduced digraph functor $\bar{\mathcal{R}}$ from the previous subsection and construct for any quiver a simplicial complex with $H_*^{\Delta,m}$ isomorphic to its simplicial homology. This construction is functorial, providing a functor we call the reduced direct flag complex $\bar{\mathcal{F}}\colon \textbf{Quiv}_m \to \textbf{ASim}$. Furthermore, $\bar{\mathcal{F}}$ is natural with respect to the isomorphisms between homologies.

\begin{definition}\label{def:ReducedDirectedFlagComplex}
    Let $G$ be a quiver. Then the \emph{reduced directed flag complex}, $\bar{\mathcal{F}}(G)$ is the abstract simplicial complex on vertex set $V_G$, with an $n$-simplex on the set of vertices $v_0,\dots,v_n \in V_G$ if and only if these is an inclusion $f \colon \Delta^n \to G$ such that $v_0,\dots,v_n$ lie in the image of $f$.
\end{definition}

We note that for digraphs the reduced directed flag complex also appears in \cite{milicevic2024} where it is called the directed Vietoris-Rips complex. In this case, the reduced directed flag complex is considered in a different theoretical context related to the homology of digraphs as pseudotopological spaces. However, analogously to the direction of this section, the directed Vietoris-Rips complex was shown to have homology coinciding with the homology of a digraph as a certain pseudotopological spaces.

The following example demonstrates the difference between the two notions of flag complex that have been introduced so far in this work.

\begin{example}\label{ex:FlagComplexDifferences}
    The reduced directed flag complex $\bar{\mathcal{F}}(G)$ is not the same as the directed flag complex $\mathcal{F}(G)$ (Definition~\ref{def:DirectedFlag}) whenever multiple or double edges are present. For example, if $G$ is a quiver that contains two vertices and precisely one double edge or multiple edge then $\text{F(x)}(G)$ is a circle while $\bar{\mathcal{F}}(G)$ is an interval.
    
    Moreover, neither $\bar{\mathcal{F}}(G)$ or $\mathcal{F}(G)$ in general coincide with the classical flag complex on the undirected graph containing $G$.
    For example, the directed cycle digraph $C$ with $V_C=\{u,v,w\}$ and edges
    
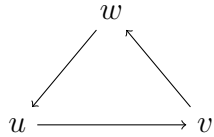
\begin{center}
        \tikz {
            \node (u) at (0,0) {$u$};
            \node (v) at (2.5,0) {$v$};
            \node (w) at (1.25,1.5) {$w$};
            \draw[->] (u) -- (v);
            \draw[->] (v) -- (w);
            \draw[->] (w) -- (u);
            }
    \end{center}
    provides $\bar{\mathcal{F}}(C)$ and $\mathcal{F}(C)$ that contain no $2$-simplices, while the classical flag complex contains a $2$-simplex.
\end{example}

Given a map of quivers $\phi \colon G_1\to G_2$ we obtain an induced map $\bar{\mathcal{F}}(\phi)\colon \bar{\mathcal{F}}(G_1) \to \bar{\mathcal{F}}(G_2)$ given by
\begin{equation}\label{eq:FlagInducedMap}
    \bar{\mathcal{F}}(\phi)(\{v_0,\dots,v_n\})
    =
    \{\phi(v_0),\dots,\phi(v_n)\}
\end{equation}
for each simplex $\{v_0,\dots,v_n\}$ in $\bar{\mathcal{F}}(G_1)$.
The induced maps make
$\bar{\mathcal{F}}\colon \textbf{Quiv}_m \to \textbf{ASim}$ a functor, as for any quiver maps $\phi \colon G_1 \to G_2$ and $\varphi \colon G_2 \to G_3$ we have
\begin{equation}\label{eq:FlagFunctoriality}
    \bar{\mathcal{F}}(\varphi \circ \phi)(\{v_0,\dots,v_n\})
    =
    \{\varphi\circ\phi(v_0),\dots,\varphi \circ \phi(v_n)\}
    =
    (\bar{\mathcal{F}}(\varphi) \circ \bar{\mathcal{F}}(\phi))(\{v_0,\dots,v_n\})
\end{equation}
for each simplex $\{v_0,\dots,v_n\}$ in $\bar{\mathcal{F}}(G_1)$.

Every map of abstract simplicial complexes is by construction provided by a function on its vertices, while a map of quivers is determined also by a function on its edges. This implies that $\bar{\mathcal{F}}$ cannot be a faithful functor. However, $\bar{\mathcal{F}}$ is faithful when restricted to digraphs, in which case there is a unique arrow in either direction between any pair of vertices. Furthermore, $\bar{\mathcal{F}}$ is not a full functor even when restricted to digraphs. Consider for example the interval digraph $I$ with a single arrow between two vertices. Then a digraph map switching the order of the vertices does not exist. However, a simplicial map switching the vertices of $\bar{\mathcal{F}}(I)$ is well defined.

The next example demonstrates that $\bar{\mathcal{F}}$ is not surjective on objects, though $\bar{\mathcal{F}}$ can realise any simplicial complex up to homotopy type. 

\begin{example}
    Consider the abstract simplicial complex $(V,S)$ on vertex set $V=\{v_0,v_1,v_2,v_3\}$ with
    \[
        S = \{ \{v_0,v_1\} \{v_0,v_2\}, \{v_0,v_3\} \{v_1,v_2\}, \{v_1,v_3\}, \{v_2,v_3\} \}
    \]
    That is the $1$-skeleton of a $2$-simplex. We now attempt to construct a quiver $G$ such that $\bar{\mathcal{F}}$ is $(V,S)$.
    
    Since $(V,S)$ contains no $2$ simplices, all triples of vertices contained in $V_G=V$ must be directed cycles (see the second part of Example~\ref{ex:FlagComplexDifferences}), otherwise there is an inclusion of $\Delta^2$ into $G$. However, once any two triples of vertices in $V_G$ are assigned arrows in the form of directed cycles the remaining two tipples are forced to be non-cyclic.
    Moreover, any additional edges added to $G$ only increases the number of possible inclusions $\colon \Delta^2 \to G$. Therefore, there does not exist a quiver $G$ such that $\bar{\mathcal{F}}(G) = (V,S)$.

    However, given an arbitrary abstract simplicial complex $(V,S)$, we can construct a quiver $G$ such that $\bar{\mathcal{F}}(G)$ and $(V,S)$ have the same homotopy type as follows.
    Using a simplicial construction similar to the cubical one from \cite{Grigoryan2014}, we will consider the Barycentric subdivision $B$ of $(V,S)$ and form a digraph $G_B$ such that $\bar{\mathcal{F}}(G_B) = B$.
    Recall that the Barycentric subdivision $B$ of $(V,S)$ is the flag complex of the graphs whose vertices are $S$ and with an edge between $s_1,s_2 \in S$ if and only if $s_1 \subset s_2$ or $s_2 \subset s_1$.
    In particular, $B$ has the same homotopy type as $(V,S)$.
    
    The digraph $G_B$ is obtained by setting $V_{G_B}$ as the vertices of $B$ and $E_{G_B}$ to be in bijection with the edges of $B$ assigning direction in each case from the smaller element of $S$ to the larger element of $S$. In this case, inclusions $f \colon \Delta^n \to G_B$ are in one to one correspondence with $n$-simplices of $B$. Where the bijection is given by assigning to $f \colon \Delta^n \to G_B$ the simplex $i(n)$. Therefore, $\bar{\mathcal{F}}(G_B) = B$ by definition.
\end{example}

The following theorem and subsequent corollary demonstrate that the computation of $H_*^{\Delta,m}(G)$ can be reduced to computing the homology of the abstract simplicial complex $\bar{\mathcal{F}}(\bar{\mathcal{R}}(G))$. In particular, for a finite quiver the size of a basis of $C_*(\bar{\mathcal{F}}(\bar{\mathcal{R}}(G)))$ is in general far smaller than for $C_*^{\Delta,m}(G)$, leading to a considerably more efficient algorithm in Appendix~\ref{sec:AlgReducedDirectedFlagComplex}.

\begin{theorem}\label{thm:SimplicalGraphHomologies}
    Let $G$ be a quiver, then
    \[
        H_*^{\Delta,m}(G) = H_*(\bar{\mathcal{F}}(G)).
    \]
    Moreover, the above isomorphism on homology is natural with respect to $\bar{\mathcal{F}}$.
\end{theorem}

\begin{proof}
    Fix a total order on $V_G$. In particular, the homology of the abstract simplicial complex $\bar{\mathcal{F}}(G)$ is independent of this choice. We will show that there are chain maps
    \[
        g\colon C^{\Delta,m}_*(G) \to C_*(\bar{\mathcal{F}}(G))
        \;\;\; \text{and} \;\;\;
        h \colon C_*(\bar{\mathcal{F}}(G)) \to C^{\Delta,m}_*(G)
    \]
    that induce a chain homotopy equivalence between chain complexes $C^{\Delta,m}_*$ and $C_*(\bar{\mathcal{F}}(G))$.
    
    First we define the chain map $g$. Given distinct $v_0,\dots,v_n\in V_G$, define 
    \[
        \sigma_{v_0,\dots,v_n} \in \Sigma_{n+1}
    \]
    to be the permutation in the symmetric group $\Sigma_{n}$ on $n$ elements such that $v_{\sigma(0)},\dots,v_{\sigma(n)}$ is ordered under the total order chosen on $V_G$.  
    Write also $\text{sgn}(\sigma_{v_0,\dots,v_n})$ for the usual sign of the permutation.
    
    We may now define the function $g\colon C^{\Delta,m}_*(G) \to C_*(\bar{\mathcal{F}}(G))$ by linearly extending
    \begin{align*}
        g(f) = 
        \begin{cases}
            \text{sgn}(\sigma_{f(0),\dots,f(n)})\{ f(0),\dots,f(n) \}
            &
            \text{if} \: f(0),\dots,f(n) \: \text{are distinct}
            \\
            0
            &
            \text{otherwise}
        \end{cases}
    \end{align*}
    where $f \colon \Delta^n \to G$ is a singular $n$-simplex.
    
    By construction, 
    \[
        (g \circ \partial_n)(f) = \sum_{i=0}^n \alpha_i \{f(0),\dots,f(n) \}\setminus\{f(i)\} 
        \;\;\; \text{and} \;\;\;
        (\partial_n \circ g)(f) = \sum_{i=0}^n \beta_i \{f(0),\dots,f(n) \}\setminus\{f(i)\} 
    \]
    for some $\alpha_0,\dots,\alpha_n,\beta_0,\dots,\beta_n \in \{1,-1\}$.
    Therefore, to show that $g$ is a chain map it sufficient to check that $\alpha_i = \beta_i$ for $i=0,\dots,n$.
    Furthermore, since $\sigma_{f(0),\dots,f(n)} \in \Sigma_n$ is a composition of transpositions, it is sufficient to check the case when $\sigma_{f(0),\dots,f(n)}$ is a transposition.

    Fix $i=0,\dots,n$. When $\sigma_{f(0),\dots,f(n)}$ exchanges two elements not equal to $f(i)$, then $\alpha_i = \beta_i = (-1)^{i+1}$. Otherwise, $\sigma_{f(0),\dots,f(n)}$ exchanges $f(i)$ with $f(j)$, for some $j=0,\dots,n$ and $j\neq i$.
    If $j<i$, then
    \begin{equation}\label{eq:TranspositionComposition}
        \sigma_{f(0),\dots,\widehat{f(i)},\dots,f(n)}
        =
        (f(i-2),f(i-1)) \circ \cdots \circ (f(j),f(j+1))
    \end{equation}
    as a composition of transpositions and
    $\text{sgn}( \sigma_{f(0),\dots,\widehat{f(i)},\dots,f(n)}) = (-1)^{i-j-1}$.
    We now have that
    \begin{align*}
        \alpha_i =
        (g \circ (-1)^i d_i^{n-1})(f)
        &=
        g((-1)^{i}(f\circ d_{i}^{n-1}))
        \\&=
        (-1)^{i-j-1}(-1)^{i}\{f(0),\dots,\widehat{f(i)},\dots,f(n)\}
        \\&=
        (-1)^{j+1}\{f(0),\dots,\widehat{f(i)},\dots,f(n)\}
        \\&=
        (-1)^{j}(-\{f(0),\dots,\widehat{f(i)},\dots,f(n)\})
        \\&=
        (-1)^{\sigma_{f(0),\dots,f(n)}(i)}(d_{\sigma_{f(0),\dots,f(n)}(i)}^{n-1} \circ g)(f)
        \\&=
        ((-1)^j d_{j}^{n-1} \circ g)(\{f(0),\dots,f(n)\})
        = \beta_i
    \end{align*}
    as $\sigma_{f(0),\dots,f(n)}$ is a transposition such that $\sigma_{f(0),\dots,f(n)}(i)=j$ by assumption.
    When $j>i$, the argument is the same as above except that equation~\eqref{eq:TranspositionComposition} is replaced with
    \[
        \sigma_{f(0),\dots,\widehat{f(i)},\dots,f(n)}
        =
        (f(i),f(i+1)) \circ \cdots \circ (f(j-2),f(j-1))
    \]
    which has sign $\text{sgn}( \sigma_{f(0),\dots,\widehat{f(i)},\dots,f(n)}) = (-1)^{j-i-1}$.
    Therefore, $g$ is a chain map.
    
    Next we define the chain map $h$. Consider distinct $v_0,\dots,v_n\in V_G$ such that there exists a singular simplex $f\colon \Delta^n \to G$ with $\{v_0,\dots,v_n\}$ the image of $f$ on vertices. 
    Then for each such possible vertex set $\{v_0,\dots,v_n\}$ above, make a choice of singular simplex
    \[
        s_{v_0,\dots,v_n}\colon\Delta^n \to G.
    \]
    Define the function $h\colon C_*(\bar{\mathcal{F}}(G)) \to C^{\Delta,m}_*(G)$ by linearly extending
    \begin{align*}
        h(\{v_0,\dots,v_n\}) = 
        \begin{cases}
            \text{sgn}(\sigma_{s_{v_0,\dots,v_n}(0),\dots,s_{v_0,\dots,v_n}(n)})
            s_{v_0,\dots,v_n}
            &
            \begin{aligned}
            &
            \text{if} \: \exists
            \: f\colon \Delta^n \to G \:
            \text{with}
            \\
            &
            \text{image} \:
            \{ v_0,\dots,v_n\}
            \\
            &
            \text{on vertices}
            \end{aligned}
            \\
            0
            &
            \text{otherwise}.
        \end{cases}
    \end{align*}
    The function $h$ can be shown to be a chain map by a similar argument used to show $g$ was a chain map above. In addition, by construction, $g\circ h = \text{id}_{C_*(\mathcal{F}(G))}$. Therefore, it remains to show that $h\circ g$ is chain homotopic to the identity on $C^{\Delta,m}_*(G)$.

    Given $v_0,\dots,v_n \in V_G$, denote by $G_{v_0,\dots,v_n}$ the unique full subquiver of $G$ containing vertices $v_0,\dots,v_n$. To each singular simplex $f\colon \Delta^n \to G$ associate the sub-chain complex $C^{\Delta,m}_*(G_{f(0),\dots,f(n)})$.
    By Lemma~\ref{lem:ExtenededSimplexStrongContraction} each $G_{f(0),\dots,f(n)}$ is strongly contractible. Hence, using the strong homotopy invariance of $H_*^{\Delta,m}$ from Theorem~\ref{thm:PreviousHomotopyInv}, the chain complexes $C^{\Delta,m}_*(G_{f(0),\dots,f(n)})$ is acyclic and provides us with an acyclic carrier $\varphi$ on the basis of singular simplices in $C^{\Delta,m}_*(G)$.
    Moreover, both $h\circ g$ and $\text{id}_{C^{\Delta,m}_*(G)}$ are carried by $\varphi$.
    Therefore, $h\circ g$ and $\text{id}_{C^{\Delta,m}_*(G)}$ are chain homotopic by the acyclic carrier theorem (Theorem~\ref{thm:AcyclicCarriers}).
    
    It remains to check that naturality of the isomorphism on homologies.
    To this end, let $\phi \colon G_1 \to G_2$ be a digraph map. Using the functoriality of $H^{\Delta,m}_*$ from Proposition~\ref{prop:InitialFuctoriality}, the map $\phi$ induces a chain map $\phi_{\#}\colon C^{\Delta,m}_*(G_1) \to C^{\Delta,m}_*(G_2)$ as in equation~\eqref{eq:InducedMapHomology} by linearly extending $\phi_\#(f)=f\circ \phi$ for each singular simplex $f\colon \Delta^n \to G$.
    Similarly, the simplicial map $F(\phi)$ defined in equation~\eqref{eq:FlagInducedMap} induces the map $F(\phi)_{\#}\colon C_*(F(G_1)) \to C_*(F(G_2))$ by linearly extending $\mathcal{F}(\phi)_\#(\{v_0,\dots,v_n\})=(\{\phi(v_0),\dots,\phi(v_n)\})$ for each simplex $\{v_0,\dots,v_n\}$ in abstract simplicial complex $F(G_1)$.

    We now observe that the following diagram of chain maps commutes
    \[
        \xymatrix{
        C_*^{\Delta,m}(G_1) \ar[d]^g \ar[r]^{\phi_\#} & C_*^{\Delta,m}(G_2) \ar[d]^g \\
        C_*(F(G_1)) \ar[r]^{F(\phi)_\#} & C_*(F(G_2))}
    \]
    as it commutes for each singular $n$-simplex, with both $(g \circ \phi_\#)(f)$ and $(F(\phi)_\# \circ g)(f)$ being 
    equal to 
    \[
        \text{sgn}(\sigma_{(\phi \circ f)(0),\dots,(\phi \circ f)(n)})\{ (\phi \circ f)(0),\dots,(\phi \circ f)(n) \}
    \]
    when $(\phi \circ f)(0),\dots,(\phi \circ f)(n)$ are distinct and $0$ otherwise.
    Therefore, the naturality condition follows as we have already shown that the chain map $g$ induces the isomorphism on homology. 
\end{proof}

\section{Computation of \texorpdfstring{$H^{\Delta,h}_*$}{singular homomorphism homology}}\label{sec:ComputingHomomorphismHomology}

In this section we provide object-wise maps from a quiver to a suitably simpler $\Delta$-set called the partial directed flag complex, whose simplicial homology coincides with $H^{\Delta,h}_*$. In particular, the object level construction lifts to a natural functorial construction on chains, and hence, to a natural functorial construction on homology. The naturality of our result implies persistent $H^{\Delta,h}_*$ can be computed more efficiently using the homology of the partial directed flag complex. Furthermore, using the partial directed flag complex, we are able to complete a proof of the weak local strong homotopy invariance of $H^{\Delta,h}_*$ on quivers. 

Throughout the section assume that $G$ is a quiver and $n\geq 0$ is an integer unless otherwise stated. Note that some results in this section require the axiom of choice if the quiver $G$ contains an infinite number of edges.

\subsection{The partially reduced quiver}\label{sec:PartiallyReducedQuiver}

In this subsection, we provide the construction of the partially reduced quiver, which transitions locally between the reduced digraph described from Section~\ref{sec:ReducedDigraph} and the original quiver depending on the presence of loops. The partially reduced quiver is local strong homotopy equivalent to the quiver itself. Therefore, combined with the local strong homotopy invariance of $H_*^{\Delta,h}$ shown in Section~\ref{sec:HomotopyTheory} we obtain a general simplification for the computation of $H_*^{\Delta,h}$ that we extend in the next subsection.

\begin{definition}
    Let $G$ be a quiver. Then define the \emph{partially reduced quiver} $\tilde{\mathcal{R}}(G)$ to be the quiver with vertices $V_{\tilde{\mathcal{R}}(G)}=V_G$ and edges $E_{\tilde{\mathcal{R}}(G)} = E^1_{\tilde{\mathcal{R}}(G)} \cup E^2_{\tilde{\mathcal{R}}(G)}$ where
    \begin{align}\label{eq:SeperatingEdgesByLoopEnds}
        E^1_{\tilde{\mathcal{R}}(G)} & = \nonumber
        \{ e_G \in E_G \: | \:
        e_G \: \text{does not have a loop}
        \}
        \\
        E^2_{\tilde{\mathcal{R}}(G)} & =
        \{ \tilde{\mathcal{R}}(G,u,v)
        \: | \:
        u,v \in V_G 
        \: \text{satisfying} \:
        \exists \:
        e \in E_G \:
        \text{with a loop such that} \:
        s(e) = u \: \text{and} \: t(e) = v 
        \}
    \end{align}
    with
    \begin{align*}
        s_{\tilde{\mathcal{R}}(G)}(e_G) = s_G(e_G), \;
        t_{\tilde{\mathcal{R}}(G)}(e_G) = t_G(e_G), \;
        s_{\tilde{\mathcal{R}}(G)}(\tilde{\mathcal{R}}(G,u,v)) = u
        \; \text{and} \;
        t_{\tilde{\mathcal{R}}(G)}(\tilde{\mathcal{R}}(G,u,v)) = v.
    \end{align*}
\end{definition}

In brief, the partially reduced quiver collapses all multiple edges with a loop to a single edge. The partially reduced quiver extends to a functor $\tilde{\mathcal{R}}\colon \textbf{Quiv}_h \to \textbf{Quiv}_h$ by setting
\begin{align}\label{eq:PartialyReducedQuiverFunctor}
        & \tilde{\mathcal{R}}(h)(v) = h(v), \nonumber
        \\
        & \tilde{\mathcal{R}}(h)(e_G) = 
        \begin{cases}
            h(e_G) & \text{if} \: h(e_G) \: \text{has no loops}
            \\
            \tilde{\mathcal{R}}(G',s(e_G),t(e_G)) & \text{otherwise,}
        \end{cases}
        \\ \nonumber
        \text{and} & \;
        \tilde{\mathcal{R}}(h)(\tilde{\mathcal{R}}(G,u,v)) = \tilde{\mathcal{R}}(G',h(u),h(v))
\end{align}
for quiver homomorphism $h \colon G \to G'$, $u,v\in V_G$, $e_G \in E^1_{\tilde{\mathcal{R}}(G)}$, and  $\tilde{\mathcal{R}}(G,u,v) \in E^2_{\tilde{\mathcal{R}}(G)}$.
This is well defined as homomorphism $h$ must send loops to loops.

We now consider the local strong homotopy class of the partially reduced quivers.

\begin{theorem}\label{thm:StrongHMultiEdgeContractionBetweenLoops}
    The quivers $G$ and $\tilde{\mathcal{R}}(G)$ are local strong $h$-homotopy equivalent.
\end{theorem}

\begin{proof}
    Define homomorphism $f\colon G \to \tilde{\mathcal{R}}(G)$ by
    \[
        f(v) = v
        \;\;\; \text{and} \;\;\;
        f(e) = 
        \begin{cases}
            e & \text{if} \; e \in E^1_{\tilde{\mathcal{R}}(G)}
            \\
            \tilde{\mathcal{R}}(G,s_G(e),t_G(e)) & \text{otherwise} 
        \end{cases}
    \]
    which is well defined by the construction of $E^1_{\tilde{\mathcal{R}}(G)}$, $E^2_{\tilde{\mathcal{R}}(G)}$ and $\tilde{\mathcal{R}}(G,s_G(e),t_G(e))$ in equation~\eqref{eq:SeperatingEdgesByLoopEnds}. 
    Given $u,v\in E_G$ such that there exists $e\in E_G$ with a loop satisfying $s_G(e)=u$ and $t_G(e)=v$, choose a unique element
    \[
        e_{u,v}\in \{ e \in E_G \: | \: s(e) = u \: \text{and} \: t(e) = v \}.
    \]
    Define $g\colon \tilde{\mathcal{R}}(G) \to G$ by
    \[
        g(v) = v, \;\;\;
        g(e_G) = e_G 
        \;\;\; \text{and} \;\;\;
        g(\tilde{\mathcal{R}}(G,u,v)) = e_{u,v}.
    \]
    Homomorphism $g$ is also well defined by the construction of $E^1_{\tilde{\mathcal{R}}(G)}$ and $E^2_{\tilde{\mathcal{R}}(G)}$ in equation~\eqref{eq:SeperatingEdgesByLoopEnds}.
    In particular, $f \circ g = \text{id}_{\tilde{\mathcal{R}}(G)}$.

    Let $G'$ be the full subquiver of $G$ on vertices $V_{G'} = \{ v\in V_G \: | \: v \: \text{has a loop} \}$. Then define $F\colon M^S_{G'\hookrightarrow G} \to G$ given by
    \begin{align*}
        & F((v,0)) = v, \;
        F((v',1)) = v', \;
        F((e,0)) = e, \;
        F((v', 0\to 1)) = e_{v',v'}, \\
        & F((e',1)) = F((e', 0\to 1)) = e_{s(e'),t(e')}, \;
        \text{and} \;
        F(e'') = e_{s(e''),t(e'')}
    \end{align*}
    for $v\in V_G$, $v' \in V_{G'}$, $e\in E_G$, $e'\in E_{G'}$, and $e''\in E(G,G')$ as given in Definition~\ref{def:MappingCylinder} of the strong mapping cylinder $M^S_{G'\hookrightarrow G}$ of $G$ and $G'$.
    By construction, the homomorphism $F$ is a well defend $1$-step local strong $h$-homotopy between $\text{id}_G$ and $g \circ f$. 
   Therefore, $\text{id}_G \simeq^{lSh}_1 g\circ h$ which completes the proof.
\end{proof}

In fact we have proved that $\tilde{\mathcal{R}}(G)$ is a local strong deformation retraction of $G$ and we obtain the following corollary.

\begin{corollary}\label{cor:ReducedHomomorphismQuiver}
    Let $G$ be a quiver, then
    \[
        H_*^{\Delta,h}(G) = H_*^{\Delta,h}(\tilde{\mathcal{R}}(G))
    \]
    naturally with respect to induced maps.
\end{corollary}

\begin{proof}
    The corollary follows from Theorem~\ref{thm:StrongHMultiEdgeContractionBetweenLoops} and the local strong $h$-homotopy invariance of $H_*^{\Delta,h}$ shown in Theorem~\ref{thm:LocalStrongHInvIntial}. The naturality of this isomorphism can be obtained directly from the construction of $\tilde{\mathcal{R}}$ as a functor in equation~\eqref{eq:PartialyReducedQuiverFunctor} and induced maps for $H_*^{\Delta,h}$ given in equation~\eqref{eq:InducedMapHomology}.
\end{proof}

\subsection{The partial directed flag complex}\label{sec:PartialFlagComplex}

Motivated by computation, we now extend the construction of the partially reduced quiver functor $\tilde{\mathcal{R}}$ from the previous subsection and construct for any quiver $G$ a $\Delta$-set called the partial directed flag complex with homology isomorphic to $H^{\Delta,h}_*(G)$. While the partial directed flag complex is not a functor, the composition with the chain functors becomes functorial. When the quiver $G$ has no loops the partial directed flag complex coincides with $\mathcal{F}(G)$. If every vertex of $G$ has a loop and a total order on $V_G$ is selected, the partial directed flag complex coincides with $\bar{\mathcal{F}}(G)$ as a $\Delta$-set. The construction of the partial directed flag complex allows us to complete the section by showing that $H^{\Delta,h}_*$ is invariant under the weak local strong homotopy of quivers given in Definition~\ref{def:QuiverLocalStrongHomotopy}. 

The structure of the subsection is similar to Section~\ref{sec:RducedDirectedFlag}, though the construction of the reduced directed flag complex is considerably more detailed. We first set out the following terminology.

\begin{definition}
    Let $G$ be a quiver. A total order $<$ on $V_G$ is called \emph{loop maximal} if $u,v \in V_G$ such that $u$ has a loop and $v$ has no loop, then $v < u$.
\end{definition}

It is clear that a loop maximal total order exists on the vertices of any quiver. We now set out the central construction of the section.

Let $G$ be a quiver and $<$ a loop maximal total order on the vertices of $G$.
Then define the \emph{partial directed flag complex} $\tilde{\mathcal{F}}_{<}(G)$ to be the following $\Delta$-set. Note that we show later that $H_*(\tilde{\mathcal{F}}_<(G))$ is independent of the choice of loop maximal total order $<$. In dimensions $0$ and $1$ we define
\begin{align*}
    \tilde{\mathcal{F}}_<(G)_0 &= V_G \\
    \tilde{\mathcal{F}}_<(G)_1 &= E^1_{\tilde{\mathcal{F}}_<(G)} \cup E^2_{\tilde{\mathcal{F}}_<(G)}
\end{align*}
where
\begin{align}\label{eq:PartailFlagEdges}
    E^1_{\tilde{\mathcal{F}}_<(G)} &= 
    \{\{e_G\}  \: | \: e_G \in E_G \: \text{has no loop} \} \nonumber \\
    E^2_{\tilde{\mathcal{F}}_<(G)} &= 
    \{ \{ \tilde{\mathcal{F}}_{<}(u,v) \} 
    \: | \:
    u,v\in V_G, \:
    u<v, \: \exists \: e \in E_G 
    \: \text{with} \:
    \\ & \;\;\;\;\;\;\;\;
    s(e) = u, \: t(e) = v 
    \; \text{or} \:
    s(e) = v, \: t(e) = u,
    \: \text{and $e$ has a loop} \}
    \nonumber
\end{align}
with face maps given by
\begin{align*}
    & d^0_0(\{e_G\}) = t(e_G),\: d^0_1(\{e_G\}) = s(e_G), \\
    \text{and} \;
    & d^0_0(\{\tilde{\mathcal{F}}_<(u,v) \}) = v, \:
    d^0_1(\{\tilde{\mathcal{F}}_<(u,v)\}) = u.
\end{align*}
It can be useful to view the construction of the $1$-skeleton above as providing a function $\tilde{\mathcal{F}}_<^1(G) \colon E_G \to \tilde{\mathcal{F}}_<(G)_1$ given by
\[
    \tilde{\mathcal{F}}_<^1(G)(e) =
    \begin{cases}
        \{ \tilde{\mathcal{F}}_{<}(s(e),t(e)) \}
        & \text{if} \: e \: \text{has a loop and} \: s(e) < t(e)
        \\
        \{ \tilde{\mathcal{F}}_{<}(t(e),s(e)) \}
        & \text{if} \: e \: \text{has a loop and} \: t(e) < s(e)
        \\
        \{ e \} & \text{otherwise}
    \end{cases}
\]
for each edge $e \in E_G$.
For simplicity of notation, from now on we denote the image of $\tilde{\mathcal{F}}_<^1(G)(e)$ by $\{\alpha_e\}$ when it is clear from the context in which $\tilde{\mathcal{F}}_<(G)$ it forms an edge. 

\begin{remark}\label{rmk:PartialCollapsTotalOrder}
    It is important to note that the image of the vertices of any singular simplicial inclusion $f\colon \Delta^n \to G$ can be given a total order $<_f$ differing from that on the indices in $\Delta^n$ instead induced by the elements of $\tilde{\mathcal{F}}_<(G)_1$ as follows. 

    For any integers $0\leq i<j \leq n$ we have $i \to j \in E_{\Delta^n}$ and
    $\{\alpha_{f(i \to j)}\} \in \tilde{\mathcal{F}}_<(G)_1$.
    Define the order $<_f$ on vertices $i$ and $j$ by $i <_f j$ if and only if $d^0_1(\{\alpha_{f(i\to j)}\}) = f(i)$ and $d^0_0(\{\alpha_{f(i\to j)}\}) = f(j)$, with $j <_f i$ otherwise.
    The total order $<_f$ is well defined since
    \begin{enumerate}[(1)]
        \item 
        vertices $f(0),\dots,f(n) \in v_G$ are distinct as $f$ is an inclusion;
        \item
        the order $<_f$ on vertices $i=0,\dots,n$ such that $f(i)$ has no loop is determined by the total order on vertices from $\Delta^n$ using the definition of $E^1_{\tilde{\mathcal{F}}_<(G)}$ in equation~\eqref{eq:PartailFlagEdges};
        \item 
        the order $<_f$ on vertices $i=0,\dots,n$ such that $f(i)$ has a loop is determined by the total order $<$ using the definition of $E^2_{\tilde{\mathcal{F}}_<(G)}$ in equation~\eqref{eq:PartailFlagEdges};
        \item
        otherwise $i <_f j$ for $i,j=0,\dots,n$ such that $f(i)$ has no loop and $f(j)$ has a loop, by loop maximality of $<$ and the definition of $E^2_{\tilde{\mathcal{F}}_<(G)}$ in equation~\eqref{eq:PartailFlagEdges}.
    \end{enumerate}
\end{remark}

We now return to the construction of $\tilde{\mathcal{F}}_<(G)_n$ for $n\geq 2$ and define
\begin{equation}\label{eq:HigherDeltaSet}
    \tilde{\mathcal{F}}_<(G)_n =
    \{ \{ \alpha_{f(i\to j)}\}_{0\leq i < j \leq n} 
    \: | \: 
    f\colon \Delta^n \to G 
    \: \text{is an inclusion} \}
\end{equation}
with face maps given by 
\begin{equation}\label{PartialFlagFaceMaps}
    d^{n-1}_k\left(\{\alpha_{e_{i,j}}\}_{0\leq i < j \leq n}\right) =
    \{\alpha_{e_{\phi_k(i),\phi_k(j)}}\}_{0\leq i < j \leq n-1}
\end{equation}
for $k=0,\dots,n$ and where $\phi_k\colon \{0,\dots,n-1\} \to \{0,\dots,n\}$ is the function
\[
    \phi_k(t) =
    \begin{cases}
        t & \text{if} \: t < k \\
        t+1 & \text{otherwise}.
    \end{cases}
\]
The face maps $d^{n}_k$ for $k=0,\dots,n+1$ satisfy the conditions of Equation~\eqref{eq:FaceMapConditions} when the image lies in dimension $n \geq 2$ or $0$ by construction and when the image lies in dimension $1$ using the fact that edges respect the total order $<_f$ from Remark~\ref{eq:HigherDeltaSet}.

\begin{definition}\label{def:InducingSingularSimplex}
    Let $n\geq 1$ and $\alpha =\{\alpha_{e_{i,j}}\}_{0\leq i < j \leq n}\in \tilde{F}_<(G)_n$. Then by its construction in equation~\eqref{eq:PartailFlagEdges}~or~\eqref{eq:HigherDeltaSet}, $\alpha$ depends on the the existence of a singular simplex inclusion $f\colon\Delta^n \to G$. We call such a singular $n$-simplex $f$ a \emph{singular simplex inducing} $\alpha$. When $n=0$ the \emph{singular simplex inducing} $v\in V_G$ is defined to be the singular zero simplex $f\colon \Delta^0 \to G$ such that $f(0) = v$.
\end{definition}

Unlike the flag complex $\mathcal{F}(G)$ and the reduced flag complex $\bar{\mathcal{F}}(G)$, a choice of loop maximal total order $<$ on $V_G$ is required in order to define $\tilde{F}_<(G)$. Consequently, we cannot directly extend $\tilde{F}_<$ to a functor similarly to $\mathcal{F}$ and $\bar{\mathcal{F}}$ in equations~\eqref{eq:FlagFunctor}~and~\eqref{eq:FlagInducedMap}, respectively. The next example demonstrates that in general even in the presence of a choice of total orders on the vertex sets of two quivers, an induced map on the resulting partial directed flag complexes cannot necessarily be defined. However, for computation of persistent homology we required only that the homology of the partial directed flag complex interacts naturally with the homology of a filtered quiver. Therefore, a functorial partial directed flag complex construction $\colon\textbf{Quiv}_h \to \textbf{DSets}$ is not essential.

\begin{example}
    Let $G_1$ and $G_2$ be the following quivers.
    \begin{center}
        \;\;\;\;\;\;\;\;\;
        
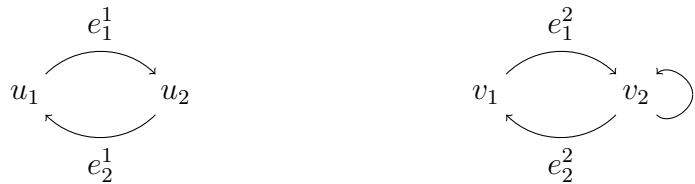
\tikz {
            \node (a) at (0,0) {$u_1$};
            \node (b) at (2,0) {$u_2$};
            \draw[->] (a) to [out=45,in=135] node[pos=0.5,above] {$e_1^1$} (b);
            \draw[->] (b) to [out=225,in=315] node[pos=0.5,below] {$e^1_2$} (a);
        }
        \;\;\;\;\;\;\;\;\;\;\;\;\;\;\;\;\;\;\;\;\;\;\;\;\;\;\;
        \tikz {
            \node (a) at (0,0) {$v_1$};
            \node (b) at (2,0) {$v_2$};
            \node (d) at (2.75,-0.145) {};
            \node (c) at (2.75,0.145) {};
            \draw[->] (a) to [out=45,in=135] node[pos=0.5,above] {$e^2_1$} (b);
            \draw[->] (b) to [out=225,in=315] node[pos=0.5,below] {$e^2_2$} (a);
            \draw[-] (b) to [out=315,in=270] (c);
            \draw[->] (d) to [out=90,in=45] (b);
        }
        \;\;\;
    \end{center}
    There is an quiver homomorphism $h \colon G_1 \to G_2$ determined by
    \[
        h(u_1) = v_1,
        \;\;\;
        h(u_2) = v_2,
        \;\;\;
        h(e_1^1) = e_1^2,
        \;\;\; \text{and} \;\;\;
        h(e_2^1) = e_2^2.
    \]
    Since $G_1$ has no loops, there are two possible loop maximal total orders on $V_{G_1}$, both of which are the same up to symmetry of the quiver. Hence, without loss of generality we choose to take $u_1 <_1 u_2$ as a total order on $V_{G_1}$. Meanwhile, $G_2$ has a loop at vertex $v_2$. Therefore, the only loop maximal total order on $V_{G_2}$ is $v_1 <_2 v_2$.

    The $1$-dimensional $\Delta$-sets $\tilde{\mathcal{F}}_{<_1}(G_1)$ and $\tilde{\mathcal{F}}_{<_2}(G_2)$ are
    
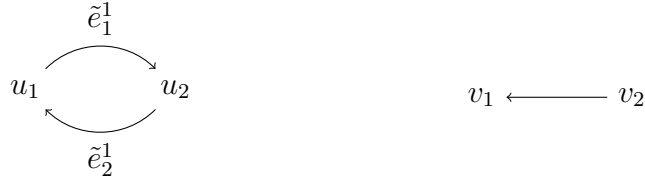
\begin{center}
        \tikz {
            \node (a) at (0,0) {$u_1$};
            \node (b) at (2,0) {$u_2$};
            \draw[->] (a) to [out=45,in=135] node[pos=0.5,above] {$\tilde{e}^1_1$} (b);
            \draw[->] (b) to [out=225,in=315] node[pos=0.5,below] {$\tilde{e}^1_2$} (a);
        }
        \;\;\;\;\;\;\;\;\;
        \;\;\;\;\;\;\;\;\;
        \;\;\;\;\;\;
        \tikz {
            \node (c) at (0,0) {};
            \node (a) at (0,1.05) {$v_1$};
            \node (b) at (2,1.05) {$v_2$};
            \draw[->] (b) -- (a);
        }
        \;
    \end{center}
    where the arrows denote the $1$-simplicies with the image of $d^0_0$ the vertex at the source of the arrow and the image of $d^0_1$ the vertex at the target of the arrow within each $\Delta$-set, respectively.

    An induced map $\tilde{\mathcal{F}}(h)\colon \tilde{\mathcal{F}}_{<_1}(G_1) \to \tilde{\mathcal{F}}_{<_2}(G_2)$ would be required to preserve the image of $h$ on vertices, with $\tilde{\mathcal{F}}(h)(u_1) = v_1$ and $\tilde{\mathcal{F}}(h)(u_2) = v_2$. However, if the image of $\tilde{\mathcal{F}}(h)$ restricted to vertices is of the form above, then $\tilde{\mathcal{F}}(h)$ cannot be a map of $\Delta$-sets as there does not exist a $1$-simplex in $\tilde{\mathcal{F}}_{<_2}(G_2)$ that can be the image of $\tilde{\mathcal{F}}(h)(\tilde{e}^1_1)$.
\end{example}

While the construction of $\tilde{F}_<(G)$ cannot be extended to a functor $\colon \textbf{Quiv}_h \to \textbf{DSets}$, we now show that after composition with the chain functor a functorial construction can be extended to chain maps. 

Recall that $C_*\colon {\bf DSet} \to \textbf{Chain}$ is the chain functor defined at the end of Section~\ref{sec:Spaces}.
Let $\mathcal{C}$ be a subcategory of ${\bf Quiv}_h$ and to every object $G\in \mathcal{C}$ associate a loop maximal total order $<_G$.
We now define $\tilde{f}_{\{<_G\}_{G\in \mathcal{C}}}\colon \mathcal{C} \to \textbf{Chain}$ which we prove is a functor in the next proposition.
On objects $\tilde{f}_{\{<_G\}_{G\in \mathcal{C}}}$ is given by
\begin{align*}
    \tilde{f}_{\{<_G\}_{G\in \mathcal{C}}}(G') = C_*(\tilde{F}_{<_{G'}}(G'))
\end{align*}
for each $G' \in \mathcal{C}$.
On quiver homomorphisms $\phi \colon G_1 \to G_2$ in $\mathcal{C}$, $\tilde{f}_{\{<_G\}_{G\in \mathcal{C}}}$ is defined by linearly extending
\begin{align}\label{eq:PartialFlagOnmorphisms}
    & \tilde{f}_{\{<_G\}_{G\in \mathcal{C}}}(\phi)(v)
    =
    \phi(v)
     \;\;\; \text{and}
     \nonumber
     \\ &
    \tilde{f}_{\{<_G\}_{G\in \mathcal{C}}}(\phi)
    \left(\{\alpha_{e_{i,j}}\}_{0\leq i < j \leq n}\right)
    \begin{cases}
        \text{sgn}(\{\alpha_{e_{i,j}}\}_{0\leq i < j \leq n},\phi)
        \{\alpha_{\phi(e_{i,j})}\}_{0\leq i < j \leq n}
        &
        \begin{aligned}
            & \text{if} \: s(\phi(e_{i,j})) \neq t(\phi(e_{i,j})) \\
            & \text{for each} \: 0 \leq i < j \leq n
        \end{aligned}
        \\
        0 & \text{otherwise}
    \end{cases}
\end{align}
where $\text{sgn}(\{\alpha_{e_{i,j}}\}_{0\leq i < j \leq n},\phi)$ is the sign of a permutation realised as follows.

Using Definition~\ref{def:InducingSingularSimplex}, any choice of singular simplex $f\colon \Delta^n \to G_1$ inducing $\{\alpha_{e_{i,j}}\}_{0\leq i < j \leq n}$ provides the total order $<_f$ on $0,\dots,1$ described in Remark~\ref{rmk:PartialCollapsTotalOrder}. Similarly, any choice of singular simplex inducing $f'\colon \Delta^n \to G_2$ of $\{\alpha_{\phi(e_{i,j})}\}_{0\leq i < j \leq n}$ provides the total order $<_{f'}$ on $0,\dots,1$.
In particular, function $\tilde{f}_{\{<_G\}_{G\in \mathcal{C}}}(\phi)$ is well defined as $f'$ can be chosen to be $\phi \circ f$, hence exits. In the case when $\phi \circ f$ is not an inclusion $\tilde{f}_{\{<_G\}_{G\in \mathcal{C}}}(\phi) = 0$ by construction. Otherwise, we set $\text{sgn}(\{\alpha_{e_{i,j}}\}_{0\leq i < j \leq n},\phi)$ to be the sign of the permutation between the total orders $<_f$ and $<_{f'}$ on $0,\dots,n$.

We now verify that $\tilde{f}_{\{<_G\}_{G\in \mathcal{C}}}$ acts functorialy.

\begin{proposition}\label{prop:PartialFlagFunctor}
    Let $\mathcal{C}$ be a subcategory of ${\bf Quiv}_h$ and to every object $G\in \mathcal{C}$ associate a loop maximal total order $<_G$. Then 
    \[
        \tilde{f}_{\{<_G\}_{G\in \mathcal{C}}}\colon \mathcal{C} \to \textbf{Chain}
    \]
    is a functor.
\end{proposition}

\begin{proof}
    Let $\phi \colon G_1 \to G_2$ and $\varphi \colon G_2 \to G_3$ be quiver homomorphisms. It is sufficient to check on each each singular simplicial generator of $C_*(\tilde{\mathcal{F}}_{<_{G_1}}(G_1))$ that $\tilde{f}_{\{<_G\}_{G\in \mathcal{C}}}$ acts functorialy.
    
    For $v \in C_0(\tilde{\mathcal{F}}_{<_{G_1}}(G_1))$, we have
    \[
        \tilde{f}_{\{<_G\}_{G\in \mathcal{C}}}(\varphi \circ \phi)(v)
        =
        (\varphi \circ \phi)(v)
        =
        \left(\tilde{f}_{\{<_G\}_{G\in \mathcal{C}}}(\varphi) \circ \tilde{f}_{\{<_G\}_{G\in \mathcal{C}}}(\phi) \right)(v).
    \]
    Consider now $\{\alpha_{e}\} \in C_1(\tilde{\mathcal{F}}_{<_{G_1}}(G_1))$. 
    If $\phi(s(e))=\phi(t(e))$ or $(\varphi\circ\phi)(s(e))=(\varphi\circ\phi)(t(e))$, then $(\tilde{f}_{\{<_G\}_{G\in \mathcal{C}}}(\varphi) \circ \tilde{f}_{\{<_G\}_{G\in \mathcal{C}}}(\phi))(\{\alpha_{e}\}) = 0$ and $\tilde{f}_{\{<_G\}_{G\in \mathcal{C}}}(\varphi \circ \phi)(\{\alpha_{e}\}) = 0$ by construction.
    Otherwise, suppose first that $\alpha_e = \tilde{\mathcal{F}}_{<_{G_1}}(u,v)$ for some $u,v\in V_{G_1}$ with $u <_{G_1} v$. By loop maximality of $<_{G_1}$ and that $\phi$, $\varphi$ are homomorphisms, we have that $v$, $\phi(v)$, and $(\varphi \circ \phi)(v)$ have a loop, implying that
    \begin{equation}\label{eq:EdgeTyep1}
        \tilde{f}_{\{<_G\}_{G\in \mathcal{C}}}(\phi)\left(\{\alpha_e\}\right)
        =
        \begin{cases}
            \;\:\{\tilde{\mathcal{F}}_{<_{G_2}}(\phi(u),\phi(v))\} & \text{if} \: \phi(u) <_{G_2} \phi(v)
            \\
            -\{\tilde{\mathcal{F}}_{<_{G_2}}(\phi(v),\phi(u))\} & \text{otherwise}
        \end{cases}
    \end{equation}
    with an analogous expressions for
    $\tilde{f}_{\{<_G\}_{G\in \mathcal{C}}}(\varphi)(\{\alpha_{\phi(e)}\})$.
    From these possibilities, we can conclude that
    \begin{align*}
        \tilde{f}_{\{<_G\}_{G\in \mathcal{C}}}(\varphi \circ \phi)(\{\alpha_e\})
        & =
        \begin{cases}
            \;\;\:\{\tilde{\mathcal{F}}_{<_{G_2}}((\varphi \circ \phi)(u),(\varphi \circ \phi)(v))\} &
                \text{if}
                \: (\varphi \circ \phi)(u) <_{G_3} (\varphi \circ \phi)(v)
            \\
            -\{\tilde{\mathcal{F}}_{<_{G_2}}((\varphi \circ \phi)(v),(\varphi \circ \phi)(u))\} & \text{otherwise}
        \end{cases}
        \\
        & =
        \left(\tilde{f}_{\{<_G\}_{G\in \mathcal{C}}}(\varphi ) \circ \tilde{f}_{\{<_G\}_{G\in \mathcal{C}}}( \phi) \right)(\{\alpha_e\}).
    \end{align*}
    where in the first case above either both $\text{sgn}(\{\alpha_e\},\phi)=\text{sgn}(\tilde{f}_{\{<_G\}_{G\in \mathcal{C}}}(\phi)(\{\alpha_e\}),\varphi)=\pm 1$
    and in the second case $\text{sgn}(\{\alpha_e\},\phi)$ and $\text{sgn}(\tilde{f}_{\{<_G\}_{G\in \mathcal{C}}}( \phi)(\{\alpha_e\}),\varphi)$ have opposite signs.
    
    The remaining possibility is that $e$ has no loop and $\{\alpha_e\} = \{e\}$.
    In this situation, we have
    \begin{equation}\label{eq:EdgeTyep2}
        \tilde{f}_{\{<_G\}_{G\in \mathcal{C}}}(\phi)(\{\alpha_e\}) = 
        \begin{cases}
            \;\;\: \{\tilde{\mathcal{F}}_{<_{G_2}}(\phi(s(e_{G_1})),\phi(t(e_{G_1})))\}
            &
            \begin{aligned}
            & \text{if} \: \phi(e) \: \text{has a loop and}
            \\ &  \{\phi(s(e)) <_{G_2} \phi(t(e))\}
            \end{aligned}
            \\
            -\{\tilde{\mathcal{F}}_{<_{G_2}}(\phi(t(e)),\phi(s(e)))\}
            &
            \begin{aligned}
            & \text{if} \: \phi(e) \: \text{has a loop and}
            \\ & \: \phi(t(e)) <_{G_2} \phi(s(e))
            \end{aligned}
            \\
            \;\;\: \{\phi(e)\} & \text{otherwise} 
        \end{cases}
    \end{equation}
    with an analogous expression for $\tilde{f}_{\{<_G\}_{G\in \mathcal{C}}}(\varphi)(\{e_{G_2}\})$ where $e_{G_2} \in E_{G_2}$ is an edge with no loop.
    It is necessary to check each of the possible combinations of cases arising for $\phi$ and $\varphi$ arising from equation~\eqref{eq:EdgeTyep2}, and the $\varphi$ case of equation~\eqref{eq:EdgeTyep1}.
    
    Firstly, when $\phi(e)$ has a loop, then as $\varphi$ is a homomorphism $(\varphi\circ\phi)(e)$ must have a loop. Therefore, using the $\phi$ version of equation~\eqref{eq:EdgeTyep2} and the $\varphi$ version of equation~\eqref{eq:EdgeTyep1}, both
    $\left(\tilde{f}_{\{<_G\}_{G\in \mathcal{C}}}(\varphi ) \circ \tilde{f}_{\{<_G\}_{G\in \mathcal{C}}}( \phi) \right)(\{\alpha_e\})$ and $\tilde{f}_{\{<_G\}_{G\in \mathcal{C}}}(\varphi \circ \phi)(\{\alpha_e\})$ are equal to 
    \begin{align}\label{eq:MiexLoopCase}
        \left\{\tilde{\mathcal{F}}_{<_{G_2}}((\varphi \circ \phi)(s(e)),(\varphi \circ \phi)(t(e)))\right\}
        \; \text{or} \;-
        \left\{\tilde{\mathcal{F}}_{<_{G_2}}((\varphi \circ \phi)(t(e)),(\varphi \circ \phi)(s(e)))\right\}
    \end{align}
    depending on whether $s(\phi(e))<_{G_3}t((\varphi \circ \phi)(e))$ or $t((\varphi \circ \phi)(e))<_{G_3}s(\phi(e))$, respectively.

    Similarly, we now consider the case when $\phi(e)$ does not have a loop but $(\varphi \circ\phi)(e)$ does have a loop. In this case, using the $\phi$ and $\varphi$ versions of equation~\eqref{eq:EdgeTyep2}, the resulting expressions for both $\tilde{f}_{\{<_G\}_{G\in \mathcal{C}}}(\varphi \circ \phi)(\alpha_e)$ and $\left(\tilde{f}_{\{<_G\}_{G\in \mathcal{C}}}(\varphi ) \circ \tilde{f}_{\{<_G\}_{G\in \mathcal{C}}}( \phi) \right)(\alpha_e)$ are the same as in equation~\eqref{eq:MiexLoopCase}, again depending on whether $s((\varphi\circ\phi)(e))<_{G_3}t((\varphi\circ\phi)(e))$ or $t((\varphi\circ\phi)(e))<_{G_3}s((\varphi\circ\phi)(e))$, respectively.

    The last case to consider occurs when $(\varphi \circ\phi)(e)$ does not have a loop, which by the property that homomorphisms send loops to loops implies that $\phi(e)$ cannot have a loop. Therefore, by applying the $\phi$ and $\varphi$ versions of equation~\eqref{eq:EdgeTyep2}, we obtain
    \[
        \tilde{f}_{\{<_G\}_{G\in \mathcal{C}}}(\varphi \circ \phi)(\alpha_e)
        =
        \{ (\varphi \circ \phi)(e_{G_1}) \}
        =
        \left(\tilde{f}_{\{<_G\}_{G\in \mathcal{C}}}(\varphi ) \circ \tilde{f}_{\{<_G\}_{G\in \mathcal{C}}}( \phi) \right)(\alpha_e)
    \]
    as required.

    Finally, using the the second line of equation~\eqref{eq:PartialFlagOnmorphisms}, fuctoriality for $n\geq 2$ on basis elements $\alpha = \{\alpha_{e_{i,j}}\}_{0 \leq i < j \leq n} \in C_n(\tilde{\mathcal{F}}_{<_{G_1}}(G_1))$ reduces to the $C_1(\tilde{\mathcal{F}}_{<_{G_1}}(G_1))$ situation on each element, being the same as the combination of the $n=1$ cases of $\{\alpha_{i,j}\}$ for $0 \leq i < j \leq n$ above up to sign.
    The signs also agree in all situations
    as $\text{sgn}(\{\alpha_{e_{i,j}}\}_{0 \leq i < j \leq n},\varphi) \circ \text{sgn}(\{\alpha_{e_{i,j}}\}_{0 \leq i < j \leq n},\phi) = \text{sgn}(\{\alpha_{e_{i,j}}\}_{0 \leq i < j \leq n},\varphi \circ \phi)$ by construction.
\end{proof}

From now on we refer to the functors provided by the pervious proposition as \emph{partial directed flag functors}. We are now ready to state and prove the main result of the section.

\begin{theorem}\label{thm:PartialSimplicalGraphHomologies}
    Let $G$ be a quiver and $<$, a loop maximal total order on the vertices of $G$. Then there is an isomorphism of graded modules
    \[
        H_*^{\Delta,h}(G) = H_*(\tilde{\mathcal{F}}_{<}(G))
    \]
    induced by a chain map $g$.
    Moreover, let $\phi\colon G_1\to G_2$ be a quiver homomorphism, and $<_{1}$, $<_2$ loop maximal total orders on the vertices of $G_1$, and $G_2$, respectively. Then the following diagram commutes.
    \[
        \xymatrixcolsep{5pc}\xymatrix{
        H_*^{\Delta,h}(G_1) \ar[d]^{g_*} \ar[r]^{\phi_*} & H_*^{\Delta,h}(G_2) \ar[d]^{g_*} \\
        H_*(\tilde{F}_{<_1}(G_1)) \ar[r]^{\tilde{f}_{\{<_1,<_2\}}(\phi)_*} & H_*(\tilde{F}_{<_2}(G_2))}
    \]
    In particular, using the notation given prior to the statement, the composition of functors $H_* \circ \tilde{f}_{\{<_G\}_{G\in \mathcal{C}}}$ is independent of the choice of loop maximal total orders $\{<_G\}_{G\in \mathcal{C}}$ and the isomorphism $g_*$ is a natural transformation between functors $H_*^{\Delta,h}$ and $H_* \circ \tilde{f}_{\{<_G\}_{G\in \mathcal{C}}}$.
\end{theorem}

\begin{proof}
    Corollary~\ref{cor:ReducedHomomorphismQuiver} states that $H_*^{\Delta,h}(G) = H_*^{\Delta,h}(\tilde{\mathcal{R}}(G))$ naturally with respect to induced maps.
    Therefore, we assume that $G=\tilde{\mathcal{R}}(G)$ from now on.
    This implies that all directed edges between vertices $u,v\in V_G$ that have a loop are unique. 
    
    To prove the first part of the statement, we will construct chain maps
    \[
        g\colon C^{\Delta,h}_*(G) \to C_*(\tilde{\mathcal{F}}_{<}(G))
        \;\;\; \text{and} \;\;\;
        h \colon C_*(\tilde{\mathcal{F}}_{<}(G)) \to C^{\Delta,h}_*(G)
    \]
    that induce a chain homotopy equivalence between chain complexes $C^{\Delta,h}_*(G)$ and $C_*(\tilde{\mathcal{F}}_{<}(G))$.

    Let $f \colon \Delta^n \to G$ be a singular $n$-simplex inclusion.
    Then define $\sigma_{<}^f\in \Sigma(n+1)$ to be the permutation that sends the usual total order on vertices $0,\dots,n$ to the total order $<_f$ given in Remark~\ref{rmk:PartialCollapsTotalOrder}.
    More precisely, we construct the permutation $\sigma_{<}^f$ as follows.
    
    Take the longest sequence $0 \leq a^f_0 < \cdots < a^f_m \leq n$ such that for each $j=0,\dots,m$ vertex $f(a^f_j)$ has a loop.
    Similarly, we also obtain $0 \leq b^f_1 < \cdots < b^f_{n-m} \leq n$ such that for each $k=1,\dots,n-m$ vertex $f(b_k)$ does not have a loop, satisfying $\{ b^f_1,\dots,b^f_{n-m} \} \cup \{ a^f_0,\dots,a^f_n \} = \{ 0,\dots,n \}$.
    Let $\{i_0,\dots,i_m\}=\{a^f_1,\dots,a^f_m\}$ be such that $f(a^f_{i_1}) < \cdots < f(a^f_{i_m})$.
    Then,
    \begin{equation}\label{eq:VertexSequence}
        \sigma_{<}^f(t) = 
        \begin{cases}
            a^f_{i_{t-n+m}} &
            \text{if} \:
            t \leq n-m \\
            b^f_{t+1} & \text{otherwise}.
        \end{cases}
    \end{equation}
    for $t=0,\dots,n$.
    
    We also define
    $s_{<}^f \in \tilde{\mathcal{F}}_{<}(G)$ given by
    \[
        s_{<}^f = f(0) \in \tilde{\mathcal{F}}_{<}(G)_0
        \;\;\; \text{when} \;\;\; n=0 \;\;\; \text{and} \;\;\;
         s_{<}^f = \{ \alpha_{f(i \to j)} \}_{0 \leq i < j \leq n}
         \;\;\ \text{when} \;\;\; n \geq 1.
    \]
    
    Now we may define $g\colon C^{\Delta,h}_*(G) \to C_*(\tilde{\mathcal{F}}_<(G))$ by linearly extending
    \begin{align}\label{eq:PartialQuotientChainMap}
        g(f) = 
        \begin{cases}
            \text{sgn}(\sigma_{<}^f) s_{<}^f
            &
            \text{if} \: f \: \text{is an inclusion}
            \\
            0
            &
            \text{otherwise}
        \end{cases}
    \end{align}
    and show that it is a chain map a follows.

    Suppose first that $f\colon \Delta^n \to G$ is a singular simplex homomorphism that is not an inclusion, the exists $0\leq i < j \leq n$ such that $f(i)$ = $f(j)$.
    In which case, for each $k=0,\dots,n$ and $k \neq i,j$, we have
    \[
        g(f) = 0 \;\;\; \text{and} \;\;\; g(f \circ \delta_k) = 0
    \]
    as $f \circ \delta_k$ is not a singular simplex inclusion.
    Meanwhile, if $f \circ \delta_i$ is not an inclusion, then $f \circ \delta_j$ is also not an inclusion and 
    $g(f \circ \delta_i) = g(f \circ \delta_j) = 0$.
    In any other case, we have
    \[
        g(f \circ \delta_i) = (-1)^i\text{sgn}(\sigma_<^{f\circ \delta_i}) s_<^{f\circ \delta_i}
        \;\;\; \text{and} \;\;\;
        g(f \circ \delta_j) = (-1)^j\text{sgn}(\sigma_<^{f\circ \delta_j}) s_<^{f\circ \delta_j}.
    \]
    We now show that the two expressions above always differ by a sign.
    
    As $f(i) = f(j)$ and $f$ is homomorphism $f(i \to j)$ is a loop at $f(i)$. The only edges in the image of $f \circ \delta_i$ that may not be in the image of $f \circ \delta_j$ are those of the form $f(k \to i)$ of $f(i \to k')$ for integers $0 \leq k < i < k' \leq n$ and $k,k' \neq j$. Similarly, the only edges in the image of $f \circ \delta_j$  that may not be in the image of $f \circ \delta_i$ are those of the form $f(k \to j)$ of $f(j \to k')$ for integers $0 \leq k < j < k' \leq n$ and $k,k' \neq i$. However, as $f(i)=f(j)$ has a loop $\alpha_{f(k \to i)} = \alpha_{f(k \to j)}$, $\alpha_{f(i \to k')} = \alpha_{f(k' \to j)}$, and $\alpha_{f(i \to k'')} = \alpha_{f(j \to k'')}$ for any integers $0 \leq k < i < k' < j \leq k'' \leq n$. Therefore, as $s_<^{f\circ \delta_i}$ and $s_<^{f\circ \delta_i}$ have the same edges, $s_<^{f\circ \delta_i} = s_<^{f\circ \delta_j}$ by the construction of simplices in $\tilde{\mathcal{F}}_<(G)$.
    
    The permutation $\sigma_{i,j}$ sending $0,\dots,n$ to $0,\dots,i-1,i+1,\dots,j,i,j+1\dots,n$ has sign $(-1)^{j-i-1}$. Since $\sigma_<^{f\circ \delta_i} = \sigma_<^{f\circ \delta_j} \circ \sigma_{i,j}$, we obtain that 
    \[
        \text{sgn}(\sigma_<^{f\circ \delta_i}) = (-1)^{j-i-1}\text{sgn}(\sigma_<^{f\circ \delta_j}).
    \]
    Therefore, $g(f \circ \delta_i) = - g(f \circ \delta_j)$.
    
    Putting together everything above, if $f\colon \Delta^n \to G$ is a singular simplex homomorphism that is not an inclusion, then
    \[
        (g \circ \partial_n)(f) = (\partial_n \circ g)(f) = 0.
    \]
    Otherwise, $f$ is an inclusion and as each face $f \circ \delta_k$ is a distinct inclusion for $k=0,\dots,n$. Hence using the notation of equation~\eqref{PartialFlagFaceMaps}, we obtain that
    \[
        (g \circ \partial_n)(f) = \sum_{k=0}^n \alpha_i \{\alpha_{f({\phi_k(i) \to \phi_k(j)})}\}_{0\leq i < j \leq n-1}
        \;\;\; \text{and} \;\;\;
        (\partial_n \circ g)(f) = \sum_{k=0}^n \beta_i \{\alpha_{f({\phi_k(i) \to \phi_k(j)})}\}_{0\leq i < j \leq n-1}
    \]
    for some $\alpha_0,\dots,\alpha_n,\beta_0,\dots,\beta_n \in \{1,-1\}$.
    Therefore, to show that $g$ is a chain map it sufficient to check that $\alpha_i = \beta_i$ for $i=0,\dots,n$.
    To achieve this
    the same argument used for the chain map $g$ in the proof of Theorem~\ref{thm:SimplicalGraphHomologies} can be applied.
    Where in fact it is sufficient to only check transpositions on a pair of elements from $a^f_{i_1},\dots,a^f_{i_m}$, as these are the only vertices whose order changes with respect to $<_f$.

    For each $s \in \tilde{\mathcal{F}}_<(G)_n$, let $f_s \colon \Delta^n \to G$ be a choice of singular simplex inclusion inducing $s$ in the sense of Definition~\ref{def:InducingSingularSimplex}. Then define $h \colon C_*(\mathcal{F}(G)) \to C^{\Delta,m}_*(G)$ by linearly extending
    \begin{align*}
        h(s) = \text{sgn}(\sigma_{<}^f) f_s
    \end{align*}
    which is a chain map for the same reasons as $g$ in the case of singular simplex inclusions.
    
    By construction, $g \circ h = \text{id}_{C_*(\tilde{\mathcal{F}(}G))}$.
    Therefore, it remains to show that $h\circ g$ is chain homotopic to the identity on $C^{\Delta,h}_*(G)$.

    Given singular simplex homomorphism $f\colon \Delta^n \to G$, define the subquiver $G_f$ of $G$ be have vertices
    \[
        V_{G_f} =
        \{ f(i) \: | \: i = 0,\dots,n \},
    \]
    edges
    \[
        E_{G_f} = \{ f'( i \to j) \: | \: f' \colon \Delta^n \to G \; \text{is a singular simplex inclusion inducing} \: s^f_< \: \text{and} \: 0 \leq i < j \leq n \}
    \]
    when $f$ is an inclusion, and
    \[
        E_{G_f} = \{ f( i \to j) \: | \: 0 \leq i < j \leq n \}
    \]
    otherwise.
    Since the total order $<_f$ restricted to the vertices $b_1,\dots,b_{n-m}$ in the image of $f$ that have no loop remains fixed regardless of the choice of singular simplex inclusion $f'$ inducing $s^f_<$, we obtain the following two facts.
    \begin{enumerate}[(1)]
        \item 
        As we assume $G = \tilde{\mathcal{R}}(G)$ has no multiple edges with a loop, the quiver $G_f$ contains no multiple edges.
        \item
        Any double edges in $G_f$ have a loop.
    \end{enumerate}
    Therefore, $G_f$ satisfies the conditions of Lemma~\ref{lem:ExtenededSimplexStrongLocalHContraction}, implying that $C^{\Delta,h}_*(G_{f})$ is acyclic.
    
    To each singular simplex homomorphism $f\colon \Delta^n \to G$ associate the sub-chain complex $C^{\Delta,h}_*(G_{f})$, to provide us with an acyclic carrier $\varphi$ on the basis of singular simplices in $C^{\Delta,h}_*(G)$.
    Moreover, both $h\circ g$ and $\text{id}_{C^{\Delta,h}_*(G)}$ are carried by $\varphi$.
    Therefore, $h\circ g$ and $\text{id}_{C^{\Delta,h}_*(G)}$ are chain homotopic by the acyclic carrier theorem (Theorem~\ref{thm:AcyclicCarriers}).
    
    It remains to check that naturality conditions in the statement of the theorem.
    To this end, let $f\colon \Delta^n \to G$ be a singular simplex homomorphism and let $\phi \colon G_1 \to G_2$ be a digraph homomorphism. Using the functionality of $H^{\Delta,h}_*(G)$ from proposition~\ref{prop:InitialFuctoriality}, the map $\phi$ induces chain map $\phi_{\#}\colon C^{\Delta,h}_*(G_1) \to C^{\Delta,h}_*(G_2)$ by linearly extending $\phi_\#(f)=f\circ \phi$. Similarly, we also have the chain map $\tilde{f}_{\{<_1,<_2\}}(\phi)$ defined in equation~\eqref{eq:PartialFlagOnmorphisms}.

    Consider the diagram of chain maps
    \begin{equation}\label{eq:ChainSquare}
        \xymatrixcolsep{5pc}\xymatrix{
        C_*^{\Delta,h}(G_1) \ar[d]^{g} \ar[r]^{\phi_\#} & C_*^{\Delta,h}(G_2) \ar[d]^{g} \\
        C_*(\tilde{F}_{<_1}(G_1)) \ar[r]^{\tilde{f}_{\{<_1,<_2\}}(\phi)} & C_*(\tilde{F}_{<_2}(G_2))}
    \end{equation}
    where $g$ is the chain map defined in equation~\eqref{eq:PartialQuotientChainMap}.
    
    Similarly to the proof of fuctoriality of the partial directed flag functor in Proposition~\ref{prop:PartialFlagFunctor}, the commutativity of the diagram above must be checked on all possible forms of the image of singular simplex homomorphism $f\colon \Delta^n \to G$ under the chain map $g\colon C_*^{\Delta,h}(G_1) \to C_*(\tilde{F}_{<_1}(G_1))$.

    When $f$ is not an inclusion, then $\phi \circ f$ is also not an inclusion and 
    \[
        (\tilde{f}_{\{<_1,<_2\}}(\phi) \circ g)(f)
        =
        (g \circ \phi_\#)(f) = 0.
    \]
    Hence, assume now that $f$ is an inclusion. Then when $n=0$, using the first part of equation~\eqref{eq:PartialFlagOnmorphisms}, we have that 
    \[
        (\tilde{f}_{\{<_1,<_2\}}(\phi) \circ g)(f)
        =
        \phi(f(0))
        =
        (g \circ \phi_\#)(f).
    \]
    We now consider all possibilities when $n = 1$. In this case $f(0)$ and $f(1)$ are distinct as $f$ is an inclusion.
    
    Firstly, when either $f(0)$ or $f(1)$ has a loop, then either $(\phi \circ f)(0)$ or $(\phi \circ f)(1)$ has a loop. Hence,
    \[
        g(f) =
        \begin{cases}
            \;\;\:\{ \tilde{\mathcal{F}}_{<_1}(f(0),f(1)) \} & \text{if} \: f(0) <_1 f(1)
            \\
            - \{\tilde{\mathcal{F}}_{<_1}(f(1),f(0))\} & \text{otherwise}
        \end{cases}
    \]
    and
    \[
        (g\circ \phi_\#)(f) =
        \begin{cases}
            \;\;\:\{ \tilde{\mathcal{F}}_{<_2}((\phi \circ f)(0),(\phi \circ f)(1)) \} & \text{if} \: (\phi \circ f)(0) <_2 (\phi \circ f)(1)
            \\
            -\{\tilde{\mathcal{F}}_{<_2}((\phi \circ f)(1),(\phi \circ f)(0))\} & \text{otherwise}.
        \end{cases}
    \]
    It now follows directly from the construction of $\tilde{f}_{\{<_1,<_2\}}(\phi)$, that $(\tilde{f}_{\{<_1,<_2\}}(\phi \circ g))(f)$ agrees with the second equation above
    by equations~\eqref{eq:PartialFlagOnmorphisms}

    Secondly, when neither $f(0)$ or $f(1)$ have a loop and neither $(\phi \circ f)(0)$ or $(\phi \circ f)(1)$ have a loop, then
    \[
        (g \circ \phi_\#)(f) = \{(e_{(\phi \circ f)(0\to 1)}\} = (\tilde{f}_{\{<_1,<_2\}}(\phi \circ g))(f).
    \]
    Finally, when neither $f(0)$ or $f(1)$ have a loop and either $(\phi \circ f)(0)$ or $(\phi \circ f)(1)$ has a loop, then
    \begin{align*}
        (g \circ \phi_\#)(f)
        & =
        \begin{cases}
            \;\;\: \{ \tilde{\mathcal{F}}_{<_2}(\phi(f(0)),\phi(f(1))) \}
            &  \phi(f(0)) <_2 \phi(f(1))
            \\
            - \{ \tilde{\mathcal{F}}_{<_2}(\phi(f(1)),\phi(f(0))) \}
            & \: \text{otherwise}.
        \end{cases}
        \\ & =
        (\tilde{f}_{\{<_1,<_2\}}(\phi \circ g))(f)
    \end{align*}
    
    For $n\geq 1$,
    using the same reasoning as the final paragraph of the proof of Proposition~\ref{prop:PartialFlagFunctor}, commutativity of diagram~\eqref{eq:ChainSquare} on generators follows from the $n=1$ case.
    
    Applying the homology functor to diagram~\eqref{eq:ChainSquare}, we induce the commutative diagram on homology in the statement of the theorem. Using the commutativity of the diagram in homology and Proposition~\ref{prop:PartialFlagFunctor}, the fuctoriality of the partial flag functors extends $g_*$ to a natural transformation between functors $H_*^{\Delta,h}$ and $H_* \circ \tilde{f}_{\{<_G\}_{G\in \mathcal{C}}}$. In addition, invariance of partial flag functors with respect to the choices of loop maximal total orders follows also from the commutative diagram in homology, as any choice of such an order commutes up to isomorphism $g_*$ with homomorphism $\phi_*$. Therefore, any choices of loop maximal total order must result in the same homomorphisms of graded modules between any pair of quivers, as required.
\end{proof}

To end the section, we apply Theorem~\ref{thm:PartialSimplicalGraphHomologies} to demonstrate that the $H_*^{\Delta,h}$ is invariant with respect to the weak local strong $h$-homotopy of quivers given in Definition~\ref{def:QuiverLocalStrongHomotopy}, differing from local strong $h$-homotopy by additionally allowing loop contractions in the sense of Definition~\ref{def:WeakLocalStrongHHomotopy}.

\begin{theorem}\label{thm:LocalStrongHInv}
    The functor $H^{h,\Delta}_*$ is invariant under weak local strong $h$-homotopy equivalence of quivers.
\end{theorem}

\begin{proof}
    Suppose that quivers $G_1$ and $G_2$ are weak local strong $h$-homotopy equivalent. By construction of weak local strong $h$-homotopy, it is sufficient to prove the statement of the theorem in the case $G_1 \simeq^{wlSh}_1 G_2$.
    
    Using Theorem~\ref{thm:LocalStrongHInvIntial}, we have that $H^{h,\Delta}_*(G_1) = H^{h,\Delta}_*(G_2)$ when the $1$-step homotopy arises from  $\phi \colon G_1 \to G_2$, $\varphi \colon G_2 \to G_1$ such that $\varphi \circ \phi \simeq_1^{lSh} \text{id}_{G_1}$, $\phi \circ \varphi \simeq_1^{lSh} \text{id}_{G_2}$.
    Therefore, without loss of generality it is sufficient to show that $H^{h,\Delta}_*(G_1) = H^{h,\Delta}_*(G_2)$ when $G_2$ is a loop contraction of $G_1$.
    Moreover, we need only consider the case when a single degenerate loop $l$ at some $v_l \in V_G$ is removed, since the same conditions apply equally to all degenerate loops simultaneously.

    Let $<$ be loop maximal total order on $G_1$. Then $<$ is also a loop maximal total order on $G_2$.
    We will construct chain maps $g \colon C_*(\tilde{\mathcal{F}}_{<}(G_1))\to C_*(\tilde{\mathcal{F}}_{<}(G_2))$ and $h \colon C_*(\tilde{\mathcal{F}}_{<}(G_2))\to C_*(\tilde{\mathcal{F}}_{<}(G_1))$ satisfying $h \circ g = \text{id}_{C_*(\tilde{\mathcal{F}}_{<}(G_1))}$ and $g\circ h = \text{id}_{C_*(\tilde{\mathcal{F}}_{<}(G_2))}$. This implies that $C_*(\tilde{\mathcal{F}}_{<}(G_2))$ and $C_*(\tilde{\mathcal{F}}_{<}(G_2))$ are chain homotopic.
    Therefore, by Theorem~\ref{thm:PartialSimplicalGraphHomologies}, we obtain that
    \[
        H_*^{\Delta,h}(G_1)
        =
        H_*(\tilde{\mathcal{F}}_{<}(G_1))
        =
        H_*(\tilde{\mathcal{F}}_{<}(G_2))
        =
        H_*^{\Delta,h}(G_2)
    \]
    as desired.

    Define $g \colon C_*(\tilde{\mathcal{F}}_{<}(G_1))\to C_*(\tilde{\mathcal{F}}_{<}(G_2))$ and $h \colon C_*(\tilde{\mathcal{F}}_{<}(G_))\to C_*(\tilde{\mathcal{F}}_{<}(G_1))$ on vertices by linearly extending the identity map on vertices and on $n$-simplices for $n\geq 1$ as follows.
    
    Given $s_1 \in \tilde{F}(G_1)_n$ denote by $f_{s_1}$ a choice singular $n$-simplex inclusion in $G_1$ inducing $s_1$. Since $f_{s_1}$ is an inclusion it contains no loops in is image. Hence, $f_{s_1}$ is also a singular simplex inclusion with image in $G_2$.
    As any singular simplex inclusion induces some $n$-simplex in $\tilde{F}(G_1)_2$, $f_{s_1}$ induces simplex  $s_2 \in \tilde{\mathcal{F}}(G_2)_n$. 
    Crucially, since $l$ is a degenerate loop, being adjacent to no multiple edges and any double edges it is adjacent to having a loop in $G_2$, any choice of $f_{s_1}$ provides the same simplex $s_2$. Therefore, the simplex $s_2$ is uniquely determined and we construct $g$ by linearly extending $g(s_1) = s_2$. The construction of $h$ on $n$-simplices for $n\geq 1$ is achieved similarly with the roles of $\tilde{F}(G_1)_n$ and $ \tilde{F}(G_2)_n$ reversed.
    
    It follows from the uniqueness of the construction of $n$-simplices above that $(g\circ h)(s_1) = s_1$ for any $s_1 \in  \tilde{F}(G_1)_1$ and $(g\circ h)(s_2) = s_2$ for any  $s_2 \in  \tilde{F}(G_1)_2$.
    Therefore, $g$ and $h$ are chain maps satisfying $h \circ g = \text{id}_{C_*(\tilde{\mathcal{F}}_{<}(G_1))}$ and $g\circ h = \text{id}_{C_*(\tilde{\mathcal{F}}_{<}(G_2))}$, as required.
\end{proof}

\bibliographystyle{amsplain}
\bibliography{References}

\appendix

\section{Algorithmic implementations}\label{sec:Algorithms}

In this appendix we detail algorithms that obtain from a filtered quiver a filtered $\Delta$-set corresponding to the images of the directed flag, reduced directed flag, and partial directed flag complexes. To investigate the efficiency of the algorithms we consider the computational complexity of each procedure. Through this analysis, we demonstrate that dependence on the number of $n$-simplices in any give dimension makes computation of the reduced directed flag and partial reduced directed flag complexes significantly more efficient than direct calculation of generators of the quasi-isomorphic chain complexes $C^{\Delta,m}_*(G)$ and $C^{\Delta,h}_*(G)$, respectively. A working python demonstration of all algorithms detailed is available at \cite{Burfitt2025}.

Each proposed algorithm is structured with a main loops over the vertices $v\in V_G$. Therefore, all procedures can be parallel processed by assigning computations over each vertex $v\in V_G$ to separate processing units. Moreover, the algorithms are designed to remain memory efficient, requiring only information on certain $n$-simplices and vertices of $G$ at any given step. Once the filtered $\Delta$-sets are obtained, the results can be combined with \cite{Smith2022} or \cite{Lutgehetmann2021} to compute persistent homology. Providing efficient computation of persistent $H^{\Delta,i}_n$, $H^{\Delta,m}_n$, and $H^{\Delta,h}_n$ on filtered quivers.

For the computational purposes of this appendix, an abstract simplicial complexes comes equipped with a total order on its vertices. In addition, quivers, simplicial complexes, and $\Delta$-sets are assumed to be finite. Throughout the appendix all indices begin at $0$, and as throughout this work $G$ is a quiver and $n$ a non-negative integer, unless stated otherwise.

\subsection{Filtered objects}

For any given quiver homology, additional information about the structure of the quiver with respect to an associated filtration can be extracted using persistent homology, the construction of which we detail in Appendix~\ref{PersistentHomology}.

\begin{definition}\label{def:FilteredQuiver}
    A \emph{filtered quiver} is a quiver $G$ together with real valued functions $f_V\colon V_G \to \mathbb{R} \cup \{ -\infty,\infty \}$ and $f_E\colon E_G \to \mathbb{R} \cup \{ -\infty,\infty \}$ such that $f_E(e)\geq f_V(s(e))$ and $f_E(e)\geq f_V(t(e))$ for each $e \in E_G$. 
\end{definition}

Given $t \in \mathbb{R} \cup \{\infty \}$, the \emph{sublevel quivers} $G^t$ of a filtered quiver $G$ are given by $V_{G^t}=f_V^{-1}([-\infty,t))$ and $E_{G^t} = f_E^{-1}([-\infty,t))$. It is the changes in the homology of the increasing family of sublevel quivers that is measured by persistent homology.

\begin{definition}\label{def:FilteredSpace}
    A \emph{filtered simplicial complex} is an abstract simplicial complex $(S,V)$ together with a real valued function $f\colon S \to \mathbb{R} \cup \{ -\infty,\infty \}$ such that for any $s,s' \in S$ with $s \subseteq s'$ we have $f(s) \leq f(s')$.
    
    A \emph{filtered $\Delta$-set} is a $\Delta$-set $X$ together with real valued functions $f_n\colon X_n \to \mathbb{R} \cup \{ -\infty,\infty \}$ such that $f_{n+1}(x_{n+1})\geq f_{n}(d_i^n(x_{n+1}))$ for each integer $n \geq 0$, $i=0,\dots,n+1$, and $x_{n+1} \in X_{n+1}$. 
\end{definition}

The \emph{sublevel simplicial complexes} $(V^t,S^t)$ of a filtered simplicial complex are given by $S^t=f^{-1}([-\infty,t))$ and $V^t = \{ v \in V \: | \: \{ v \} \in S^t \}$ for each $t \in \mathbb{R} \cup \{\infty \}$. The \emph{sublevel $\Delta$-sets} $X^t$ of a filtered $\Delta$-set $X$ are given by $X_n^t=f_n^{-1}([-\infty,t))$ for each $t \in \mathbb{R} \cup \{\infty \}$ and integer $n\geq 0$.

As all quivers, simplicial complexes, and $\Delta$-sets considered in this section are finite, there are only finitely many values of $t$ for which the sublevel quivers, simplicial complexes, or $\Delta$-sets of any such filtered object are distinct.

\subsection{Data structures}\label{sec:DataStructures}

Before detailing the quiver homology algorithms, we first set out the data structures used to store quivers, filtered quivers, simplices, filtered simplices, simplicial complexes, filtered simplicial complexes, $\Delta$-sets, and filtered $\Delta$-sets. Quivers were defined in Section~\ref{sec:Quivers}, and simplicial complexes and $\Delta$-sets in Section~\ref{sec:Spaces} earlier in this work.

For the purposes of setting out the flag complex $\mathcal{F}$ and partially reduced flag complex $\tilde{\mathcal{F}}$ procedures, we additionally require a notion of $k$-partial $(n+1)$-dimensional simplices within a $\Delta$-set for $k=0,\dots,n+1$. In particular, a $0$-partial $(n+1)$-simplex in $X$ is an $n$-simplex of $X$ and the boundary of an $(n+1)$-simplex provides an $(n+1)$-partial $(n+1)$-simplex in $X$. Formally, a $k$-partial $n$-simplex in $X$ consists of an sequential collection of $n$-simplices subject to a subset of the face conditions from equation~\eqref{eq:FaceMapConditions} and is defined precisely as follows.

\begin{definition}\label{def:PartialSimplex}
Let $X$ be a $\Delta$-set, $n\geq 0$ and $k=0,\dots,n+1$. Then a \emph{$k$-partial $(n+1)$-simplex} in $X$ consists of $n$-simplices $s_{n+1-k},\dots,s_{n+1} \in X_{n}$ satisfying
    \[
        d_{i}^{n-1} ( s_j ) = d_{j-1}^{n-1} ( s_i )
    \]
    for each $i,j=n+1-k,\dots,n+1$ with $i < j$.
\end{definition}

During the of computation of the persistent homology of $\tilde{\mathcal{F}}(G)$, we additionally require cell complexes. A \emph{cell complex} generalises the construction of a $\Delta$-set, consisting of a sequence of sets of $n$-cells $C_n$ whose boundary consists of a $\mathbb{Z}$-linear combination of lower dimensional cells.

The data structures assigned to each object appearing in this appendix are as follows.
In particular, the computational structure for simplicial complexes, filtered simplicial complexes, $\Delta$-sets, and filtered $\Delta$-sets coincide with the inputs used in \cite{Smith2022} and \cite{Lutgehetmann2021}.

\begin{enumerate}[(1)]
    \item
    Let $(V,S)$ be an abstract simplicial complex. Then for algorithmic purposes, $(V,S)$ consists of a finite totally ordered set of vertices and a finite totally ordered set $S$ of simplices consisting of subsets of vertices $V$ closed under taking subsets.
    \begin{itemize}
        \item[--]
        In practice, $S$ is stored as a list of lists, where each list at any given index $n$ contains lists of length $n+1$. Each length $n+1$ list being a list of the ordered vertices of an $n$-simplex in $(V,S)$. Therefore, $V$ is stored as the index $0$ components of $S$. However, we note that it will only be strictly necessary to hold consecutive sets of $n$ and $(n+1)$-simplices in memory for each computational step of the algorithms set out in this section.  
    \end{itemize}
    \item
    A filtered simplicial complex consists of the structure of a simplicial complex provided above with each simplex additionally being assigned a filtration value in $\mathbb{R} \cup \{ -\infty,\infty\}$ satisfying the conditions of Definition~\ref{def:FilteredSpace}. In practice, the filtration values of the filtered simplicial complex are stored as an additional final entry in each simplex list.
    \item
    Let $X$ be a $\Delta$-set. Then for algorithmic purposes, $X$ consists of a finite totally ordered set of finite totally ordered sets $X_n$ of $n$-simplices in each dimension $n$, up to the maximal dimension.
    \begin{itemize}
        \item[--]
        In practice, a $\Delta$-set $X$ is a list of lists, with each inner list corresponding to $X_n$ in each index $n$. Each $X_n$ list consist of a list of $n$-simplices. The $n$-simplices being a list of position indexes in the $X_{n-1}$ list ordered by their image under face maps $d^{n-1}_i$ for $i=0,\dots,n$. However, we note that it will only be strictly necessary to hold consecutive $X_n$ and $X_{n+1}$ in memory for each computational step of the algorithms set out in this section.
    \end{itemize}
    \item 
    A filtered $\Delta$-set consists of the structure of a $\Delta$-set provided above with each $n$-simplex additionally being assigned a filtration value in $\mathbb{R} \cup \{ -\infty,\infty\}$ satisfying the condition in Definition~\ref{def:FilteredSpace}. In practice, the filtration values of the filtered $\Delta$-set are stored as an additional final entry in each simplex list.
    \item
    We view a $k$-partial $(n+1)$-simplex of a $\Delta$-set $X$ similarly to an $(n+1)$-simplex, with the exception that we require only information on $d^n_{i}$ for $i=n+1-k,\dots,n+1$.
    \begin{itemize}
        \item[--]
        In practice, a $k$-partial $(n+1)$-simplex in $\Delta$-set $X$ consists of a list of $k$ indexes of $n$-simplices in the list corresponding to $X_n$. When $X$ is filtered, there is also an additional final list entry containing the filtration value coinciding with the maximal filtration value among the elements of $X_n$ indexed in the $k$-partial $(n+1)$-simplex.
    \end{itemize}
    \item 
    Let $C$ be a cell complex. For algorithmic purposes, a cell complex $C$ is a generalisation of a $\Delta$-set consisting of a finite totally ordered set of finite totally ordered sets $C_n$ of $n$-cells in each dimension $n$, up to the maximal dimension. Since we work with $\mathbb{Z}_2$ coefficients no coefficient information for each face will be required.
    \begin{itemize}
        \item[--]
        In practice, a cell complex $C$ consists in each dimension $n$ of a list $C_{n}$ of arbitrary length lists of indices of cells in list $C_{n}$ forming the boundary of the cell.
        However, we note that it will only be  necessary to hold consecutive $C_n$ and $C_{n+1}$ in memory for each computational step of the algorithms set out in this section. 
    \end{itemize}
    \item
    A filtered cell complex consists of the structure of a cell complex provided above with each $n$-cell additionally being assigned a filtration value in $\mathbb{R} \cup \{ -\infty,\infty\}$ greater than or equal to any of cell in its boundary. In practice, the filtration values of each filtered cell is stored as an additional final entry in the cell list.
\end{enumerate}

For filtered quivers, simplicial complexes, and $\Delta$-sets, we may identify the non-filtered object with a filtered objects in which all filtration value are set to be equal to $-\infty$. For simplicity, from now on we always assume that the unfiltered objects specified above correspond to filtered objects of such type.

\subsection{Computing persistent homology}\label{PersistentHomology}

Persistent homology can be obtained with recept to coefficient in any field. Let $(C_*,\partial_*)$ be a chain complex. Then the \emph{boundary matrix} of the differential $\partial_n \colon C_{n} \to C_{n-1}$ with recept to chosen totally ordered bases of $C_{n}$ and $C_{n-1}$ is the usual matrix representing the linear maps $\partial_n$. Each column being the image of each $C_{n}$ basis element expressed as a linear combination of the $C_{n-1}$ basis.

For a $\Delta$-set $X$ or similarly a cell complex more generally, the simplices $X_n$ in each dimension provide a canonical basis of each $C_n(X)$. Moreover, as made precise in Section~\ref{sec:Spaces}, abstract simplicial complexes can also be considered as $\Delta$-sets once a total order is chosen on their vertices.

To compute persistent homology up some dimension $n$, we form a \emph{filtered boundary matrix} with a row for every $k$-simplex and whose columns are in bijection with all the columns of boundary matrices $C_k(X)$ for $k=1,\dots,n+1$ and ordered by the filtration values of the corresponding simplices. Alongside the filtered boundary matrix we additionally record a vector containing the filtration values of each column. Once the filtered boundary matrix is obtained, standard procedures can be applied to efficiently obtain the persistent homology \cite{Bauer2021, Lutgehetmann2020}.

In practice, for speed of computation, persistent homology is almost always computed with respect to $\mathbb{Z}_2$ coefficients. Specifically, there are no signs to consider and operations can be made efficiently using binary arithmetic. Moreover, again for reasons of computation seed, the cohomology is usually computed rather then the homology \cite{Bauer2021}. However, the coboundary matrix can easily be obtained from the boundary matrix by taking the transpose. Therefore, for simplicity we discuses only bounder matrices in the reminder of the appendix and assume that all coefficients lie in $\mathbb{Z}_2$.

\subsection{Directed flag complex}\label{sec:AlgDirectedFlagComplex}

We now present an algorithm for computing the directed flag complex $\mathcal{F}(G)$ of a filtered quiver $G$. The directed flag complex was originally presented in Definition~\ref{def:DirectedFlag}. In the case of digraphs, an efficient algorithm for computing $\mathcal{F}(G)$ was provided in \cite{Lutgehetmann2020}. However, for the purposes of quivers more generally, we need to construct a more sophisticated procedure utilizing the structure of $\Delta$-sets not restricting ourselves to only abstract simplicial complexes. We note that during the computation of $(n+1)$-simplices from $n$-simplices, the algorithm presented keeps the simplices of the directed flag complex separated into lists of those obtained by extending an $n$-simplex by the same maximal vertex with respect to the total order induced by the corresponding singular simplex inclusion $\colon \Delta^{n} \to G$.

Given a filtered quiver $G$, the steps of the algorithm for acquiring the filtered $\Delta$-set $\mathcal{F}(G)$ are as follows. As there is no inclusion of quivers from $\Delta^1$ onto a loop, we may additionally assume that $G$ has no loops throughout the procedure.
\begin{enumerate}[(1)]
    \item 
    Remove all loops from $G$, set the vertices of $\mathcal{F}(G)$ to be the vertices of $G$ and the edges of $\mathcal{F}(G)$ to be the edges of $G$.
    \item 
    Obtain the set of $1$-simplices $\mathcal{F}(G)_1^v$ with maximal vertex $v$ for each $v \in V_G$ as the sets of edges $e\in E_G$ such that $t(e) = v$.
    \item 
    For each $n \geq 1$ and $v\in V_G$, obtain the $(n+1)$-simplices $\mathcal{F}(G)_{n+1}^v$ of ${\mathcal{F}}(G)$ with maximal vertex $v\in V_G$ by considering in turn each $n$-simplex $s_{n+1}^{n} \in {\mathcal{F}}(G)_n$. Given an $n$-simplex $s_{n+1}^{n}$, an element of $\mathcal{F}(G)_{n+1}^v$ is obtained from any sequence of $n$-simplices $s_0^{n},\dots,s_{n+1}^{n} \in \mathcal{F}(G)_{n}^v$ such that $s_0^{n},s_1^{n},\dots,s_{n+1}^{n}$ provides a well defined boundary for an $(n+1)$-simplex satisfying equation~\eqref{eq:FaceMapConditions}. The filtration value of each new simplex is the maximum of the filtration values of $s_0^{n},\dots,s_{n+1}^{n}$. Further details on this step of the procedure are provide in Algorithm~\ref{alg:DirectedFlagComplex}.
    \item 
    The procedure terminates when, either each $S^v_{n+1}$ is empty or the maximal desired $(n+1)$-skeleton of $\mathcal{F}(G)$ for computation of persistent homology up to dimension $n$ has been obtained.
\end{enumerate}
To obtain each $\mathcal{F}(G)_n$ in full, set
\[
    \mathcal{F}(G)_n = \cup_{v \in V_G} \mathcal{F}(G)_n^v.
\] 

The precise procedure for the central step (step (3) above) is detailed below in Algorithm~\ref{alg:DirectedFlagComplex}. In particular, the extensions of $s_{n+1}^{n}$ by a vertex $v$ to an $(n+1)$-simplex is further broken down into an inductive procedure detailed separately in Algorithm~\ref{alg:ExtendSimplex}. This procedure, begins with $s_{n+1}^{n}$ as a $0$-partial $(n+1)$-simplex. Then at each inductive step, all possibilities to extend to $k$-partial $(n+1)$-simplices for $k=1,\dots,n+1$ are considered. More precisely, the $(k+1)$-partial $(n+1)$-simplices are constructed by identifying each possible subsequent $s^n_{n-k} \in \mathcal{F}(G)_n^v$ satisfying
\begin{equation}\label{eq:PartailSimplexRelations}
    d_{n-k}^{n-1}(s_{j}^n) = d_{j-1}^{n-1}(s_{n-k}^n)
\end{equation}
for $j=n-k+1,\dots,n+1$. The entire sequence of faces satisfies equation~\eqref{eq:FaceMapConditions} through compatibility with the previously constructed simplices of the $k$-partial $(n+1)$-simplex. The procedure terminates after the the $k=n$ induction step. 

\begin{algorithm}
    \caption{Algorithm to be recursive applied begging with an $n$-simplex $s_{n+1}^n \in \mathcal{F}(G)_n$ and vertex $v \in V_G$ of a quiver $G$ in order to obtain all $(n+1)$-simplices of the flag complex $\mathcal{F}(G)$ with $s_{n+1}^n$ as a face and maximal vertex $v$ when induced by the corresponding singular simplex inclusion $\colon \Delta^{n+1} \to G$. This is achieved by sequentially constructing $(k+1)$-partial $(n+1)$-simplices from each $k$-partial $(n+1)$-simplex with image $s_{n+1}^n$ under $d^n_0$ and maximal vertex $v$ for $k=1,\dots,n+1$.
    All data structures used are set out in Section~\ref{sec:DataStructures} and simplices are assumed to be assigned a filtration value of $-\infty$ by default if not otherwise specified.}
    \label{alg:ExtendSimplex}
    \begin{algorithmic}
        \Input
            \Desc{$n$}{\;\;\;\;\;\;\;\;\;\;\;\;\; integer dimension of input simplices greater than $0$}
            \Desc{$k$}{\;\;\;\;\;\;\;\;\;\;\;\;\; integer greater than or equal to zero and less than $n+1$}
            \Desc{$\mathcal{F}(G)_n^v$}{\;\;\;\;\;\;\;\;\;\;\;\;\; set of $n$-simplices in $\mathcal{F}(G)$ with maximal vertex $v$}
            \Desc{$\mathcal{F}(G)_{n+1,k+1}^v$}{\;\;\;\;\;\;\;\;\;\;\;\;\; set of $k$-partial $(n+1)$-simplices in $\mathcal{F}(G)$ with maximal vertex $v$}
        \EndInput
        \Output
            \Desc{$\mathcal{F}(G)_{n+1,k+1}^v$}{\;\;\;\;\;\;\;\;\;\;\;\;\;\;\;\; set of $(k+1)$-partial $(n+1)$-simplices in $\mathcal{F}(G)$ with maximal vertex $v$}
        \EndOutput
        \Procedure{Extend to simplex}{$\mathcal{F}(G)_n^v$, $\mathcal{F}(G)_{n+1,k+1}^v$, $n$, $k$}
            \State $\mathcal{F}(G)_{n+1,k+1}^v \gets \emptyset$
            \For{each $S = \{ s^n_{n-k+1},\dots,s^n_{n+1} \} \in \mathcal{F}(G)_{n+1,k+1}^v$}
                \For{each $s^n_{n-k}\in \mathcal{F}(G)_n^v$}
                    \If{$d_{n-k}^{n-1}(s^n_j)=d^{n-1}_{j-1}(s^n_{n-k})$ for $j=n-k+1,\dots,n+1$}
                        \State $\mathcal{F}(G)_{n+1,k+1}^v \gets \mathcal{F}(G)_{n+1,k+1}^v \cup \{S \cup \{ s^n_{n-k} \}\}$
                        \State(where $\{S \cup \{ s^n_{n-k} \}\}$ is assigned the maximum filtration value of $S$ and $s^n_{n-k}$)
                    \EndIf
                \EndFor
            \EndFor
            \Return $\mathcal{F}(G)_{n+1,k+1}^v$
        \EndProcedure
    \end{algorithmic}
\end{algorithm}

\begin{algorithm}
    \caption{Algorithm for obtaining the simplices of the directed flag complex $\mathcal{F}(G)$ of a quiver $G$ in dimension $n+1$ from those in dimension $n$. In particular, the procedure makes use of the \emph{Extend simplex} function provided in Algorithm~\ref{alg:ExtendSimplex}. For the purposes of computation of the next dimensional simplices, simplices are partitioned by their maximal vertex with respect to the total order on the their vertices induced by the corresponding singular simplex inclusion $\colon \Delta^{n} \to G$. All data structures used are set out in Section~\ref{sec:DataStructures} and simplices are assumed to be assigned a filtration value $-\infty$ by default if not otherwise specified.}
    \label{alg:DirectedFlagComplex}
    \begin{algorithmic} 
        \Input
            \Desc{$n$}{\;\;\;\;\;\;\;\;\;\;\;\;\;\;\;\;\;\; integer dimension of input simplices greater than $1$}
            \Desc{$V_G$}{\;\;\;\;\;\;\;\;\;\;\;\;\;\;\;\;\;\; vertex set of quiver $G$}
            \Desc{$\{\mathcal{F}(G)_n^v\}_{v\in V_G}$}{\;\;\;\;\;\;\;\;\;\;\;\;\;\;\;\;\;\; set of sets of $n$-simplices in $\mathcal{F}(G)$ with greatest vertex $v\in V_G$}
        \EndInput
        \Output
            \Desc{$\{\mathcal{F}(G)_{n+1}^v\}_{v\in V_G}$}{\;\;\;\;\;\;\;\;\;\;\;\;\;\;\;\;\;\;\;\;\; sets of sets of $(n+1)$-simplices in $\mathcal{F}(G)$ with greatest vertex $v\in V_G$}
        \EndOutput
        \Procedure{Directed flag complex}{$V_G$, $\{\mathcal{F}(G)_n^v\}_{v\in V_G}$, $n$}
            \For{each $v \in V_G$}
                \State $\mathcal{F}(G)_{n+1}^v \gets \emptyset$
                \For{each $u \in V_G \setminus \{v\}$}
                    \For{each $s_{n+1}^n \in \mathcal{F}(G)_{n}^u$}
                        \State $S^{v}_{n+1} \gets \{ \{s^n_{n+1}\} \}$ (with the filtration value of $s_{n+1}^n$)
                        \For{$k=0,\dots,n$}
                            \State $S^{v}_{n+1} \gets$ \Call{Extend to simplex}{$\mathcal{F}(G)_n^v$, $S^{v}_{n+1}$, $n$, $k$}
                        \EndFor
                    \EndFor
                    $\mathcal{F}(G)_{n+1}^v \gets \mathcal{F}(G)_{n+1}^v \cup S^{v}_{n+1}$
                \EndFor
            \EndFor
            \Return $\{\mathcal{F}(G)_{n+1}^v\}_{v\in V_G}$
        \EndProcedure
    \end{algorithmic}
\end{algorithm}

We remark that as each simplex is stored as the indices of the simplices in one dimension lower, the filtered boundary matrix can be obtained immediately from the data structure of the $\Delta$-set. This means that filtered boundary matrix could be computed directly during the construction of $\mathcal{F}(G)$ without the need for an additional computational step.

It should be also noted that in the case of digraphs in dimension $1$, the procedure proposed above coincides with that of \cite{Lutgehetmann2020}. However, in higher dimensions under conditions when there are considerably less simplices than edge of $G$, it would be more efficient to use the procedure above as opposed to the one from \cite{Lutgehetmann2020}.

\subsubsection{Computational complexity}\label{sec:ComplexityDirectedFlag}

The time complexity of step (1) of the computation of the directed flag complex $\mathcal{F}(G)$ detailed above is $\mathcal{O}(|V_G|)$ and step (2) is $\mathcal{O}(|E_G|)$. The part of the procedure with the greatest time complexity is step (3), which can be described as follows.

Recall that we denote by $\mathcal{F}(G)_n$ the set of $n$-dimensional simplices in $\mathcal{F}(G)$. We note that $|\mathcal{F}(G)_1| \leq |E_G|$, as $E_G$ can contain loops. Moreover, for each $v\in V_G$ we denote by $\mathcal{F}(G)_n^v$ the set of $n$-dimensional simplices in $\mathcal{F}(G)$, whose greatest vertex with respect to the total order on the their vertices induced by the corresponding singular simplex inclusion $\colon \Delta^{n} \to G$, is $v$. In particular, $|\mathcal{F}(G)_n|=\sum_{v\in V_G}|\mathcal{F}(G)_n^v|$.

The initial spitting of the edges of $G$ by terminal vertices has time complexity $\mathcal{O}(|V_G||\mathcal{F}(G)_1|)$. For the computation of $(n+1)$-simplices of $\mathcal{F}(G)$ for $n\geq 1$ at each iteration of step (3), we initially consider each pair of a vertex $v\in V_G$ and $x\in \mathcal{F}(G)_{n}$. For each such pair, we then check the compatibility of simplices in $\mathcal{F}(G)_n^v$ in turn to obtain $k$-partial $(n+1)$-simplices in sequence. In particular, there are at most 
$\frac{(n+1)(n+2)}{2}$
steps in total to verify the conditions in equation~\eqref{eq:PartailSimplexRelations} for each $(n+1)$-simplex obtained. In addition, the number of $k$-partial $(n+1)$-simplices is bounded above by $|\mathcal{F}(G)_n^v|\cdots \left(|\mathcal{F}(G)_n^v|-k+1 \right)$. Therefore, the total time complexity is bounded above by
\begin{equation}\label{eq:FlagComplexComplexity}
    \mathcal{O}\left(
    \frac{(n+1)(n+2)}{2}
    |\mathcal{F}(G)_{n}| \sum_{v\in V_G}
    \max(1,|\mathcal{F}(G)_n^v|)
    \sum_{k=0}^{n} \prod_{i=1}^k \max(1,|\mathcal{F}(G)_n^v|-i+1)  \right).
\end{equation}
However, we note that the $|\mathcal{F}(G)_n^v|\cdots(|\mathcal{F}(G)_n^v|-k+1)$ bound on the number of $k$-partial $(n+1)$-simplices is not sharp and could be greatly improved. For example, if $s_1\neq s_2 \in \mathcal{F}(G)_{n}^v$ and $d^{n-1}_0(s_1)=d^{n-1}_0(s_2)$, then $s_1$ and $s_2$ cannot both be contained in the same $k$-partial $(n+1)$-simplex for any $k=0,\dots,n+1$. Nevertheless, the term demonstrates the heavy dependence of the time complexity on the number of simplices in each dimension, particularly as the dimension grows.

\subsection{Reduced directed flag complex}\label{sec:AlgReducedDirectedFlagComplex}

In this subsection we detail an algorithm for computing the reduced directed flag complex $\bar{\mathcal{F}}(G)$ of a  filtered quiver $G$. The reduced directed flag complex was originally presented in Definition~\ref{def:ReducedDirectedFlagComplex}. By Proposition~\ref{prop:StrongMultiEdgeContraction}, to obtain the abstract simplicial complex $\bar{\mathcal{F}}(G)$ we need only consider the reduced digraph $\bar{\mathcal{R}}(G)$ of a quiver $G$. That is we may first obtain a digraph from $G$ by removing all loops and duplicate multiple edges. The algorithm presented is similar to that from \cite{Lutgehetmann2020} for the directed flag complex $\mathcal{F}(G)$ of a digraph, with the additional step of remove duplicate simplices on the same set of vertices after computing each set of $n$-simplices.

Given a filtered quiver $G$, the steps of the algorithm for acquiring the filtered simplicial complex $\bar{\mathcal{F}}(G)$ are as follows.
\begin{enumerate}[(1)]
    \item 
    Obtain $\bar{\mathcal{R}}(G)$ from $G$ by removing loops and reducing all multiple edges to a single edge, keeping the minimal filtration value among multiple edges on the same ordered pair of vertices.
    \item
    Set the $0$-simplices of $\bar{\mathcal{F}}(G)$ to be the vertices $V_G$ of $G$. Add a $1$-simplex to $\bar{\mathcal{F}}(G)$ between vertices $u,v \in V_G$ whenever there is an edge $u \to v \in E_G$ or $v \to u \in E_G$. With the filtration value of the edge being the minimal filtration value among edges $u\to v$ or $v \to u$.
    \item 
    For each $n \geq 1$, obtain the $(n+1)$-simplices of ${\mathcal{F}}(G)$ by considering in turn each vertex $v\in V_G$ and each $n$-simplex $s_n$ from ${\mathcal{F}}(G)$ as follows. If there is a directed edge in $\bar{\mathcal{R}}(G)$ from each vertex of $s_n$ to $v$, then add the simplex extending $s_n$ by $v$ to the $(n+1)$-simplices of ${\mathcal{F}}(G)$. The filtration value of the new simplex is the maximum of the filtration values among $s_n$ and the edges from vertices of $s_n$ to $v$.
    \item
    Remove any duplicate $n$-simplices from the filtered simplicial complex $\bar{\mathcal{F}}(G)$ retaining the lowest filtration values.
    \item 
    The inductive step of the procedure terminates when, either there are no new $(n+1)$-simplices added to $\bar{\mathcal{F}}(G)$ or the maximal desired $(n+1)$-skeleton of $\bar{\mathcal{F}}(G)$ for computation of persistent homology up to dimension $n$ has been obtained.
    \item
    Remove any duplicate $(n+1)$-simplices from the filtered simplicial complex $\bar{\mathcal{F}}(G)$ retaining the lowest filtration values.
\end{enumerate}
We note that due to the fact that an $(n+1)$-simplex of $\bar{\mathcal{F}}(G)$ on the same set of vertices might not be obtainable from any particular $n$-simplex and vertex $v\in V_G$, the duplicate $n$-simplices cannot be removed until all $(n+1)$-simplices have been obtained. The precise algorithm for the central step (step (3) above) of the produce is detailed in \cite{Lutgehetmann2020}.

\subsubsection{Computational complexity}\label{sec:ComplexityReducedFlag}

Recall that we denote by $\mathcal{F}(G)_n$ the set of $n$-dimensional simplices in $\mathcal{F}(G)$. Moreover, for each $v\in V_G$ we denote by $\mathcal{F}(G)_n^v$ the set of $n$-dimensional simplices in $\mathcal{F}(G)_n$, whose greatest vertex with respect to the total order on the their vertices induced by the corresponding singular simplex inclusion $\colon \Delta^{n} \to G$, is $v$. In particular, $|\mathcal{F}(G)_n| = \sum_{v\in V_G}|\mathcal{F}(G)_n^v|$.

The time complexity of steps (1) and (2) in the computation of the reduced directed flag complex $\bar{\mathcal{F}}(G)$ detailed above is $\mathcal{O}(|E_G|^2)$. Steps (4) and (6) have time complexity $\mathcal{O}(|\mathcal{F}(\bar{\mathcal{R}}(G))^2_n|)$ in any given dimension $n$. The part of the procedure with typically the greatest time complexity is step (3), which can be described as follows.

Given $v\in V_G$, denote by $E_v^t$ the set of edge $e\in \bar{\mathcal{R}}(G)$ such that $t(e) = v$. Initially spitting edges into sets $E_v^t$ has time complexity $\mathcal{O}(|V_G||\mathcal{F}(G)_1|)$. For the computation of $n$-simplices of $\mathcal{F}(\bar{\mathcal{R}}(G))$ at each iteration of step (3), we initially consider each pair of a vertex $v\in V_G$ and $n$-simplex $x\in \mathcal{F}(\bar{\mathcal{R}}(G))_{n-1}$. For each such pair, we then check the compatibility of the simplices with edges $E_v^t$ in turn, for which there are at most $n$ steps in total to check the existence of an edges from a vertex of the simplex to $v$. Therefore, the total time complexity is bounded above by
\begin{equation}\label{eq:ReducedFlagComplexity}
    \mathcal{O}\left(n |\mathcal{F}(\bar{\mathcal{R}}(G))_{n}| \sum_{v\in V_G}\max(1, |E^t_v|)  \right).
\end{equation}
Equation~\eqref{eq:ReducedFlagComplexity} can be applied to demonstrate the significant improvement realised when obtaining $H_*^{\Delta,m}(G)$ as the homology of $\bar{\mathcal{F}}(G)$. An algorithm for providing directly the generators of $C_*^{\Delta,m}(G)$ can be obtained from the procedure above by modifying step (3) to check if there is a directed edge in $G$ rather than $\bar{\mathcal{R}}(G)$ from each vertex of $s_n$ to $v$ or that the vertex of $s_n$ is equal to $v$. That is, as opposed to just checking for the existence of the edge.

In this case, due to additionally allowing $s_n$ vertices that are not distinct, within each $n$-simplex of $\mathcal{F}(G)_i$ there are an additional $\binom{n-1}{n-i}$ $i$-simplices lying within the $n$-simplex for $i=0,\dots,n$. Given $v \in V_G$, denote by $C_n^{\Delta,m}(G)_v$ the submodule of $C_n^{\Delta,m}(G)$ whose basis of singular $n$-simplices have greatest vertex $v$ with respect to the total order on the their vertices induced by the corresponding singular simplex $\colon \Delta^{n} \to G$. Then for each $v\in V_G$, we have
\begin{equation}\label{eq:MappingFlagComplexity}
    \text{rank}(C_n^{\Delta,m}(G)_v) = 1 + \sum^{n}_{i=1} \binom{n-1}{n-i} |\mathcal{F}(G)^v_i|
\end{equation}
where the additional first term $1$ corresponds to the singular $n$-simplex all whose vertices are $v$. When performing the direct $C_*^{\Delta,m}(G)$ computation we replace the number of $n$-dimensional simplices $\mathcal{F}(G)_{n}$ by $\text{rank}(C_n^{\Delta,m}(G)) = \sum_{v \in V_G} \text{rank}(C_n^{\Delta,m}(G)_v)$.
Therefore, as the computational complexity in equation~\eqref{eq:ReducedFlagComplexity} depends on a multiple of $|\mathcal{F}(\bar{R}(G))_{n}|$ the direct algorithm for generators of $C_*^{\Delta,m}(G)$ is prohibitively slower for any digraph or quivers that contains many cliques.

\subsection{Partial directed flag complex}

In this subsection we present as a combination of the procedures for the directed flag complex $\mathcal{F}(G)$ and reduced directed flag complex $\bar{\mathcal{F}}(G)$ provided in Sections~\ref{sec:AlgDirectedFlagComplex}~and~\ref{sec:AlgReducedDirectedFlagComplex}, an algorithm for computing the partial directed flag complex $\tilde{\mathcal{F}}(G)$ of a filtered quiver $G$. The construction of the partial directed flag complex was originally provided in Section~\ref{sec:PartialFlagComplex}. 

This is achieved by treating simplices in the flag complex $\mathcal{F}(G)$ as certain joins of a simplex in the full subquiver $G_{l}$ of $G$ on the vertices with loops and a simplex in the full subquiver $G_{nl}$ of $G$ on the vertices that do not have loops. The advantage of this strategy is that it enables the computation of $\tilde{\mathcal{F}}(G)$ to be reduced to the computation of the two sub-$\Delta$-sets corresponding to each full subquiver.

\begin{remark}\label{rmk:JoinSimplexStructure}
    More precisely, for the inductive stage of the algorithm the simplices are recorded as a triple $(S,S_{nl},S_{l})$ containing
    \begin{enumerate}[(i)]
        \item 
        $S$ the simplex as if it lay in $\mathcal{F}(G)$,
        \item
        $S_l$ the index of the sub-simplex of $S$ in $\mathcal{F}(G_{nl})$,
        \item
        and $S_{l}$ the vertices of the sub-simplex of $S$ on vertices with loops as if it were a simplex in $\bar{\mathcal{F}}(G_l)$.
    \end{enumerate}
\end{remark}

When computing the boundary matrix, the structure above allows us to detect duplicate simplices of $\tilde{\mathcal{F}}(G)$ in $\mathcal{F}(G)$ by checking when the non-loop simplex and loop simplex vertices agree.

Similarly to the use of $\bar{\mathcal{R}}(G)$ in the previous section, we may apply Theorem~\ref{thm:StrongHMultiEdgeContractionBetweenLoops} and reduce complexity by performing the procedure on the smaller quiver $\tilde{\mathcal{R}}(G)$ rather than $G$ directly. That is, we can first remove all duplicate multiple edges form $G$ that have a loop prior to applying the main part of the procedure.
 
The available efficient procedures for computing persistent homology, take a cell complex as input. Therefore, in order to be compatible we must ensure the output of the present algorithm takes this form. As the maximal subquiver of $G$ on vertices that have a loop changes with the sublevel sets of the filtration, we are required in certain situations to add addition filtered cells extending a subcomplex of $\mathcal{F}(G^t)$ to a cell complex with $\tilde{\mathcal{F}}(G^t)$ as a deformation retraction for each $t\in \mathbb{R}$ and sublevel set $G^t$.

Given a filtered quiver $G$, the steps of the algorithm for obtaining the filtered $\Delta$-set $\tilde{\mathcal{F}}(G)$ are as follows. As with the algorithms for $\mathcal{F}(G)$ and $\bar{\mathcal{F}}(G)$, the procedure terminates when, either no new $(n+1)$-simplices are obtained or the maximal desired $(n+1)$-skeleton of $\mathcal{F}(G)$ for computation of persistent homology up to dimension $n$ has been acquired.

\begin{enumerate}[(1)]
    \item 
    Identify all vertices of sub-quivers $G_l$ and $G_{nl}$ and their filtration values. 
    \item 
    Remove all multiple edges that have a loop with a filtration value smaller then the edge filtration value itself, leaving at least one edge. When only one edge remains, it is assigned the smallest filtration value among the multiple edges between the same ordered pair of vertices. Then remove all loops from $G$.
    \item
    Set the $0$ simplices of $\tilde{\mathcal{F}}(G)$ and $\mathcal{F}(G)$ to be the vertices of $G$ and the $1$-simplices of $\mathcal{F}(G)$ to be the remaining edges of $G$.
    \item
    For each $n \geq 1$, obtain the $(n+1)$-simplices of $\mathcal{F}(G)$ in the same way as step (3) of the procedure in Section~\ref{sec:AlgDirectedFlagComplex} using Algorithm~\ref{alg:DirectedFlagComplex}.
    In addition, during the procedure record for each simplex $S$ the set of vertices $S_l$ of each simplex that has a loop and the index (if any) of the last vertex $S_{nl}$ without a loop from which $S$ is obtained as an extension. In the case when $S_l$ is empty, store in $S_{nl}$ the index of the simplex itself.
    \item 
    For each $n \geq 2$ and simplex triple $(S,S_{nl},S_{l})$ in dimension $n$, recover inductively all indices (if any) of the maximal non-loop face. When $S_l$ is empty this is already the case. Otherwise, this is achieved by replacing $S_{nl}$ with the corresponding value in the simplex in the image of the face map of index presently recorded in $S_{nl}$. 
    \item 
    For each $n \geq 1$, repeat the following inductive steps.
    \item
    For each $\{v_0,\dots,v_n\} \subseteq V_G$, compute any required additional cells between $n$-simplices by applying the following steps.
    \begin{enumerate}[(i)]
        \item 
        For each set $\{S^1,\dots,S^m\}$ of $n$-simplices such that $S^1,\dots,S^m$ have vertices $v_0,\dots,v_n$ and
        \[
            S^i_{l} = S^j_{l} \;\;\; \text{and} \;\;\; S^i_{nl} = S^j_{nl}
        \]
        for $i,j = 1,\dots,m$, consider the following labeled complete graph on $m$ vertices. The vertices of the graph are $S^1,\dots,S^m$. An edge between $S^i$ and $S^j$ for $i \neq j$ is assigned the maximum value among 
        \begin{enumerate}[(I)]
            \item 
            the filtration value of $S^i$, 
            \item
            the filtration value of $S^j$,
            \item
            when $n=1$ the minimum filtration value any loop at the vertices of $S^i$ and $S^j$, or 
            \item
            when $n \geq 2$ the minimum filtration value of a sequences of extra cell obtained in dimension-$(n-1)$ consecutively sharing a face such that the first cell in the sequence shares a face with $S^i$ and the last with $S^j$.
        \end{enumerate}
        We denote the possibly empty sequence from part (IV) above by $E^{n-1}_{i,j}$.
        \item 
        Find a minimally edge weighted spanning tree of the labeled graph decried above. This can be achieved by applying a standard procedure such as Prim’s Algorithm.
        \item 
        For each edge between a pair $S^i$ and $S^j$in the spanning tree, form an extra cell with boundary $S^i$,  $S^j$, and all elements of $E^{n-1}_{i,j}$, with filtration value identical to the edge weight.
    \end{enumerate}
    \item
    Add $n$-simplices and any additional cells obtained as part of the $(n-1)$-dimensional step to $\tilde{\mathcal{F}}(G)_n$.
\end{enumerate}

Once $\tilde{\mathcal{F}}(G)$ has been obtained by the procedure above, its filtered boundary matrix can be constructed
directly from the indices and filtration values stored in each simplex or cell of $\tilde{\mathcal{F}}(G)$.

\subsubsection{Computational complexity}

The time complexity of steps (1) and (2) of the computation of the partial directed flag complex $\tilde{\mathcal{F}}(G)$ detailed above are bounded above by $\mathcal{O}(|E_G|^2)$. The time complexity of step (4) in each dimension $n$ is identical to that of the directed flag complex $\mathcal{F}(G)$ (after the reduction of $G$ made in step (2)) provided in Appendix~\ref{sec:ComplexityDirectedFlag} equation~\eqref{eq:FlagComplexComplexity}. Steps (5) and (8) have time complexity $\mathcal{O}(|\tilde{\mathcal{F}}(G)_n|)$ in each dimension $n$. The part of the procedure that might typically have greater time complexity is step (7), which for each $n\geq 1$ can be described as follows.

We first note that the main loop in step (7) is over subsets of vertices $\{v_0,\dots,v_n\} \subseteq V_G$, for which there are $\binom{n+1}{|V_G|}$ choices. Let $E_n^{\{ v_0,\dots,v_n \}}$ denote the set of $n$-simplices in $\tilde{F}(G)_n$ on the vertex set $\{ v_0,\dots,v_n \}$ and let $E_{n}^C$ denote the set of extra cell obtained at the $(n-1)$-dimensional step. For part (i) of step (7), the worst case time complexity is achieved when $m = |E_n^{\{ v_0,\dots,v_n \}}|$. In this case for part (i) of step (7), there are $\binom{2}{|E_n^{\{v_0,\dots,v_n\}}|}$ edges in the compete graph, and to label each edge we must consider the each element of $E_{n}^C$. Finally, part (ii) of step (7) is applied to a complete graph and has know complexity of order the number of its edge, which in our case is $\binom{2}{|E_n^{\{ v_0,\dots,v_n \}}|}$ and less complex than the previous step. Therefore, the total time complexity is bounded above by
\[
    \mathcal{O}\left( \left|E_n^C\right| \sum_{\{v_0,\dots,v_n\} \subseteq V_G} \max\left(1,  \binom{2}{\left|E_n^{\{ v_0,\dots,v_n \}}\right|} \right)  \right).
\]

Similarly to the end of Section~\ref{sec:ComplexityReducedFlag}, we can apply our time complexity to quantify the improvement realised when obtaining $H_*^{\Delta,h}(G)$ directly as opposed to the homology of $\tilde{\mathcal{F}}(G)$. An algorithm for computing directly the generators of $C_*^{\Delta,h}(G)$ can be obtained from the procedure set out in this section by retaining all multiple edges and loops in step (2) and skipping step (7). This is because the extra relations represented by the extra cells are incorporated by the additional simplices formed by singular simplex homomorphisms using loop edges. We demonstrate now that in the $C_*^{\Delta,h}(G)$ case, the advantage of skipping step (7) is typically considerably outweighed by the increase in complexity of step (4).

More precisely, let $n \geq 0$ be an integer and suppose there is a singular $i$-simplex inclusion on vertices vertices $v_0,\dots,v_i \in V_G$, $m$ of which have loops for some $i=0,\dots,n$. Then there are $\binom{m+n-i-1}{n-i}$ $n$-simplices with vertices $v_0,\dots,v_i$ appearing as generators of $C_n^{\Delta,h}(G)$ that are not singular simplicial generators of $C_n^{\Delta,i}(G)$.
An expression similar to equation~\eqref{eq:MappingFlagComplexity} can now be derived for the total number of additional simplices required when computing the generators of $C_n^{\Delta,h}(G)$ as apposed to simplices of $\mathcal{F}(G)_n$.

Recall that $\mathcal{F}(G)_n^v$ denotes the set of $n$-dimensional simplices in $\mathcal{F}(G)$, whose greatest vertex with respect to the total order on vertices induced by the corresponding singular simplex inclusion $\colon \Delta^{n} \to G$, is $v\in V_G$.
The computational complexity for step (4) in computing $C_n^{\Delta,h}(G)$ given in equation~\eqref{eq:FlagComplexComplexity} depends heavily on a multiple of $|\mathcal{F}(G)_{n}^v|$.
Therefore, when the values of each $|\mathcal{F}(G)_n^v|$ is replaced by the corresponding multiple
detailed above, the direct algorithm for generators of $C_*^{\Delta,h}(G)$ becomes prohibitively slower at step (4) for quivers containing many loops. 

\end{document}